\documentclass[11pt]{article}

\usepackage{mathrsfs}
\usepackage{enumerate}
\usepackage[section]{algorithm}
\usepackage{algorithmic}
\usepackage{multirow}
\usepackage{xcolor}

\usepackage{graphicx}
\usepackage{amsmath,amsfonts,amsthm,amsbsy,amssymb}
\usepackage{fixmath}
\usepackage{amssymb,latexsym}

\usepackage{hyperref}
\usepackage{cleveref}

\def\be{\pmb{e}}

\def\bw{\pmb{w}}
\def\bx{\pmb{x}}

\def\bz{\pmb{z}}

\def\bone{\pmb{1}}

\def\bbO{\mathbb{O}}
\def\bbP{\mathbb{P}}
\def\bbR{\mathbb{R}}

\def\scrH{\mathscr{H}}

\def\scrL{\mathscr{L}}
\def\scrM{\mathscr{M}}

\def\scrR{\mathscr{R}}

\def\cR{{\cal R}}

\def\wtd{\widetilde}
\def\what{\widehat}

\usepackage{accents}
\newcommand\munderbar[1]{%
  \underaccent{\bar}{#1}}

\DeclareMathOperator{\diag}{diag}

\DeclareMathOperator*{\opt}{opt}
\DeclareMathOperator{\rank}{rank}
\DeclareMathOperator{\rmd}{d}

\DeclareMathOperator{\tr}{tr}

\DeclareMathOperator{\F}{F}
\DeclareMathOperator{\HH}{H}
\DeclareMathOperator{\T}{T}

\DeclareMathOperator{\KKT}{KKT}

\DeclareMathOperator{\OptSM}{OptSM}

\def\scrL{\mathscr{L}}
\def\scrR{\mathscr{R}}

\newtheorem{theorem}{Theorem}[section]
\newtheorem{lemma}{Lemma}[section]

\theoremstyle{definition}

\newtheorem{remark}{Remark}[section]

\allowdisplaybreaks

\numberwithin{equation}{section}
\numberwithin{figure}{section}
\numberwithin{table}{section}
\title{NEPv Approach for Optimization on Stiefel Manifold with the $(2,1)$-norm Regularization}

\author{
Ren-Cang Li%
\thanks{Department of Mathematics, University of Texas at Arlington, Arlington, TX 76019-0408, USA.
        Supported in part by NSF DMS-2407692.
        Email: {\tt rcli@uta.edu}.}
\and Li Wang%
\thanks{Department of Mathematics, University of Texas at Arlington, Arlington, TX 76019-0408, USA.
        Supported in part by NSF DMS-2407692.
        Email: {\tt li.wang.edu}.}
\and Lei-Hong Zhang%
\thanks{School of Mathematical Sciences, Soochow University, Suzhou 215006, Jiangsu, China.
 Supported in part by the National Natural Science Foundation of China (NSFC-12471356, NSFC-12371380), Jiangsu Shuangchuang Project (JSSCTD202209),  Academic Degree and Postgraduate Education Reform Project of Jiangsu Province, and China Association of Higher Education under grant 23SX0403.
        Email: {\tt longzlh@suda.edu.cn}.}
 \and
Zhaojun Bai%
\thanks{Department of Computer Science and Department of Mathematics, University of California, Davis, CA 95616 USA.
        Email: {\tt zbai@ucdavis.edu}.}
}

\date{March 8, 2025
}

\begin{document}

\maketitle

\centerline{\em In memory of Professor Nicholas John Higham}

\begin{abstract}
Row-sparse projection provides a useful tool in machine learning (ML) when it comes to, for example, feature selection, aiming
to choose most relevant features for various ML objectives. One way to seek a high quality row-sparse projection is to combine
an ML objective, such as the ones for PCA, LDA, and OCCA, with the matrix $(2,1)$-norm regularization  which is nonsmooth.
Such combinations result in challenging optimization problems on the Stiefel manifold that need to be solved
efficiently. In this paper, a unifying NEPv framework is established to
efficiently deal with optimization on the Stiefel manifold with the $(2,1)$-norm regularization.
The effect of the $(2,1)$-norm regularization is also investigated. The wide applicability
of the framework
is demonstrated through the combinations of common learning objectives in today's data science applications with
the $(2,1)$-norm regularization.
Numerical experiments are presented to illustrate the use of the NEPv approach and
to gain insights as to what a proper regularizing parameter should have in real-world applications.
%able to achieve for applications such as
%feature selections.

%show the numerical behavior of the NEPv approach as regularizing parameter varies and
%thereby offer suggestions as
%. , providing
%a needed understanding of the regularization

\bigskip
\noindent
{\bf Keywords:}
Feature selection,
CCA,
MAXBET,
LDA,
$(2,1)$-norm regularization,
Sparse projection,
NEPv

\smallskip
\noindent
{\bf Mathematics Subject Classification}  58C40; 65F30; 65H17; 65K05; 90C26; 90C32
\end{abstract}

\clearpage
\tableofcontents

\clearpage
\section{Introduction}\label{sec:intro}
The matrix $(2,1)$-norm is often used as an effective regularizer to induce a row-sparse projection
in single and multi-view learning for the purpose of feature selection %\cite{lims:2023,walx:2024,lawa:2021,wrma:2022}.
\cite{fahw:2024,zhnl:2018}.
Given $P\in\bbR^{n\times k}$ (usually $k\ll n$), its $(2,1)$-norm is defined as
\begin{equation}\label{eq:(2,1)-norm:defn}
\|P\|_{2,1}=\sum_{i=1}^n\|P_{(i,:)}\|_2,
\end{equation}
where $P_{(i,:)}$ denotes the $i$th row of $P$ and $\|P_{(i,:)}\|_2$ is its vector 2-norm. In this paper, we are concerned with
the NEPv framework for
optimization on Stiefel manifold with the $(2,1)$-norm regularization. Specifically, we are interested
in solving
\begin{equation}\label{eq:OptSTM+L21}
\OptSM_{2,1}:\qquad\max_{P\in\bbO^{n\times k}}\Big\{f(P):=g(P)-\alpha\|P\|_{2,1}\Big\},
\end{equation}
where $g(\cdot)$ is  defined on some neighborhood of the Stiefel manifold
\begin{equation}\label{eq:STMkn}
\bbO^{n\times k}:=\{P\in\bbR^{n\times k}\,:\,P^{\T}P=I_k\}
\end{equation}
and is smooth, and $I_k$ is the $k\times k$ identity matrix.
It is noted that
$\|P\|_{2,1}$ is not differentiable as soon as one or more rows $P_{(i,:)}$ become zero rows, making
$\OptSM_{2,1}$~\eqref{eq:OptSTM+L21} a non-smooth optimization problem.

The ubiquity  of problem $\OptSM_{2,1}$~\eqref{eq:OptSTM+L21} lies in that
$g(P)$ can be taken to be any of today's common objectives
in data science such as
\begin{equation}\label{eq:obj-common}
\frac {\tr(P^{\T}A_2P)}{\tr(P^{\T}A_1P)}, \quad
\left[\frac {\tr(P^{\T}D)}{\sqrt{\tr(P^{\T}AP)}}\right]^2, \quad
\sum_{i=1}^N\|P^{\T}A_iP\|_{\F}^2, \quad
\left[\frac {\tr(P^{\T}A_2P)+\tr(P^{\T}D)}{[\tr(P^{\T}A_1P)]^{\theta}}\right]^2.
\end{equation}
These four objective functions are
for the orthogonal linear discriminant analysis (OLDA)
        \cite{cazb:2018,fish:1936,ngbs:2010,zhln:2010,zhln:2013},
the orthogonal canonical correlation analysis (OCCA)
        \cite{cugh:2015,zhwb:2022},
the uniform multidimensional scaling (UMDS) \cite{zhzl:2017},
and the $\theta$-trace ratio ($\Theta$TR) for multiview learning \cite{wazl:2023}, respectively.
It is noted that,  in \eqref{eq:obj-common}, we square the objectives of the original OCCA \cite{zhwb:2022} and $\Theta$TR \cite{wazl:2023}
for a reason that will become clear later in \cref{sec:CEgs}.
Given such ubiquity, a comprehensive investigation in terms of theory and computations for $\OptSM_{2,1}$~\eqref{eq:OptSTM+L21} is of high importance.
Our goal in this paper are two-fold:
\begin{enumerate}[(1)]
  \item Establish an NEPv  framework to solve $\OptSM_{2,1}$~\eqref{eq:OptSTM+L21} where $g(P)$ is
        some convex composition of atomic functions for NEPv recently introduced in \cite{li:2024}.
        Special attention will be given to those atomic functions: $\tr((P^{\T}AP)^m)$ and $\tr((P^{\T}D)^m)$
        that appear ubiquitously in subspace learning.
  \item Explain how the framework works for the common machine learning objectives, such as those in \eqref{eq:obj-common},
      from
        MAXBET \cite{wazl:2022a,zhys:2020} and the unbalanced Procrustes problem \cite{chtr:2001,edas:1999,elpa:1999,godi:2004,huca:1962,zhdu:2006},
        LDA \cite{cazb:2018,zhln:2010,zhln:2013}, OCCA \cite{cugh:2015,zhwb:2022}, and
        $\Theta$-trace ratios \cite{wazl:2023}, among others.
        %Machine learning examples will be provided to demonstrate the effectiveness of the NEPv framework.
\end{enumerate}

The rest of this paper is organized as follows. In \cref{sec:NEPv-rev}, we provide a brief review of the NEPv
framework in \cite{li:2024} to set up the stage for \cref{sec:NEPv+L21} in which we lay out our unifying NEPv framework
for optimization on Stiefel manifold with the $(2,1)$-norm regularization, including a theory and an SCF iteration.
In \cref{sec:LOCG}, we explain an LOCG (locally optimal conjugate gradient) type acceleration of the SCF iteration.
The effect of the $(2,1)$-norm regularization as  regularizing parameter $\alpha$ moves towards infinite is investigated in
\cref{sec:EffL21} and the results can and will be used to explain numerical behaviors of the NEPv approach in \cref{sec:egs}.
In \cref{sec:CEgs}, we explain how to apply the newly established NEPv framework to a few  objectives commonly used
in data science today and regularized by the matrix $(2,1)$-norm for generating row-sparse projections.
Numerical experiments are reported in \cref{sec:egs} to
%onfirm our claims in the earlier sections as well as
demonstrate the efficiency of the approach and its effectiveness in producing row sparse solutions.
Finally, we draw our conclusions in \cref{sec:concl}.

{\bf Notation.}
%\begin{itemize}
%  \item
  $\bbR^{m\times n}$  is the set of $m\times n$ real matrices,  $\bbR^n=\bbR^{n\times 1}$, and $\bbR=\bbR^1$;
        $\bbR_+:=\{x\in\bbR\,:\,x\ge 0\}$ and $\bbR_{++}:=\{x\in\bbR_+\,:\,x> 0\}$.
        %$\bbC^{m\times n}$,  $\bbC^n$, and $\bbC$ except for the complex numbers;
%  \item $\STM{k}{n}$  denotes the Stiefel manifold
%        $$
%        \bbO^{n\times k}=\{P\in\bbR^{n\times k}\,:\,P^{\T}P=I_k\}\subset\bbC^{n\times k};
%        $$
%  \item
  $I_n\in\bbR^{n\times n}$ is the identity matrix or simply $I$ if its size is clear from the context, and $\be_j$ is the $j$th column of $I$ of an apt size. By some neighborhood of the Stiefel manifold $\bbO^{n\times k}$, we mean
\begin{equation}\label{eq:O-delta}
\bbO^{n\times k}_{\delta}:=\{P\in\bbR^{n\times k}\,:\,\|P^{\T}P-I_k\|<\delta\},
\end{equation}
where $0<\delta$ is a constant and $\|\cdot\|$ is some matrix norm.
 % \item

  For $B\in\bbR^{m\times n}$, $B^{\T}$ stands for its transpose (as a matrix/vector),
  %$\cR(B)$ is the column subspace of a matrix $B$, spanned by its columns, whose dimension is
      $\rank(B)$ its rank and
      $B_{(i,:)}$ is its $i$th row. For $A\in\bbR^{n\times n}$, $A\succ 0\, (\succeq 0)$ means that it is symmetric and positive definite (semi-definite), and
        accordingly
        $A\prec 0\, (\preceq 0)$ if $-A\succ 0\, (\succeq 0)$.

%  \item
  The {\em thin\/} SVD of $B\in\bbR^{m\times n}$ ($m\ge n$) is $B=U\Sigma V^{\T}$ where
        $$
        \Sigma=\diag(\sigma_1(B),\sigma_2(B),\ldots,\sigma_n(B))\in\bbR^{n\times n},
        \,\,
        U\in\bbO^{m\times n},\,\,
        V\in\bbO^{n\times n}.
        $$
        %and $s=\min\{m,n\}$.
        The singular values $\sigma_j(B)$  are always arranged decreasingly as
        $
        \sigma_1(B)\ge\cdots\ge\sigma_n(B)\ge 0.
        $
        %and $\sigma_{\min}(B)=\sigma_n(B)$;
        Accordingly, $\|B\|_2$, $\|B\|_{\F}$, and $\|B\|_{\tr}$ are the spectral, Frobenius, and
        trace         norms of $B$:
        $$
        \|B\|_2=\sigma_1(B),\,\,
        \|B\|_{\F}=\Big(\sum_{i=1}^n[\sigma_i(B)]^2\Big)^{1/2},\,\,
        \|B\|_{\tr}=\sum_{i=1}^n\sigma_i(B),
        $$
        respectively. The trace norm is  also known as the {\em nuclear norm}. Unless otherwise explicitly stated, SVD always
        refers to {\em thin\/} SVD in this paper.
        With the SVD $B=U\Sigma V^{\T}$, a polar decomposition of $B$ is given by $B=QH$ \cite{high:2008} where $Q=UV^{\T}$
        and $H=V^{\T}\Sigma V$,  and
        $Q$ is called  an {\em orthogonal polar factor\/} and it is unique if $\rank(B)=n$ \cite{li:1993b,li:2014HLA}.

\section{The NEPv approach, a brief review}\label{sec:NEPv-rev}
In \cite{li:2024}, a unifying framework for the NEPv approach is established to solve
\begin{equation}\label{eq:OptSTM}
\OptSM:\max_{P\in\bbO^{n\times k}} g(P),
\end{equation}
i.e., $\OptSM_{2,1}$~\eqref{eq:OptSTM+L21} without the
$(2,1)$-norm regularizing term $-\alpha\|P\|_{2,1}$,
everything else being equal. It is a smooth optimization problem since $g(\cdot)$ is assumed to be smooth
on some  neighborhood $\bbO^{n\times k}_{\delta}$ of the Stiefel manifold.

\subsection{The NEPv Ansatz}
The NEPv framework established in \cite{li:2024} is built upon the following Ansatz  where
a subset $\bbP$ of $\bbO^{n\times k}$ is involved and has to be determined dependently of $g(P)$, such as the one in \Cref{tbl:AF-NEPv} later when $g(P)$ involves matrix trace function $\tr((P^{\T}D)^m)$.

\smallskip\noindent
{\bf The NEPv Ansatz.}
{\em
For function $g$ defined in some neighborhood $\bbO_{\delta}^{n\times k}$ of the Stiefel manifold $\bbO^{n\times k}$,
%maximization problem \eqref{eq:main-opt},
there is a
%Consider maximization problem \eqref{eq:main-opt}.
%There is
symmetric matrix-valued function $H(P)\in\bbR^{n\times n}$ such that for
$\what P\in\bbO^{n\times k},\,P\in\bbP\subseteq\bbO^{n\times k}$,
if
$$
\tr(\what P^{\T}H(P)\what P)\ge\tr(P^{\T}H(P)P)+\eta
\quad\mbox{for some $\eta\in\bbR$},
$$
then
there exists $Q\in\bbO^{k\times k}$ such that $\wtd P=\what PQ\in\bbP$ and
$g(\wtd P)\ge g(P)+\omega\eta$,
where $\omega$ is some positive constant, independent of $P$ and $\what P$.
}

When this Ansatz is satisfied, the SCF iteration:
      \begin{equation}\label{eq:SCF-form:NEPv:intro}
      \framebox{
      \parbox{12.0cm}{given $P_0$, iteratively
      compute partial eigendecomposition $H(P_{i-1})\what P_i=\what P_i\Omega_i$
      associated with the $k$ largest  eigenvalues of $H(P_{i-1})$
                  for $\what P_i\in\bbO^{n\times k}$,  and postprocess $\what P_i$ to $P_i$.
      }}
      \end{equation}
can be called to solve $\OptSM$~\eqref{eq:OptSTM} and it is guaranteed that $\{g(P_i)\}_{i=0}^{\infty}$
is monotonic increasing, and any accumulation point of $\{P_i\}_{i=0}^{\infty}$ is
a solution of the following nonlinear eigenvalue problem with eigenvector dependency (NEPv)
\begin{equation}\label{eq:OptSTM:NEPv}
H(P)\,P=P\Omega, \quad P\in\bbO^{n\times k},\quad \Omega=\Omega^{\T}\in\bbR^{k\times k}.
\end{equation}
%while objective value $g(P_i)$ monotonically increases.
It turns out that for the
objectives,  previously appeared in the literature, such as those in \eqref{eq:obj-common},
{\bf the NEPv Ansatz} is indeed fulfilled, albeit unawarely when they were first studied before the emergence of
the Ansatz in \cite{li:2024}. In fact, each of these objectives was investigated individually, but, now with the Ansatz,
they can be collectively investigated together under one umbrella \cite{li:2024}.

We notice that, in SCF \eqref{eq:SCF-form:NEPv:intro}, there is an unsettling postprocessing step from $\what P_i$ to $P_i$,
which precisely is due to the existence of the $Q$-matrix in the Ansatz to go from $\what P$ to $\wtd P$. More comments on
this issue will come after \Cref{alg:NEPvSCF4+L21} and in \cref{sec:CEgs}.

%We notice that, in SCF \eqref{eq:SCF-form:NEPv:intro}, there is an unsettling postprocessing step from $\what P_i$ to $P_i$, which precisely is about
%the existence of a $Q$-matrix in the Ansatz to go from $\what P$ to $\wtd P$. For all the concrete objectives
%later in \cref{sec:CEgs}, either $P_i=\what P_i$ or $P_i=\what P_iQ_i$ where $Q_i$ is an orthogonal polar factor
%of $\what P_i^{\T}D$ when $\tr((P^{\T}D)^m)$ is involved in the objective of interest.

The first order optimality condition, also known as the KKT condition, of $\OptSM$~\eqref{eq:OptSTM}
is given by \cite{abms:2008,li:2024}
\begin{equation}\label{eq:KKT}
\frac{\partial g(P)}{\partial P}=P\Lambda
\quad \mbox{with} \quad
\Lambda^{\T}=\Lambda\in\bbR^{k\times k},\quad P\in\bbO^{n\times k}.
\end{equation}
In \cite[Theorem~6.1]{li:2024}, sufficient and necessary conditions are also given, as to the equivalency
between NEPv~\eqref{eq:OptSTM:NEPv} and the KKT condition~\eqref{eq:KKT}. It is noted that \Cref{thm:H(P)-eligibility}
below is independent of {\bf the NEPv Ansatz}.

\begin{theorem}[{\cite[Theorem~6.1]{li:2024}}]\label{thm:H(P)-eligibility}
Let $H(P)\in\bbR^{n\times n}$ be a symmetric matrix-valued function for $P\in\bbO^{n\times k}$. Suppose that
\begin{equation}\label{eq:cond:KKT=NEPv}
H(P)P-\frac{\partial g(P)}{\partial P}=P\,\scrM(P)\quad\mbox{for $P\in\bbO^{n\times k}$},
\end{equation}
where $\scrM(P)\in\bbR^{k\times k}$ is some matrix-valued function.
$P\in\bbO^{n\times k}$ is a solution to the KKT condition \eqref{eq:KKT} if and only if
it is a solution to NEPv \eqref{eq:OptSTM:NEPv}
and $\scrM(P)$ is  symmetric.
\end{theorem}

\subsection{Atomic function}
In \cite{li:2024}, a theory of atomic function for NEPv
is invented to generate a large collection of objectives that satisfy the Ansatz. Specifically,
we call a function $g$ defined in some neighborhood $\bbO_{\delta}^{n\times k}$ {\em an atomic function
for NEPv\/} if there exists a symmetric matrix-valued function $H(P)\in\bbR^{n\times n}$
%
%Suppose that for  function $f$ defined on some neighborhood $\bbO^{n\times k}_{\delta}$, we have symmetric matrix-valued function $H(P)$ that satisfies condition~\eqref{eq:cond:KKT=NEPv}.
%\marginpar{\tiny updated}
such that
\begin{subequations}\label{eq:cond4AF-NEPv}
\begin{align}
\tr(P^{\T}H(P)P)&=\munderbar\gamma\, g(P)\quad\mbox{for $P\in\bbP\subseteq\bbO^{n\times k}$},
              \label{eq:cond4AF-NEPv-a} \\
\intertext{and given $P\in\bbP$ and $\what P\in\bbO^{n\times k}$, there exists $Q\in\bbO^{k\times k}$ such
that $\wtd P=\what PQ\in\bbP$ and}
\tr(\what P^{\T}H(P)\what P)&\le\munderbar\alpha \,g(\wtd P)+\munderbar\beta \,g(P),
           %\quad\mbox{for $P\in\bbP\subseteq\bbO^{n\times k}$},
      \label{eq:cond4AF-NEPv-b}
\end{align}
\end{subequations}
where $\munderbar\alpha>0,\,\munderbar\beta\ge 0$,  and $\munderbar\gamma=\munderbar\alpha+\munderbar\beta$ are constants.
In particular, it is shown that
both
%\footnote {More generally $\tr((P^{\T}D)^m)$ and $\tr((P^{\T}AP)^m)$, with appropriate $H(P)$, are atomic functions for NEPv.
%       Since we won't use them in our concrete examples, we omit such generality.}
$\tr((P^{\T}D)^m)$ and $\tr((P^{\T}AP)^m)$
%\begin{align}
%\tr(P^{\T}D) &\quad\mbox{with}\quad H(P)=DP^{\T}+PD^{\T}\,\,\mbox{and}\,\,\bbP=\{P\in\bbO^{n\times k}\,:\,P^{\T}D\succeq 0\},
%         \label{eq:trPtD} \\
%\tr(P^{\T}AP) &\quad\mbox{with}\quad H(P)=2A\,\,\mbox{and}\,\,\bbP=\bbO^{n\times k}, \label{eq:trPtAP}
%\end{align}
are atomic functions for NEPv, with details shown in \Cref{tbl:AF-NEPv}
\cite[Theorems~7.4 and~7.5]{li:2024} where $m\ge 1$ is an integer.
We caution the reader that \eqref{eq:cond4AF-NEPv-b} may be an equality for some
atomic functions and that is significant.
It is shown that any atomic function for NEPv satisfies the Ansatz with $\omega=1/\munderbar\alpha$
\cite{li:2024}.

\iffalse
\begin{table}[t]
\renewcommand{\arraystretch}{1.3}
\caption{\small Atomic functions $\tr(P^{\T}D)$ and $\tr(P^{\T}AP)$}\label{tbl:AF-NEPv}
\centerline{\small
\begin{tabular}{|c|c|c|c|c|}
  \hline
             & $H(P)$ & $\bbP$ & $(\munderbar\alpha,\munderbar\beta)$ & \eqref{eq:cond4AF-NEPv-b} \\ \hline
$\tr(P^{\T}D)$  & $DP^{\T}+PD^{\T}$ & $\bbO_{D+}^{n\times k}:=\{P\in\bbO^{n\times k}\,:\,P^{\T}D\succeq 0\}$ & $(2,0)$ & inequality \\ \hline
$\tr(P^{\T}AP)$ & $2A$ & $\bbO^{n\times k}$ or $\bbO_{D+}^{n\times k}$ & $(2,0)$ & equality \\ \hline
\end{tabular}
}
\end{table}
\fi

\begin{table}[t]
\renewcommand{\arraystretch}{1.3}
\caption{\small Atomic functions $\tr((P^{\T}D)^m)$ and $\tr((P^{\T}AP)^m)$}\label{tbl:AF-NEPv}
\centerline{\small
\begin{tabular}{|c|c|c|c|c|}
  \hline
             & $m$ & $H(P)$ & $\bbP$ &  \eqref{eq:cond4AF-NEPv-b} \\ \hline
$\tr((P^{\T}D)^m)$ & $\ge 1$  & { $m\,\big[D(P^{\T}D)^{m-1}P^{\T}+P(D^{\T}P)^{m-1}D^{\T}\big]$}
                & $\bbO_{D+}^{n\times k}$ &  inequality \\ \hline
\multirow{2}{*}{$\tr((P^{\T}AP)^m)$} & $1$ & $2A$ & $\bbO^{n\times k}$  &  equality \\ \cline{2-5}
 & $>1$ & $2m\,A(PP^{\T}A)^{m-1}$ ($A\succeq 0$ required) & $\bbO^{n\times k}$  &  inequality \\ \hline
\multicolumn{5}{l}{\small * $\bbO_{D+}^{n\times k}:=\{P\in\bbO^{n\times k}\,:\,P^{\T}D\succeq 0\}$, and
                            $(\munderbar\alpha,\munderbar\beta)=(2,2(m-1))$ in \eqref{eq:cond4AF-NEPv} for all.}
\end{tabular}
}
\end{table}

\subsection{Convex composition}
Consider now a convex composition of atomic functions for NEPv:
\begin{equation}\label{eq:g-cvx-comp}
g(P)=\psi\circ T_0(P)
\quad\mbox{with}\quad
T_0(P)=\begin{bmatrix}
                                        g_1(P) \\
                                        g_2(P) \\
                                        \vdots \\
                                        g_N(P) \\
                                      \end{bmatrix},
\end{equation}
where $\psi\,:\,\mathfrak{D}\subseteq\bbR^N\to\bbR$ is
convex and differentiable, $T_0\,:\, P\in\bbO^{n\times k} \to T_0(P)\in\mathfrak{D}$, and
each $g_i(P)$ is an atomic function,
with some symmetric matrix-valued function $H_i(P)\in\bbR^{n\times n}$, such that:
for $1\le i\le N$
\begin{subequations}\label{eq:cond4AF-NEPv:cvx}
\begin{align}
\tr(P^{\T}H_i(P)P)&=\munderbar\gamma_i\, g_i(P)\quad\mbox{for $P\in\bbP\subseteq\bbO^{n\times k}$},
              \label{eq:cond4AF-NEPv:cvx-a} \\
\intertext{and given $\what P\in\bbO^{n\times k}$ and $P\in\bbP$, there exists $Q\in\bbO^{k\times k}$ such
that $\wtd P=\what PQ\in\bbP$ and} %\subseteq\bbO^{n\times k}$ and}
\tr(\what P^{\T}H_i(P)\what P)
   &\le\munderbar\alpha g_i(\wtd P)+\munderbar\beta_i g_i(P),
   % \quad\mbox{for $P,\,\wtd P\in\bbP\subseteq\bbO^{n\times k}$},
      \label{eq:cond4AF-NEPv:cvx-b}
\end{align}
\end{subequations}
%and $Q$ can be taken to be $I_k$ if $\bbP=\bbO^{n\times k}$,
where $\munderbar\alpha>0,\,\munderbar\beta_i\ge 0$,  and
$\munderbar\gamma_i=\munderbar\alpha+\munderbar\beta_i$ are  constants. It is noted that
there are three consistency assumptions among those atomic functions $g_i(\cdot)$:
1) the same $\bbP$, 2) the same $Q$ for given $P$ and $\what P$, and 3) the same
$\munderbar\alpha>0$.

Denote the partial derivatives of $\psi$
with respect to $\bx=[x_1,x_2,\ldots,x_N]^{\T}\in\mathfrak{D}\subseteq\bbR^N$ by
\begin{equation}\label{eq:psi-i}
\psi_i(\bx)=\frac {\partial \psi(\bx)}{\partial x_i}\quad\mbox{for $1\le i\le N$},
\end{equation}
and consider the symmetric matrix-valued function
\begin{equation}\label{eq:H(P)-comp-form}
H(P)=\sum_{i=1}^N\psi_i(T_0(P))\,H_i(P)
\end{equation}
and the associated NEPv \eqref{eq:OptSTM:NEPv} to go with $g(\cdot)$ in \eqref{eq:g-cvx-comp}.
%It is shown that $g(P)$ of \eqref{eq:g-cvx-comp}, with $H(P)$ as in \eqref{eq:H(P)-comp-form}, satisfies
%{\bf the NEPv Ansatz} with $\omega=1/\munderbar\alpha$
%if $\psi_i(\bx)\ge 0$ for those $i$ for which \eqref{eq:cond4AF-NEPv:cvx-b} does not become an equality.

\begin{theorem}\label{thm:main-NEPv-cvx}
Consider $g$ as in \eqref{eq:g-cvx-comp} and $\psi$ is convex and differentiable with partial derivatives denoted by
$\psi_i$ as in \eqref{eq:psi-i}. Let $H(P)$ be given by \eqref{eq:H(P)-comp-form} with $H_i(P)$ for $1\le i\le N$
satisfying \eqref{eq:cond4AF-NEPv:cvx}.
\begin{enumerate}[{\rm (a)}]
  \item {\rm \cite[Theorem 8.1]{li:2024}} If $\psi_i(\bx)\ge 0$ for those $i$ for which \eqref{eq:cond4AF-NEPv:cvx-b} does not become an equality, then
        {\bf the NEPv Ansatz}  holds   with $\omega=1/\munderbar\alpha$.
         %for $g=\psi\circ T_0$ with $H(\cdot)$ in \eqref{eq:H(P)-comp-form}.
  \item {\rm \cite[Theorem 6.3]{li:2024}}
        Let $P_*\in\bbP$ be an maximizer of $\OptSM$~\eqref{eq:OptSTM}. Then $P_*$ satisfies
        NEPv \eqref{eq:OptSTM:NEPv}, and the eigenvalues of $\Omega_*:=P_*^{\T}H(P_*)P_*$ are
        the $k$ largest eigenvalues of $H(P_*)$.
\end{enumerate}
\end{theorem}

\section{The NEPv framework for $(2,1)$-norm regularization}\label{sec:NEPv+L21}
We return to $\OptSM_{2,1}$~\eqref{eq:OptSTM+L21}  with $g(P)$ as given in \eqref{eq:g-cvx-comp}
satisfying \eqref{eq:cond4AF-NEPv:cvx} with $\munderbar\alpha=2$.
Rewrite
$$
\|P\|_{2,1}=\sum_{i=1}^n\|\be_i^{\T}P\|_2
    =\sum_{i=1}^n\sqrt{\be_i^{\T}PP^{\T}\be_i}
    =\sum_{i=1}^n\sqrt{\tr(P^{\T}\be_i\be_i^{\T}P)},
$$
and hence
\begin{subequations}\label{eq:OptSTM+L21:-obj}
\begin{align}
f(P)=g(P)-\alpha\|P\|_{2,1}&=g(P)-\alpha\sum_{i=1}^n\sqrt{\tr(P^{\T}\be_i\be_i^{\T}P)} \label{eq:OptSTM+L21:-obj-1} \\
    &=\phi\circ T(P), \label{eq:OptSTM+L21:-obj-2}
\end{align}
where, for $\bx=[x_i]\in\bbR^{N+n}$,
\begin{equation}\label{eq:OptSTM+L21:-obj-3}
\phi(\bx)=\psi(\bx_{(1:N)})-\alpha\sum_{i=1}^n\sqrt{x_{N+i}}, \quad
T(P)=\begin{bmatrix}
       T_0(P) \\
       \tr(P^{\T}\be_1\be_1^{\T}P) \\
       \vdots \\
       \tr(P^{\T}\be_n\be_n^{\T}P)
     \end{bmatrix}
     \equiv\begin{bmatrix}
       g_1(P) \\
                                        g_2(P) \\
                                        \vdots \\
                                        g_N(P) \\
       \tr(P^{\T}\be_1\be_1^{\T}P) \\
       \vdots \\
       \tr(P^{\T}\be_n\be_n^{\T}P)
     \end{bmatrix},
\end{equation}
\end{subequations}
and $\bx_{(1:N)}\in\bbR^N$ is the subvector of $\bx$, consisting of its first $N$ components.

Since $\psi(\bx_{(1:N)})$ is convex in $\bx_{(1:N)}\in\mathfrak{D}\subseteq\bbR^N$ by assumption and
$-\alpha\sqrt{x}$ is convex in $\bbR_+$ because of
$$
\frac {\rmd^2 (-\alpha\sqrt{x})}
      {\rmd x^2}=\frac 14\alpha x^{-3/2}\ge 0,
$$
we find that $\phi(\bx)$ is convex for $\bx\in\mathfrak{D}\times\bbR_+^n$. Already, $g_i(P)$ for $1\le i\le N$
are assumed atomic functions that satisfy \eqref{eq:cond4AF-NEPv:cvx} with $\munderbar\alpha=2$ and we point out that
each $\tr(P^{\T}\be_i\be_i^{\T}P)$ is also atomic function with $H_{N+i}(P)=2\be_i\be_i^{\T}$ and with $\munderbar\alpha=2$,
according to \Cref{tbl:AF-NEPv}. Hence the NEPv approach reviewed in \cref{sec:NEPv-rev}
may be applicable, with symmetric matrix-valued function
\begin{equation}\label{eq:H(P)-comp-form+L21}
H(P)=\sum_{i=1}^N\phi_i(T(P))\,H_i(P)-\alpha\sum_{i=1}^n\frac {\be_i\be_i^{\T}}{\sqrt{\tr(P^{\T}\be_i\be_i^{\T}P)}}
   \in\bbR^{n\times n}.
\end{equation}
However, there is one obstacle when it comes to the application, i.e.,
each $\sqrt{\tr(P^{\T}\be_i\be_i^{\T}P)}$ is not
differentiable in $P$ when $\tr(P^{\T}\be_i\be_i^{\T}P)=0$. This is also reflected
in the expression of $H(P)$ in \eqref{eq:H(P)-comp-form+L21}
where $\sqrt{\tr(P^{\T}\be_i\be_i^{\T}P)}$ for $1\le i\le n$ appear in the denominators.
A singularity occurs whenever one or more of $\tr(P^{\T}\be_i\be_i^{\T}P)$ are $0$.
Even if such a singularity does not occur, near singularity is bound to happen,
after all the precise purpose of introducing
the (2,1)-norm regularization is to promote row sparsity in $P$ in the first place, i.e., making some of the rows of $P$
extremely small, i.e., some $\|\be_i^{\T}P\|_2\ll 1$, if not exactly $0$.
When $\|\be_i^{\T}P\|_2\ll 1$, it means that, from the feature selection point of view,
the $i$th feature is insignificant and may be deselected. In practice,
this can be realized by pre-selecting a small tolerance
$\varepsilon_0$ and then regarding any rows such that
$\|\be_i^{\T}P\|_2\le\varepsilon_0$ insignificant. How small should $\varepsilon_0$ be? Since
$$
\sum_{i=1}^n\|\be_i^{\T}P\|_2^2=k
\quad\Rightarrow\quad
\frac kn\le\max_{1\le i\le n}\|\be_i^{\T}P\|_2^2\le 1.
$$
For data science applications, $\varepsilon_0=10^{-3}\sqrt{k/n}$ should be good enough.
%Usually $\varepsilon_0$ about $10^{-3}\sim 10^{-5}$ is sufficient for any data science application.

Near singularity creates numerical difficulty.
In view of these discussions, instead of $\OptSM_{2,1}$~\eqref{eq:OptSTM+L21} with $f(P)$ as in \eqref{eq:OptSTM+L21:-obj}, we may solve
a perturbed problem as follows:
\begin{equation}\label{eq:OptSTM+L21:eps}
\OptSM_{2,1}^{\varepsilon}:\qquad
\max_{P\in\bbO^{n\times k}} \left\{ f_{\varepsilon_0}(P):=g(P)-\alpha\sum_{i=1}^n\sqrt{\tr(P^{\T}\be_i\be_i^{\T}P)+\varepsilon_0^2}\right\},
\end{equation}
again with $g(P)$ as in \eqref{eq:g-cvx-comp} satisfying \eqref{eq:cond4AF-NEPv:cvx}.
With the same $T(P)$ as in \eqref{eq:OptSTM+L21:-obj-3}, we find that
\begin{equation}\label{eq:feps->form}
f_{\varepsilon_0}(P)=\phi_{\varepsilon_0}\circ T(P)
\quad\mbox{with}\quad
\phi_{\varepsilon_0}(\bx)=\psi(\bx_{(1:N)})-\alpha\sum_{i=1}^n\sqrt{x_{N+i}+\varepsilon_0^2}.
\end{equation}
%Function $\phi_{\varepsilon_0}(\bx)$ is still convex for for $\bx\in\mathfrak{D}\times\bbR_+^n$.
%Accordingly, we have
%\begin{subequations}\label{eq:OptSTM+L21:KKT'-eps}
%\begin{equation}\label{eq:OptSTM+L21:scrH-eps}
%\scrH_{\varepsilon_0}(P):=\frac {\partial f_{\varepsilon_0}(P)}{\partial P}
%    =-2AP+2D-\alpha\sum_{i=1}^n\frac {\be_i\be_i^{\T}P}{\sqrt{\tr(P^{\T}\be_i\be_i^{\T}P)+\varepsilon_0^2}},
%\end{equation}
%and the first order optimality condition of \eqref{eq:OptSTM+L21:eps}
%\cite[section~2]{li:2024}
%\begin{equation}\label{eq:OptSTM+L21:KKT-eps}
%\scrH_{\varepsilon_0}(P)=P\Lambda, \quad P\in\bbO^{n\times k},\quad \Lambda=\Lambda^{\T}\in\bbR^{k\times k}.
%\end{equation}
%\end{subequations}
Correspondingly, we use the following symmetric matrix-valued function
\begin{subequations}\label{eq:OptSTM+L21:NEPv'-eps}
\begin{equation}\label{eq:OptSTM+L21:H(P)-eps}
H_{\varepsilon_0}(P)=\sum_{i=1}^N\psi_i(T_0(P))\,H_i(P)
   -\alpha\sum_{i=1}^n\frac {\be_i\be_i^{\T}}{\sqrt{\tr(P^{\T}\be_i\be_i^{\T}P)+\varepsilon_0^2}}
   \in\bbR^{n\times n},
\end{equation}
whose associated NEPv is
\begin{equation}\label{eq:OptSTM+L21:NEPv-eps}
H_{\varepsilon_0}(P)\,P=P\Omega, \quad P\in\bbO^{n\times k},\quad \Omega=\Omega^{\T}\in\bbR^{k\times k}.
\end{equation}
\end{subequations}
Also note that the KKT condition of $\OptSM_{2,1}^{\varepsilon}$~\eqref{eq:OptSTM+L21:eps} is
\begin{subequations}\label{eq:KKT+L21:work-eps}
\begin{equation}\label{eq:KKT+L21:work-eps-1}
\scrH_{\varepsilon_0}(P)
   =P\Lambda
\quad \mbox{with} \quad
\Lambda^{\T}=\Lambda\in\bbR^{k\times k},\quad P\in\bbO^{n\times k},
\end{equation}
where
\begin{equation}\label{eq:partD+L21:work-eps}
\scrH_{\varepsilon_0}(P):=\frac{\partial f_{\varepsilon_0}(P)}{\partial P}
   =\sum_{i=1}^N\psi_i(T_0(P))\,\frac{\partial g_i(P)}{\partial P}
     -\alpha\sum_{i=1}^n\frac {\be_i\be_i^{\T}P}{\sqrt{\tr(P^{\T}\be_i\be_i^{\T}P)+\varepsilon_0^2}}.
\end{equation}
\end{subequations}

\Cref{thm:main-NEPv-cvx+L21:eps} below forms the foundation of the eventual SCF iteration in
\Cref{alg:NEPvSCF4+L21} to solve NEPv~\eqref{eq:OptSTM+L21:NEPv-eps} and in return the optimization problem~\eqref{eq:OptSTM+L21:eps}.
%Again we find that for $P\in\bbO^{n\times k}$
%\begin{equation}\label{eq:H(p)-scrH(P)-eps}
%H_{\varepsilon_0}(P)P-\scrH_{\varepsilon_0}(P)=P\big(2D^{\T}P\big).
%\end{equation}
%We will also have the corresponding versions of \Cref{thm:H(P)-eligibility,thm:main-NEPv-cvx}, without the need to
%assume that $P$ has no zero rows. They are stated as follows.

\begin{algorithm}[t]
\caption{The NEPv approach for solving $\OptSM_{2,1}^{\varepsilon}$~\eqref{eq:OptSTM+L21:eps}} \label{alg:NEPvSCF4+L21}
\begin{algorithmic}[1]
\REQUIRE $g(P)$ as in \eqref{eq:g-cvx-comp} satisfying \eqref{eq:cond4AF-NEPv:cvx} with $\munderbar\alpha=2$,
         regularization parameter $\alpha>0$,
         $\varepsilon_0>0$,
         and initial approximation $P^{(0)}\in\bbP$;
\ENSURE  an approximate maximizer of $\OptSM_{2,1}^{\varepsilon}$~\eqref{eq:OptSTM+L21:eps}.
\FOR{$j=0,1,\ldots$ until convergence}
    \STATE compute $H^{(j)}=H_{\varepsilon_0}(P^{(j)})\in\bbR^{n\times n}$ where $H_{\varepsilon_0}(P)$
           is as in \eqref{eq:OptSTM+L21:H(P)-eps};
    \STATE solve symmetric eigenvalue problem (SEP)
           $H^{(j)}\what P^{(j)}=\what P^{(j)}\Omega_j$ for $\what P^{(j)}\in\bbO^{n\times k}$,
           an orthonormal basis matrix of the eigenspace of $H^{(j)}$ associated with its first $k$ largest eigenvalues;
    \STATE calculate $Q_j\in\bbO^{k\times k}$,
           according to and as required by \eqref{eq:cond4AF-NEPv:cvx}, and let $P^{(j+1)}=\what P^{(j)}Q_j\in\bbP$;
\ENDFOR
\RETURN the last $P^{(j)}$.
\end{algorithmic}
\end{algorithm}

%\begin{theorem}[{\cite[Theorem 6.1]{li:2024}}]\label{thm:H(P)-eligibility-eps}
%$P\in\bbO^{n\times k}$ is a solution to the KKT condition \eqref{eq:OptSTM+L21:KKT'-eps} if and only if
%it is a solution to NEPv \eqref{eq:OptSTM+L21:NEPv'-eps}
%and $D^{\T}P$ is  symmetric.
%\end{theorem}

\begin{theorem}%[{\cite[Theorem 8.1]{li:2024}}]
\label{thm:main-NEPv-cvx+L21:eps}
Consider $\OptSM_{2,1}^{\varepsilon}$~\eqref{eq:OptSTM+L21:eps}. Let $H_{\varepsilon_0}(P)$ be given by \eqref{eq:OptSTM+L21:NEPv'-eps} and
$\bbP$ be the one inherited from \eqref{eq:cond4AF-NEPv:cvx} for $g_i(P)$.
Suppose that $\munderbar\alpha=2$ in \eqref{eq:cond4AF-NEPv:cvx} and that $\psi_i(\bx_{(1:N)})\ge 0$ for those $i$ for which \eqref{eq:cond4AF-NEPv:cvx-b} does not become an equality.
%$$
%\bbP:=\{P\in\bbO^{n\times k}\,:\,P^{\T}D\succeq 0\}\subseteq\bbO^{n\times k}.
%$$
\begin{enumerate}[{\rm (a)}]
  \item {\rm \cite[Theorem 8.1]{li:2024}} For $\what P\in\bbO^{n\times k}$ and $P\in\bbP$,  if
        \begin{equation}\label{eq:NEPv-assume-eps}
        \tr(\what P^{\T}H_{\varepsilon_0}(P)\what P)\ge\tr(P^{\T}H_{\varepsilon_0}(P)P)+\eta
        \quad\mbox{for some $\eta\in\bbR$},
        \end{equation}
        then
        there exists $Q\in\bbO^{k\times k}$ such that $\wtd P=\what PQ\in\bbP$ and
        $f_{\varepsilon_0}(\wtd P)\ge f_{\varepsilon_0}(P)+\eta/2$.
  \item {\rm \cite[Theorem 6.3]{li:2024}}
        Let $P_*\in\bbP$ be an maximizer of $\OptSM_{2,1}^{\varepsilon}$~\eqref{eq:OptSTM+L21:eps}. Then $P_*$ satisfies
        NEPv \eqref{eq:OptSTM+L21:NEPv-eps}, and the eigenvalues of $\Omega_*:=P_*^{\T}H_{\varepsilon_0}(P_*)P_*$ are
        the $k$ largest eigenvalues of $H_{\varepsilon_0}(P_*)$.
\end{enumerate}
\end{theorem}

A self-consistent-field (SCF) iteration
to solve \eqref{eq:OptSTM+L21:NEPv-eps} is outlined in \Cref{alg:NEPvSCF4+L21}, based on
\Cref{thm:main-NEPv-cvx+L21:eps}(a).
A few comments regarding its implementation are in order.
\begin{enumerate}[(1)]
  \item A reasonable stopping criterion at Line 1 is
         \begin{equation}\label{eq:stop-1}
         \varepsilon_{\KKT}:=\frac {\big\|\scrH_{\varepsilon_0}(P)-P\Lambda_{\varepsilon_0}(P)\big\|_{\F}}
                                   {\xi}
               \le\epsilon,
%         \varepsilon_{\KKT}:=\frac {\big\|H_{\varepsilon_0}(P)-P\big[P^{\T}H_{\varepsilon_0}(P)P\big]\big\|_{\F}}
%                                   {\xi}
%               \le\epsilon,
         \end{equation}
         where $\Lambda_{\varepsilon_0}(P):=\big(P^{\T}\scrH_{\varepsilon_0}(P)+[\scrH_{\varepsilon_0}(P)]^{\T}P\big)/2$,
         $\epsilon$ is a preselected tolerance, and $\xi$ is an appropriate normalization factor, dependent on
         the expression for $\scrH_{\varepsilon_0}(P)$. In general,
         $$
         \xi=\sum_{i=1}^N|\psi_i(T_0(P))|\,\left\|\frac{\partial g_i(P)}{\partial P}\right\|_{\F}+n\alpha
         $$
         may work well. Later for our concrete examples in \cref{sec:CEgs}, we will use something that is specifically related to $\scrH_{\varepsilon_0}(P)$ in question.
  \item  According to \Cref{thm:main-NEPv-cvx+L21:eps}(a), there is no need to compute the partial eigendecomposition at Line~3
         accurately up to the working precision but rather it suffices to make \eqref{eq:NEPv-assume-eps} satisfied with a relatively large $\eta>0$.
         This is helpful for overall computational efficiency  when $n$ is large and the partial eigendecomposition is computed iteratively \cite{demm:1997,knya:2001,li:2015,parl:1998,saad:1992}.
  \item  At Line 4 it refers to \eqref{eq:cond4AF-NEPv:cvx} for the calculation of $Q_j$.
         Exactly how it is computed depends on the structure of $g$ in \eqref{eq:g-cvx-comp}.
         For example, when $\tr((P^{\T}D)^m)$ is involved, often $Q_j$ is an orthogonal polar factor of
         $\big[\what P^{(j)}\big]^{\T}D$. We will revisit this issue in \cref{sec:CEgs} for each
         concrete $g(P)$.
\end{enumerate}
Finally with \Cref{thm:main-NEPv-cvx+L21:eps}(a), the general convergence theorems, \cite[Theorems~6.3~and~6.4]{li:2024},
apply.
To save space, we omit  stating them here, except mentioning that
$\{f_{\varepsilon_0}(P^{(j)})\}_{j=0}^{\infty}$
is monotonically increasing, and any accumulation point of $\{P^{(j)}\}_{j=0}^{\infty}$ is
a solution to NEPv \eqref{eq:OptSTM+L21:NEPv-eps}.

\section{Acceleration via LOCG}\label{sec:LOCG}
Algorithm~\ref{alg:NEPvSCF4+L21} involves solving a large eigenvalue problem at its line 3 for large $n$.
One option is to employ an iterative eigen-solver. Another option
is to borrow the idea of the locally optimal conjugate gradient technique (LOCG),
which draws inspiration from optimization \cite{poly:1987,taka:1965} and has been
increasingly used in numerical linear algebra for linear systems and eigenvalue problems \cite{beli:2022,imlz:2016,knya:2001,li:2015,yali:2021}
and more recently in \cite{wazl:2022a} for maximizing the sum of coupled traces
and in \cite{li:2024} for general optimization on the Stiefel manifold.

Without loss of generality, let $P^{(-1)}\in\bbO^{n\times k}$ be the approximate maximizer of $\OptSM_{2,1}^{\varepsilon}$~\eqref{eq:OptSTM+L21:eps}
from the very previous iterative step, and $P\in\bbO^{n\times k}$ the current approximate maximizer.
We are now looking for the next approximate maximizer
$P^{(1)}$, along the line of LOCG, according to
\begin{equation}\label{eq:LOCG}
P^{(1)}=\arg\max_{Y\in\bbO^{n\times k}}f_{\varepsilon_0}(Y),\,\,\mbox{s.t.}\,\, \cR(Y)\subseteq\cR([P,\scrR(P),P^{(-1)}]),
\end{equation}
where $\scrR(P)$ is the gradient of $f_{\varepsilon_0}(\cdot)$ at $P$ with respect to the Stiefel manifold:
\begin{equation}\label{eq:R(P)}
\scrR(P)%:=\grad f_{|{{\mathbb O}^{n\times k}}}(P)=
=\scrH_{\varepsilon_0}(P)-P\cdot\frac 12\Big[P^{\T}\scrH_{\varepsilon_0}(P)+\scrH_{\varepsilon_0}(P)^{\T}P\Big],
\end{equation}
with $\scrH_{\varepsilon_0}(P)$ given as in \eqref{eq:partD+L21:work-eps}.
%\begin{equation}\label{eq:opt-OptSTM+L21:scrH-eps}
%\scrH_{\varepsilon_0}(P):=\frac {\partial f_{\varepsilon_0}(P)}{\partial P}
%    =\sum_{i=1}^N\psi_i(T_0(P))\,\frac {\partial g_i(P)}{\partial P}-\alpha\sum_{i=1}^n\frac {\be_i\be_i^{\T}P}{\sqrt{\tr(P^{\T}\be_i\be_i^{\T}P)+\varepsilon_0^2}},
%\end{equation}
Initially for the first iteration, we don't have $P^{(-1)}$ and
it is understood that $P^{(-1)}$ is absent from \eqref{eq:LOCG}, i.e.,
simply $\cR(Y)\subseteq\cR([P,\scrR(P)])$.

We still have to numerically solve \eqref{eq:LOCG}. For that purpose, let $W\in\bbO^{n\times m}$ be an orthonormal basis matrix of subspace
$\cR([P,\scrR(P),P^{(-1)}])$. Generically, $m=3k$ but $m<3k$ can happen.
It can be implemented by the Gram-Schmidt orthogonalization process, starting with orthogonalizing the columns of $\scrR(P)$ against $P$ since
$P\in\bbO^{n\times k}$ already. In MATLAB, to fully take advantage of its optimized functions, we simply set
$W=[\scrR(P),P^{(-1)}]$ (or $W=\scrR(P)$ for the first iteration) and then  do
\begin{equation}\label{eq:W-compute}
\framebox{
\begin{minipage}{10cm}
\tt      W=W-P*(P'*W); W=orth(W); W=W-P*(P'*W); W=orth(W);\\
      W=[P,W];
\end{minipage}
}
\end{equation}
where the first line  performs the classical Gram-Schmidt orthogonalization twice to almost ensure that
the resulting  columns of $W$ are fully orthogonal to the columns of $P$ at the end of the first line,
and {\tt orth} is a MATLAB function for orthogonalization, which uses the thin SVD. Another alternative is the thin QR:
{\tt [W,$\sim$]=qr(W,0)}, which is cheaper.
It is important to note that the first $k$ columns of
the final $W$ are the same as those of $P$.

Now it follows from $\cR(Y)\subseteq\cR([P,\scrR(P),P^{(-1)}])=\cR(W)$ that in \eqref{eq:LOCG}
\begin{subequations}\label{eq:LOCGsub}
\begin{equation}\label{eq:LOCGsub:Y}
Y=WZ\quad\mbox{for $Z\in\bbO^{m\times k}$}.
\end{equation}
Problem \eqref{eq:LOCG} becomes
\begin{equation}\label{eq:LOCGsub-1}
Z_{\opt}=\arg\max_{Z\in\bbO^{m\times k}} \wtd f_{\varepsilon_0}(Z),
\end{equation}
where, upon setting $\bw_i^{\T}=\be_i^{\T}W$,
\begin{equation}\label{eq:LOCGsub-2}
\wtd f_{\varepsilon_0}(Z):=f_{\varepsilon_0}(WZ)
   =g(WZ)-\alpha\sum_{i=1}^n\sqrt{\tr(Z^{\T}\bw_i\bw_i^{\T}Z)+\varepsilon_0^2}\,.
\end{equation}
\end{subequations}
Finally $P^{(1)}=WZ_{\opt}$ for \eqref{eq:LOCG}.
According to \cite[subsection~6.3]{li:2024},  $\wtd f_{\varepsilon_0}(Z)$ satisfies
{\bf the NEPv Ansatz}, with the symmetric matrix-valued function
\begin{align}
\wtd H_{\varepsilon_0}(Z)
    &=W^{\T}H_{\varepsilon_0}(WZ)W \nonumber\\
    &=\sum_{i=1}^N\psi_i(T_0(WZ))\,W^{\T}H_i(WZ)W
    -\alpha\sum_{i=1}^n\frac {\bw_i\bw_i^{\T}}{\sqrt{\tr(Z^{\T}\bw_i\bw_i^{\T}Z)+\varepsilon_0^2}}
   \in\bbR^{m\times m}, \label{eq:NEPv-LOCG-2}
\end{align}
a much smaller matrix in size than $H_{\varepsilon_0}(P)$.
The associated NEPv is
\begin{equation}\label{eq:NEPv-LOCG}
\wtd H_{\varepsilon_0}(Z)\,Z=Z\wtd\Omega, \quad Z\in\bbO^{m\times k},\quad \wtd\Omega=\wtd\Omega^{\T}\in\bbR^{k\times k},
\end{equation}
There are corresponding versions of both \Cref{thm:main-NEPv-cvx+L21:eps} and, in principle,
\Cref{alg:NEPvSCF4+L21} can be used to solve NEPv~\eqref{eq:NEPv-LOCG}.
\Cref{alg:NEPvLOCG} summarizes the LOCG-accelerated NEPv approach.

\begin{algorithm}[t]
\caption{The LOCG-accelerated NEPv approach for solving $\OptSM_{2,1}^{\varepsilon}$~\eqref{eq:OptSTM+L21:eps}}
\label{alg:NEPvLOCG}
\begin{algorithmic}[1]
\REQUIRE $g(P)$ as in \eqref{eq:g-cvx-comp} satisfying \eqref{eq:cond4AF-NEPv:cvx} with $\munderbar\alpha=2$,
         regularization parameter $\alpha>0$,
         $\varepsilon_0>0$,
         and initial approximation $P^{(0)}\in\bbP$;
\ENSURE  an approximate maximizer of $\OptSM_{2,1}^{\varepsilon}$~\eqref{eq:OptSTM+L21:eps}.
\STATE $P^{(-1)}=[\,]$; \% null matrix
\FOR{$j=0,1,\ldots$ until convergence}
    \STATE compute $W\in\bbO^{n\times m}$ such that $\cR(W)=\cR(\big[P^{(j)},\scrR(P^{(j)}),P^{(j-1)}\big])$ as in \eqref{eq:W-compute}, where
           $\scrR(P^{(j)})$ is calculated according to \eqref{eq:R(P)};
    \STATE solve \eqref{eq:LOCGsub-1} for $Z_{\opt}$ by \Cref{alg:NEPvSCF4+L21} with
    $\wtd H_{\varepsilon_0}(\cdot)$ in \eqref{eq:NEPv-LOCG-2}, initially $Z^{(0)}$
           being the first $k$ columns of $I_m$;
    \STATE $P^{(j+1)}=WZ_{\opt}$;
\ENDFOR
\RETURN the last $P^{(j)}$.
\end{algorithmic}
\end{algorithm}

\begin{remark}\label{rk:SCF4npd+LOCG}
There are a few  comments in order, regarding \Cref{alg:NEPvLOCG}.
\begin{enumerate}[(i)]
  %\item The comment we made for \Cref{alg:NEPvSCF4+L21} about supplying a decent initial still applies.
  \item The stopping criterion \eqref{eq:stop-1} can be used at Line 2.
  \item It is important to compute $W$ at Line~3 in such a way, as explained moments ago, that its first $k$
        columns are exactly the same as those of $P^{(j)}$.
        This is because as $P^{(j)}$ converges, $P^{(j+1)}$ changes little from $P^{(j)}$ and hence $Z_{\opt}$
        is increasingly close to the first $k$ columns of $I_m$. This explains the choice of $Z^{(0)}$
        at Line~4.
  \item At Line 4, some saving can be achieved by reusing qualities that are already computed.
        This will be explained in more detail in our concrete examples in \cref{sec:CEgs}.
  \item An area of improvement is to solve \eqref{eq:LOCGsub-1} with an accuracy, fractionally better than the
        current $P^{(j)}$ as an approximate solution of $\OptSM_{2,1}^{\varepsilon}$~\eqref{eq:OptSTM+L21:eps}.
        Specifically, if we use
        \eqref{eq:stop-1} at Line~2 here to stop the for-loop: Lines 2--6, with tolerance $\epsilon$, then instead of using the same
        $\epsilon$ for \Cref{alg:NEPvSCF4+L21} at its Line~1,
        we can use a fraction, say $1/8$,
        of $\varepsilon_{\KKT}$ evaluated at the current approximation $P=P^{(j)}$ as the stopping tolerance within \Cref{alg:NEPvSCF4+L21}.
\end{enumerate}
\end{remark}
%The reader is referred to \cite{li:2024}

%
%It can verified that
%\begin{subequations}\label{eq:KKT-reduced}
%\begin{equation}\label{eq:KKT-reduced-1}
%\wtd\scrH(Z):=\frac {\partial \wtd f(Z)}{\partial Z}=\left.W^{\T}\frac {\partial f(P)}{\partial P}\right|_{P=WZ}=W^{\T}\scrH(WZ),
%\end{equation}
%and the first order optimality condition for \eqref{eq:LOCGsub-1} is
%\begin{equation}\label{eq:KKT-reduced-2}
%\wtd\scrH(Z)=Z\wtd\Lambda
%\quad \mbox{with} \quad
%\wtd\Lambda^{\T}=\wtd\Lambda\in\bbR^{k\times k},\quad Z\in\bbO^{m\times k}.
%\end{equation}
%\end{subequations}

\section{Effect of $(2,1)$-norm regularization}\label{sec:EffL21}
For a properly chosen $\alpha$, it is expected that optimal $P$ of $\OptSM_{2,1}$~\eqref{eq:OptSTM+L21} will have a few (or many) rows with very small
$\|P_{(i,:)}\|_2$, which suggests that the $i$th feature in the data points  is insignificant and thus can be
deselected. But owing to the fact that $P\in\bbO^{n\times k}$, $P$ has at least $k$ rows that have
nontrivial norms, i.e., not small, as guaranteed by \Cref{thm:not-small} below.

\begin{theorem}[{\cite{wazl:2026}}]\label{thm:not-small}
Let $P\in\bbO^{n\times k}$ and rearrange $\{\|P_{(i,:)}\|_2\}_{i=1}^n$ descendingly as
$$
\|P_{(i_1,:)}\|_2\ge\|P_{(i_2,:)}\|_2\ge\cdots\ge\|P_{(i_n,:)}\|_2.
$$
Then
$$
\|P_{(i_j,:)}\|_2\ge\sqrt{\frac {k-j+1}{n-j+1}}\ge\frac 1{\sqrt{n-k+1}}\quad\mbox{for $1\le j\le k$}.
$$
\end{theorem}

%\begin{proof}
%For $1\le j\le k$, noticing $\|P_{(i,:)}\|_2\le 1$ for any $i$, we have
%$$
%k=\sum_{j=1}^n\|P_{(i_j,:)}\|_2^2
% \le (j-1)+(n-j+1)\|P_{(i_j,:)}\|_2^2,
%$$
%implying
%$$
%\|P_{(i_j,:)}\|_2^2\ge\frac {k-j+1}{n-j+1}\ge\frac 1{n-k+1},
%$$
%as was to be shown.
%\end{proof}

As a consequence of this theorem, $P\in\bbO^{n\times k}$ has at most $n-k$ rows potentially zeros. In fact, such $P$ exists,
for example, any $P$ taking the form
\begin{equation}\label{eq:extreme-STM-pt}
\begin{bmatrix}
                          \what P \\
                          0_{(n-k)\times k}
                        \end{bmatrix}
\quad\mbox{with $\what P\in\bbO^{k\times k}$},
\end{equation}
and, any row-permuted one of it,
have exactly $n-k$ zero rows.
Consider again $\OptSM_{2,1}$~\eqref{eq:OptSTM+L21}. By making $\alpha$ larger and larger, increasingly it becomes minimizing
$\|P\|_{2,1}$ over $P\in\bbO^{n\times k}$. Naturally, we are wondering
\begin{equation}\label{eq:minL21}
\mbox{what is}\quad \arg\min_{P\in\bbO^{n\times k}}\|P\|_{2,1}?
\end{equation}
In what follows, $\|\bx\|_p$ (for $1\le p\le\infty$) stands for the $\ell_p$-norm of a column vector $\bx$.

\begin{lemma}\label{lm:basic-1}
Given $\bz\equiv [z_i]\in\bbR^n$ satisfying
\begin{equation}\label{eq:min||z||1a}
0\le z_i\le 1\,\,\mbox{for $1\le i\le n$, and}\,\,
\|\bz\|_2^2\equiv\sum_{i=1}^nz_i^2=k
\end{equation}
where integer $1\le k\le n$, we have
\begin{equation}\label{eq:min||z||1b}
\sqrt{kn}\ge\|\bz\|_1\equiv\sum_{i=1}^nz_i\ge k
\end{equation}
with the upper bound $\sqrt{kn}$ achieved by and only by $z_i=\sqrt{k/n}$ for $1\le i\le n$ and
the lower bound $k$ achieved by and only by those $\bz\equiv [z_i]$:
\begin{equation}\label{eq:optimal-z}
\mbox{$(z_1,\ldots,z_n)$ is a permutation of $(\,\underbrace{1,\ldots,1}_{k},\underbrace{0,\ldots,0}_{n-k}\,)$}.
\end{equation}
% such that $k$ of the $z_i$'s are $1$ and the rest $n-k$ of
%the $z_i$'s are $0$, i.e., $[\bone_k^{\T},0,\ldots,0]^{\T}$ and all of its permutations.
\end{lemma}

\begin{proof}
The case $k=n$ is a trivial one because then \eqref{eq:min||z||1a} implies all $z_i=1$. Thus $\|\bz\|_1=n=k$.

Assume in what follows that $1\le k<n$.
By the  Cauchy-Bunyakovsky-Schwarz inequality, we get
$$
\sum_{i=1}^nz_i\le\sqrt n\sqrt{\sum_{i=1}^nz_i^2}=\sqrt{nk}
$$
with equality if and only if $\bz$ parallels to the vector of all ones, yielding the first inequality in \eqref{eq:min||z||1b}
and the unique optimal $\bz$ with all $z_i=\sqrt{k/n}$.
On the other hand,
notice that $0\le z_i\le 1$ implies $0\le z_i^2\le z_i\le 1$ and hence
$$
\sum_{i=1}^nz_i\ge \sum_{i=1}^nz_i^2=k,
$$
yielding the second inequality in \eqref{eq:min||z||1b} immediately.

It remains to justify the characterization \eqref{eq:optimal-z} on $\bz\equiv [z_i]$ that achieves the lower bound $k$ in \eqref{eq:min||z||1b}.
%: $k$ of the $z_i$'s are $1$ and the rest $n-k$ of the $z_i$'s are $0$.
To that end, we consider
\begin{equation}\label{eq:min||z||1b:pf-0'}
\min_{\bz\in\mathfrak{D}_n}\|\bz\|_1
\quad\mbox{subject to}\,\, \|\bz\|_2^2=k\,\,\mbox{and all $0\le z_i\le 1$},
\end{equation}
%minimizing $\|\bz\|_1$ over $\bz\in\mathfrak{D}_n$ subject to
%$\|\bz\|_2^2=k$,
where
$
\mathfrak{D}_n:=\{\bz\equiv [z_i]\in\bbR^n\,:\, 0\le z_i\le 1\,\,\,\mbox{for $1\le i\le n$}\}
$
is the $n$-dimensional unit cube. The minimum of $\|\bz\|_1$ in $\mathfrak{D}_n$ subject to $\|\bz\|_2^2=k$ occurs at
one of the critical points of $\|\bz\|_1$ over $\bz\in\mathfrak{D}_n$, subject to $\|\bz\|_2^2=k$.
We know that the critical points consist of the ones in the interior of $\mathfrak{D}_n$
and on its boundary, subject to $\|\bz\|_2^2=k$.
\begin{enumerate}[(a)]
  \item The critical points in the interior of $\mathfrak{D}_n$ subject to $\|\bz\|_2^2=k$ are determined through
        Lagrangian
        $$
        \scrL(\bz, \lambda)=\sum_{i=1}^nz_i-\frac {\lambda}2\big(\sum_{i=1}^nz_i^2-k\big).
        $$
        Specifically, the interior critical points satisfy
        $$
        \frac {\partial \scrL(\bz, \lambda)}{\partial z_i}=1-\lambda z_i=0
        \quad\Rightarrow\quad
        z_i=1/\lambda\,\,\mbox{for $1\le i\le n$},
        $$
        and thus $z_i=\sqrt{k/n}$ upon using $\|\bz\|_2^2=k$. So there is just one interior critical
        point: $\bz$ with all $z_i=\sqrt{k/n}$, for which $\|\bz\|_1=\sqrt{kn}>k$.
  \item Consider the boundary of $\mathfrak{D}_n$, i.e., some $z_i\in\{0,1\}$. It can be seen that
        the boundary points can be divided into groups
        $$
        \mathfrak{B}_{(n_0,n_1)}
        =\{\mbox{$\bz\in\mathfrak{D}_n$ has $n_0$ entries  $0$ and $n_1$ entries  $1$ and the rest in $(0,1)$}\}.
        $$
        The entries of $\bz\in\mathfrak{B}_{(n_0,n_1)}$, upon a proper permutation, can be listed as
        $$
        \underbrace{z_{i_1},\ldots, z_{i_{\hat n}}}_{\hat n}, \underbrace{0,\ldots,0}_{n_0}, \underbrace{1,\ldots,1}_{n_1},
        $$
        where $\hat n=n-n_0-n_1$ and $1\le i_1<\cdots<i_{\hat n}\le n$. For our purpose, we are interested only in those such that
        \begin{equation}\label{eq:min||z||1b:pf-0}
        \|\bz\|_2^2=n_1+\sum_{j=1}^{\hat n}z_{i_j}^2=k
        \quad\Rightarrow\quad
        \sum_{j=1}^{\hat n}z_{i_j}^2=k-n_1=:\hat k.
        \end{equation}
        In particular, $n> k\ge n_1$. There are two cases: $n_1=k$ or $n_1<k$, to consider.

        If $n_1=k$ then $\hat n=0$ and thus $n_0=n-k$
        for which case $\bz$ meets the characterization \eqref{eq:optimal-z} already and $\|\bz\|_1=k$.

        Next consider $n_1< k$, and then $\hat n>\hat k$ by \eqref{eq:min||z||1b:pf-0} because all $z_{i_j}\in(0,1)$.
        Therefore, minimizing $\|\bz\|_1$ over $\bz\in\mathfrak{B}_{(n_0,n_1)}$ is turned into, for any given indices $1\le i_1<\cdots<i_{\hat n}\le n$,
        \begin{equation}\label{eq:min||z||1b:pf-1}
        \min\sum_{j=1}^{\hat n}z_{i_j}
        \quad\mbox{subject to}\,\, \sum_{j=1}^{\hat n}z_{i_j}^2=\hat k\,\,\mbox{and all $0\le z_{i_j}\le 1$},
        \end{equation}
        where $\hat n>\hat k$. We allow $z_{i_j}$ possibly taking $0$ or $1$ to so that
        \eqref{eq:min||z||1b:pf-1} is essentially the problem of minimizing $\|\bw\|_1$
        over $\bw\in\mathfrak{D}_{\hat n}$ subject to $\|\bw\|_2^2=\hat k$, having
        the same type as \eqref{eq:min||z||1b:pf-0'}.
%        minimizing $\|\bz\|_1$
%        over $\mathfrak{D}_n$ subject to $\|\bz\|_2^2=k$.
        It can be seen that any  boundary case of \eqref{eq:min||z||1b:pf-1}, i.e., some $z_{i_j}\in\{0,1\}$,
        is a boundary case of original \eqref{eq:min||z||1b:pf-0'}, and so
        we only need to worry about interior critical point(s) for \eqref{eq:min||z||1b:pf-1}.
        By what we have done in item~(a),
        problem \eqref{eq:min||z||1b:pf-1} has one interior critical point in $\mathfrak{D}_{\hat n}$ given by all $z_{i_j}=\sqrt{\hat k/\hat n}$, yielding
        $$
        \|\bz\|_1=\hat n\cdot\sqrt{\hat k/\hat n}+n_1=\sqrt{\hat n\hat k}+n_1> \hat k+n_1=k.
        $$
%        All cases for the boundary of $\mathfrak{D}_{\hat n}$ for \eqref{eq:min||z||1b:pf-1} are covered
%        as these cases are part of those for the boundary of $\mathfrak{D}_n$ for minimizing $\|\bz\|_1$ and thus
%        do not need to be investigated separately.

        In summary, we find $\|\bz\|_1>k$ at all critical points $\bz$ on the boundary of $\mathfrak{D}_n$,
        subject to $\|\bz\|_2^2=k$,  except for those that meet the characterization \eqref{eq:optimal-z}, and for the latter, clearly $\|\bz\|_1=k$.
\end{enumerate}
The proof is completed.
\end{proof}

\begin{theorem}\label{thm:extreme-STM-pt}
Let $P\in\bbO^{n\times k}$. Then $\|P\|_{2,1}\ge k$ with equality if and only if
$P$ can be permuted, through permuting its rows, into a matrix in the form of \eqref{eq:extreme-STM-pt}.
% is essential a permutation matrix, except that its $k$ entries of $1$ can be either $1$ or $-1$.
\end{theorem}

\begin{proof}
Let $\bz\equiv [z_i]\in\bbR^n$ with $z_i=\|P_{(i,:)}\|_2$. Then $\bz$ satisfies \eqref{eq:min||z||1a}, and thus
$\|P\|_{2,1}=\|\bz\|_1\ge k$ by \Cref{lm:basic-1}.   When $\|\bz\|_1=\|P\|_{2,1}=k$, by \Cref{lm:basic-1},
$\bz$ can be permuted to $[\bone_k^{\T},0,\ldots,0]^{\T}$. The same permutation will permute
$P$, through permuting its rows, to a matrix in the form of \eqref{eq:extreme-STM-pt} which is also in $\bbO^{n\times k}$
because $P$ is. On the other hand, if $P$ can be permuted, through permuting its rows,
to a matrix as in \eqref{eq:extreme-STM-pt},
then  $\|P\|_{2,1}=k$.
\end{proof}

\begin{theorem}\label{thm:extreme-STM-pt:lim}
For any $\alpha>0$, let $P_{\alpha}$ be a maximizer of $\OptSM_{2,1}$~\eqref{eq:OptSTM+L21}. Consider $\{P_{\alpha}\,:\alpha>0\}$ and let $P_*$
be an accumulation point of $\{P_{\alpha}\,:\alpha>0\}$ as $\alpha\to\infty$. Then $P_*\in\bbO^{n\times k}$ and $\|P_*\|_{2,1}=k$, and thus $P_*$ can be permuted, through permuting its rows, to a matrix in the form of \eqref{eq:extreme-STM-pt}.
\end{theorem}

\begin{proof}
Because all $P_{\alpha}\in\bbO^{n\times k}$, we know $P_*\in\bbO^{n\times k}$. We claim that $\|P_*\|_{2,1}=k$; otherwise
$\|P_*\|_{2,1}>k$. Suppose that were the case and let $\delta=\|P_*\|_{2,1}-k>0$.
Since $P_*$ is an accumulation point of $\{P_{\alpha}\,:\alpha>0\}$ as $\alpha\to\infty$, there exists an increasing sequence $\{\alpha_j\}_{j=1}^{\infty}$
such that
\begin{equation}\label{eq:extreme-STM-pt-pf-1}
\lim_{j\to\infty}\alpha_j=\infty, \quad \lim_{j\to\infty} P_{\alpha_j}=P_*.
\end{equation}
By the definition of $P_{\alpha}$, we know
$g(P_{\alpha_j})-\alpha_j\|P_{\alpha_j}\|_{2,1}\ge g(P)-\alpha_j\|P\|_{2,1}$ for any $P\in\bbO^{n\times k}$ and any $j$, or equivalently
\begin{equation}\label{eq:extreme-STM-pt-pf-2}
\frac {g(P_{\alpha_j})}{\alpha_j}-\|P_{\alpha_j}\|_{2,1}\ge\frac {g(P)}{\alpha_j}-\|P\|_{2,1}
\quad\mbox{for any $P\in\bbO^{n\times k}$ and any $j$}.
\end{equation}
Also because of the smoothness assumption on $g(P)$ in $\OptSM_{2,1}$~\eqref{eq:OptSTM+L21}, there exists $\Gamma>0$ such that $|g(P)|\le\Gamma$
for all $P\in\bbO^{n\times k}$. By \eqref{eq:extreme-STM-pt-pf-1} and the continuity of $\|\cdot\|_{2,1}$, there exists integer $j_0$
such that
\begin{equation}\label{eq:extreme-STM-pt-pf-3}
\frac {\Gamma}{\alpha_{j_0}}\le\frac {\delta}4, \quad \|P_{\alpha_{j_0}}\|_{2,1}\ge\|P_*\|_{2,1}-\frac {\delta}4=k+\frac 34\delta.
\end{equation}
Let $\wtd P\in\bbO^{n\times k}$ such that $\|\wtd P\|_{2,1}=k$. Such $\wtd P$ exists, for example the one in
\eqref{eq:extreme-STM-pt}. We have by \eqref{eq:extreme-STM-pt-pf-3}
\begin{equation}\label{eq:extreme-STM-pt-pf-4}
\frac {g(\wtd P)}{\alpha_{j_0}}-\|\wtd P\|_{2,1}
  \ge-\frac {\Gamma}{\alpha_{j_0}}-k\ge -\frac {\delta}4-k.
\end{equation}
On the other hand, we also have, by \eqref{eq:extreme-STM-pt-pf-3},
$$
\frac {g(P_{\alpha_{j_0}})}{\alpha_{j_0}}-\|P_{\alpha_{j_0}}\|_{2,1}
  \le \frac {\Gamma}{\alpha_{j_0}}-k-\frac 34\delta\le -k-\frac {\delta}2,
$$
which, together with \eqref{eq:extreme-STM-pt-pf-4}, lead to
$$
\frac {g(P_{\alpha_{j_0}})}{\alpha_{j_0}}-\|P_{\alpha_{j_0}}\|_{2,1}
   \le -k-\frac {\delta}2<-k-\frac {\delta}4\le\frac {g(\wtd P)}{\alpha_{j_0}}-\|\wtd P\|_{2,1},
$$
contradicting \eqref{eq:extreme-STM-pt-pf-2} with $j=j_0$ and $P=\wtd P$.
Therefore $\|P_*\|_{2,1}=k$, as was to be shown.
\end{proof}

What \Cref{thm:extreme-STM-pt:lim} says,  any maximizer of $\OptSM_{2,1}$~\eqref{eq:OptSTM+L21}
selects exactly $k$ rows (or features in the terminology of machine learning) in the limit as $\alpha\to\infty$.

\begin{remark}\label{rk:extreme-STM-pt}
There are a few comments in order.
\begin{enumerate}[(i)]
  \item Although $\OptSM_{2,1}$~\eqref{eq:OptSTM+L21} may have more than one maximizer for given $\alpha>0$, $P_{\alpha}$ in \Cref{thm:extreme-STM-pt:lim} can be any one of the maximizers.
  \item \Cref{thm:extreme-STM-pt:lim} does not say that $\lim P_{\alpha}$ as $\alpha\to\infty$ exists
        but merely that any accumulation point $P_*$, as $\alpha\to\infty$,  has $\|P_*\|_{2,1}=k$.
  \item In actual computations, we may select an increasing sequence $\{\alpha_j\}$ and compute one maximizer
        $P_{\alpha_j}$ for each $\alpha_j$ up to $\alpha_j$ sufficiently large so that $\|P_{\alpha_j}\|_{2,1}\approx k$. When that is the case, $P_{\alpha_j}$ will have $n-k$ rows with negligible norms while the norms of the other $k$ rows are
        approximately $1$. Conceivably, $P_{\alpha_{j-1}}$ can be used as an initial for computing he next $P_{\alpha_j}$.
  \item \Cref{thm:extreme-STM-pt:lim} does not require $g(P)$ taking the form as in \eqref{eq:g-cvx-comp}
        that our NEPv approach in \cref{sec:NEPv+L21} relies on for convergence guarantee. In fact,
        the proof of the theorem only demands that $g(\cdot)$ is bounded on $\bbO^{n\times k}$.
\end{enumerate}
\end{remark}

\Cref{thm:extreme-STM-pt:lim} is about problem $\OptSM_{2,1}$~\eqref{eq:OptSTM+L21}, but what we are promoting computationally is
the perturbed problem \eqref{eq:OptSTM+L21:eps}. Naturally we may wonder if we could have something
similar for the latter.

\begin{lemma}\label{lm:basic-2}
Given $\epsilon_0\ge 0$ and $\bz=[z_i]\in\bbR^n$ satisfying \eqref{eq:min||z||1a}, we have
\begin{equation}\label{eq:min||z||1b-2}
\sqrt{n(k+n\epsilon_0^2)}\ge h(\bz):=\sum_{i=1}^n\sqrt{z_i^2+\epsilon_0^2}\ge k\sqrt{1+\epsilon_0^2}+(n-k)\epsilon_0,
\end{equation}
where the upper bound $\sqrt{n(k+n\epsilon_0^2)}$ on $h(\bz)$ is achieved by and only by $\bz$ with $z_i=\sqrt{k/n}$
               for $1\le i\le n$, and
the lower bound $k\sqrt{1+\epsilon_0^2}+(n-k)\epsilon_0$ on $h(\bz)$ is achieved by and only by $\bz\equiv [z_i]$
as characterized by \eqref{eq:optimal-z}.
% such that $k$ of the $z_i$'s are $1$ and the rest $n-k$ of
%the $z_i$'s are $0$, i.e., $[\bone_k^{\T},0,\ldots,0]^{\T}$ and all of its permutations.
As a consequence, it is implied that
\begin{equation}\label{eq:min||z||1b-2'}
\sqrt{n(k+n\epsilon_0^2)}> k\sqrt{1+\epsilon_0^2}+(n-k)\epsilon_0
\quad\mbox{if $k<n$}.
\end{equation}
\end{lemma}

\begin{proof}
Inequality \eqref{eq:min||z||1b-2'} is a corollary of \eqref{eq:min||z||1b-2} because, as we will show in a minute,
$\sqrt{n(k+n\epsilon_0^2)}$ is the maximum of $h(\cdot)$ achieved by and only by $\bz$ with all $z_i=\sqrt{k/n}$
whereas the right-side of \eqref{eq:min||z||1b-2'} is $h(\cdot)$ evaluated at $\bz=[\bone_k^{\T},0,\ldots,0]^{\T}$,
a point that is different from the maximizer as described moments ago if $k<n$.

The case $k=n$ is a trivial one because then \eqref{eq:min||z||1a} implies all $z_i=1$.
Thus $h(\bz)=n\sqrt{1+\epsilon_0^2}$.
Assume in what follows that $1\le k<n$.

By the  Cauchy-Bunyakovsky-Schwarz
inequality, we get
$$
\sum_{i=1}^n\sqrt{z_i^2+\epsilon_0^2}\le\sqrt n\cdot\sqrt{\sum_{i=1}^n(z_i^2+\epsilon_0^2)}=\sqrt{n(k+n\epsilon_0^2)}
$$
with equality if and only if all $\sqrt{z_i^2+\epsilon_0^2}$ are the same,
yielding all $z_i=\sqrt{k/n}$. This proves the first inequality in \eqref{eq:min||z||1b-2} and when and only when it becomes an equality.

Next, we prove the second inequality in \eqref{eq:min||z||1b-2}.
To that end, we consider
\begin{equation}\label{eq:minh(z)1b:pf-0'}
\min_{\bz\in\mathfrak{D}_n} h(\bz)
\quad\mbox{subject to}\,\, \|\bz\|_2^2=k\,\,\mbox{and all $0\le z_i\le 1$},
\end{equation}
%minimizing
%$h(\bz)$
%over $\bz\in\mathfrak{D}_n$, subject to $\|\bz\|_2^2=k$,
where
$\mathfrak{D}_n$ (and $\mathfrak{B}_{(n_0,n_1)}$ below too)
are the same ones as in the proof of \Cref{lm:basic-1}.
By calculus, $h(\bz)$ achieves its minimum at one of its critical points over $\bz\in\mathfrak{D}_n$, subject to $\|\bz\|_2^2=k$.
The critical points consist of the ones in the interior of $\mathfrak{D}_n$
and on its boundary, subject to $\|\bz\|_2^2=k$.
\begin{enumerate}[(i)]
  \item The critical points in the interior of $\mathfrak{D}_n$ subject to $\|\bz\|_2^2=k$ are determined through
        Lagrangian
        $$
        \scrL(\bz, \lambda)=h(\bz)-\frac {\lambda}2\big(\sum_{i=1}^nz_i^2-k\big).
        $$
        Specifically, the interior critical points satisfy
        \begin{equation}\label{eq:min||z||1b-2:pf-1}
        \frac {\partial \scrL(\bz, \lambda)}{\partial z_i}
          =\frac {z_i}{\sqrt{z_i^2+\epsilon_0^2}}-\lambda z_i=0
        \quad\Rightarrow\quad
        z_i\left(\frac 1{\sqrt{z_i^2+\epsilon_0^2}}-\lambda\right)=0\,\,\mbox{for $1\le i\le n$}.
        \end{equation}
        Since we are looking at the interior points of $\mathfrak{D}_n$, i.e., all $0<z_i<1$, we get from
        \eqref{eq:min||z||1b-2:pf-1} that $1/\sqrt{z_i^2+\epsilon_0^2}=\lambda$ for all $i$ and thus all $z_i$ are the same,
        yielding all $z_i=\sqrt{k/n}$ by $\|\bz\|_2^2=k$, for which $h(\bz)=\sqrt{n(k+n\epsilon_0^2)}> k\sqrt{1+\epsilon_0^2}+(n-k)\epsilon_0$.
  \item Consider the boundary of $\mathfrak{D}_n$, i.e., some $z_i\in\{0,1\}$. The lines of argument
        in the first paragraph of item (b) in the proof of \Cref{lm:basic-1} remain valid.

        If $n_1=k$ then $\hat n=0$ and thus $n_0=n-k$
        for which case $\bz$ meets the characterization \eqref{eq:optimal-z} already and $h(\bz)=k\sqrt{1+\epsilon_0^2}+(n-k)\epsilon_0$.

        Next consider $n_1< k$, and then $\hat n>\hat k$ by \eqref{eq:min||z||1b:pf-0} because all $z_{i_j}\in(0,1)$.
        Therefore, minimizing $h(\bz)$ over $\bz\in\mathfrak{B}_{(n_0,n_1)}$ is turned into, for any given indices $1\le i_1<\cdots<i_{\hat n}\le n$,
        \begin{equation}\label{eq:min||z||1b-2:pf-2}
        \min\sum_{j=1}^{\hat n}\sqrt{z_{i_j}^2+\epsilon_0^2}
        \quad\mbox{subject to}\,\, \sum_{j=1}^{\hat n}z_{i_j}^2=\hat k\,\,\mbox{and all $0\le z_{i_j}\le 1$},
        \end{equation}
        where $\hat n>\hat k$. By what we have done in item~(i),
        problem \eqref{eq:min||z||1b-2:pf-2} has one interior critical point in $\mathfrak{D}_{\hat n}$ given by all $z_{i_j}=\sqrt{\hat k/\hat n}$, yielding
        \begin{align*}
        h(\bz)&=n_0\epsilon_0+n_1\sqrt{1+\epsilon_0^2}+\sqrt{\hat n(\hat k+\hat n\epsilon_0^2)} \\
          &>n_0\epsilon_0+n_1\sqrt{1+\epsilon_0^2}+\hat k\sqrt{1+\epsilon_0^2}+(\hat n-\hat k)\epsilon_0 \\
          &=k\sqrt{1+\epsilon_0^2}+(n-k)\epsilon_0.
        \end{align*}
        Any boundary case for \eqref{eq:min||z||1b-2:pf-2} is a boundary case for the original \eqref{eq:minh(z)1b:pf-0'}, and thus do not need to be investigated separately.

        In summary, we find $h(\bz)>k\sqrt{1+\epsilon_0^2}+(n-k)\epsilon_0$ at all critical points $\bz$ on the boundary of $\mathfrak{D}_n$,
        subject to $\|\bz\|_2^2=k$,  except for those that meet the characterization \eqref{eq:optimal-z}, and for the latter, clearly
        $h(\bz)=k\sqrt{1+\epsilon_0^2}+(n-k)\epsilon_0$.
\end{enumerate}
The proof is completed.
\end{proof}

Similar to \Cref{thm:extreme-STM-pt:lim}, with the help of \Cref{lm:basic-2}, we arrive at the following theorem.

\begin{theorem}\label{thm:extreme-STM-pt:work-eps}
Let $P\in\bbO^{n\times k}$ and $\epsilon_0\ge 0$. Then
\begin{equation}\label{eq:||z||1:work-eps}
\sum_{i=1}^n\sqrt{\|P_{(i,:)}\|_2^2+\epsilon_0^2}\ge k\sqrt{1+\epsilon_0^2}+(n-k)\epsilon_0
\end{equation}
with equality if and only if
$P$ can be permuted, through permuting its rows, to one in the form of \eqref{eq:extreme-STM-pt}.
\end{theorem}

By \Cref{lm:basic-2}, \eqref{eq:||z||1:work-eps} becoming an equality is equivalent to $\|P\|_{2,1}=k$.
Finally, \Cref{thm:extreme-STM-pt:lim} almost holds verbatim for the perturbed problem \eqref{eq:OptSTM+L21:eps}.

\begin{theorem}\label{thm:extreme-STM-pt:work-eps'}
For any $\alpha>0$, let $P_{\alpha}$ be a maximizer of $\OptSM_{2,1}^{\varepsilon}$~\eqref{eq:OptSTM+L21:eps}. Consider $\{P_{\alpha}\,:\alpha>0\}$ and let $P_*$
be an accumulation point of $\{P_{\alpha}\,:\alpha>0\}$ as $\alpha\to\infty$. Then $P_*\in\bbO^{n\times k}$ and
$$
\sum_{i=1}^n\sqrt{\|P_{*(i,:)}\|_2^2+\epsilon_0^2}= k\sqrt{1+\epsilon_0^2}+(n-k)\epsilon_0,
$$
and thus $P_*$ can be permuted, through permuting its rows, to a matrix in the form of \eqref{eq:extreme-STM-pt}.
\end{theorem}

% Although, in stating \Cref{alg:NEPvSCF4+L21}, we request an to start the SCF iteration, we have to be very careful about what such as initial is to supply. Later in \Cref{sec:EffL21} we argue that

%With random initial $P^{(0)}\in\bbP$,
%        chances are great for the eventually computed $P^{(j)}$ to fall into one of the equivalent classes of KKT points,
%        which does not materially maximize $g(P)$ at all, defeating the purpose of finding a row sparse optimizer to maximize $g(P)$.
%        Numerically, we do witness such a phenomenon. In light of this, we suggest to supply an initial $P^{(0)}\in\bbP$  that
%        approximately maximizes $g(P)$ somewhat. One way to do so is to run \Cref{alg:NEPvSCF4+L21} with $\alpha=0$ first, i.e.,
%        solve $\OptSM$~\eqref{eq:OptSTM} roughly and supply the computed solution as an initial.

\section{Applications to common learning objectives}\label{sec:CEgs}
In this section, we will substantiate $\OptSM_{2,1}$~\eqref{eq:OptSTM+L21} with a few concrete objectives $g(P)$, in the form
\eqref{eq:g-cvx-comp}, that have been commonly used in machine learning. In cases where matrix trace function $\tr(P^{\T}D)$
is involved, $D$ plays a role of selector \cite{luli:2024,teli:2025} and we need
$$
\bbP\in\bbO_{D+}^{n\times k}:=\{P\in\bbO^{n\times k}\,:\,P^{\T}D\succeq 0\}\subset \bbO^{n\times k}
$$
as in \Cref{tbl:AF-NEPv},
and accordingly we take $Q_j$ at Line 4 of \Cref{alg:NEPvSCF4+L21} to be an orthogonal polar factor
of $\big[\what P^{(j)}\big]^{\T}D$ \cite{high:1986,high:2008};
otherwise $\bbP=\bbO^{n\times k}$  and $Q_j=I_k$.

For all concrete objectives in what follows, \Cref{thm:main-NEPv-cvx+L21:eps} and \Cref{alg:NEPvSCF4+L21,alg:NEPvLOCG}
are  readily applicable.
In fact, \Cref{thm:main-NEPv-cvx+L21:eps}(b) can be slightly strengthened. As is, $P_*\in\bbP$ is a condition to \Cref{thm:main-NEPv-cvx+L21:eps}(b). For all concrete objectives in this section, this condition can be removed because always $P_*\in\bbP$
is guaranteed for any maximizer. Also those concrete objectives fall into the kinds investigated
in \cite{ball:2022,luli:2024} on linear convergence analysis of SCF.

\subsection{MAXBET type with the $(2,1)$-norm regularization}\label{ssec:MAXBET+L21}
In statistical learning, sometimes, the following optimization problem
\begin{equation}\label{eq:OptSTM-MAXBET}
\max_{P\in\bbO^{n\times k}}\Big\{g(P):=\tr(P^{\T}AP)+2\tr(P^{\T}D)\Big\}
\end{equation}
 has to be solved,
where symmetric $A\in\bbR^{n\times n}$ and $D\in\bbR^{n\times k}$. Most commonly,
the MAXBET\footnote {MAXBET is a statistical term that appeared first in \cite{vdge:1984} in 1984
and then in \cite{tebe:1988}.}
subproblem can be turned into this form \cite[subsection~5.1]{wazl:2022a} and so can
the unbalanced Procrustes problem \cite{chtr:2001,edas:1999,elpa:1999,godi:2004,huca:1962,zhdu:2006}.

When $D=0$ and $A\succeq 0$, \eqref{eq:OptSTM-MAXBET} is an equivalent problem for PCA.
Hence the associated $\OptSM_{2,1}$~\eqref{eq:OptSTM+L21} can be considered as a mathematical model for sparse PCA.
In particular, as $\alpha\to\infty$, it may pick up the most significant
$k\times k$ principle submatrix of $A$.

Objective $g(P)$ of \eqref{eq:OptSTM-MAXBET} takes the form \eqref{eq:g-cvx-comp}:
\begin{equation}\label{eq:obj-MAXBET}
g(P)=\psi\circ T_0(P)
\quad\mbox{with}\quad
\psi(\bx)=-x_1+2x_2, \quad
T_0(P)=\begin{bmatrix}
         \tr(P^{\T}AP) \\
         \tr(P^{\T}D)
       \end{bmatrix}.
\end{equation}
Since $\psi(\bx)$ is linear in $\bx\equiv[x_i]\in\bbR^2$, it is convex for $\bx\in\bbR^2$. Both
$\tr(P^{\T}AP)$ and $\tr(P^{\T}D)$ are atomic functions for NEPv, according to \Cref{tbl:AF-NEPv}, and more
importantly, \eqref{eq:cond4AF-NEPv:cvx-b} for $\tr(P^{\T}AP)$ is an equality.
% meaning that the sign of $\psi_1(\bx)=-1$ is fine as far as the applicability of \Cref{thm:main-NEPv-cvx} is concerned.

For $\OptSM_{2,1}^{\varepsilon}$~\eqref{eq:OptSTM+L21:eps} with $g(P)$  in \eqref{eq:OptSTM-MAXBET}, we have, corresponding to \eqref{eq:OptSTM+L21:NEPv'-eps} and \eqref{eq:KKT+L21:work-eps},
\begin{align}
\scrH_{\varepsilon_0}(P) %:=\frac {\partial f_{\varepsilon_0}(P)}{\partial P}
    &=2AP+2D-\alpha\sum_{i=1}^n\frac {\be_i\be_i^{\T}P}{\sqrt{\tr(P^{\T}\be_i\be_i^{\T}P)+\varepsilon_0^2}},
         \label{eq:partD-MAXBET+L21:scrH-eps} \\
H_{\varepsilon_0}(P)
     &=2A+2\big(DP^{\T}+PD^{\T})
       -\alpha\sum_{i=1}^n\frac {\be_i\be_i^{\T}}{\sqrt{\tr(P^{\T}\be_i\be_i^{\T}P)+\varepsilon_0^2}}.
         \label{eq:opt-MAXBET+L21:H(P)-eps}
\end{align}
It can be verified that $H_{\varepsilon_0}(P)P-\scrH_{\varepsilon_0}(P)=P\big(2D^{\T}P\big)$ for $P\in\bbO^{n\times k}$.
%\begin{equation}\label{eq:H(p)-scrH(P)-eps:MAXBET}
%H_{\varepsilon_0}(P)P-\scrH_{\varepsilon_0}(P)=P\big(2D^{\T}P\big).
%\end{equation}
Thus by \Cref{thm:H(P)-eligibility}, $P\in\bbO^{n\times k}$ is a solution to the KKT condition \eqref{eq:KKT+L21:work-eps-1} with
\eqref{eq:partD-MAXBET+L21:scrH-eps} if and only if
it is a solution to NEPv \eqref{eq:OptSTM+L21:NEPv-eps} with \eqref{eq:opt-MAXBET+L21:H(P)-eps}
and $D^{\T}P$ is  symmetric.
A reasonable normalization quantity $\xi$ in \eqref{eq:stop-1} for the stopping criterion is
$$
\xi=2\|A\|_{\F}+2\|D\|_{\F}+n\alpha.
$$

%\Cref{thm:main-NEPv-cvx+L21:eps} and \Cref{alg:NEPvSCF4+L21} are now readily applicable, and so is
For \Cref{alg:NEPvLOCG}, we notice from \eqref{eq:NEPv-LOCG-2} that now
\begin{equation}
\wtd H_{\varepsilon_0}(Z)
   =2\wtd A+2\big(\wtd DZ^{\T}+Z\wtd D^{\T})
    -\alpha\sum_{i=1}^n\frac {\bw_i\bw_i^{\T}}{\sqrt{\tr(Z^{\T}\bw_i\bw_i^{\T}Z)+\varepsilon_0^2}}
   \in\bbR^{m\times m}, \label{eq:NEPv-LOCG:MAXBET}
\end{equation}
where $\wtd A=W^{\T}AW$, $\wtd D=W^{\T}D$.
Previously in \Cref{rk:SCF4npd+LOCG} we mentioned some saving can be achieved from properly implementing
\Cref{alg:NEPvLOCG}. For the case, we may use
        $$
        AP^{(j+1)}=(AW)Z_{\opt}, \quad [P^{(j+1)}]^{\T}AP^{(j+1)}=Z_{\opt}^{\T}(W^{\T}AW)Z_{\opt}
        $$
        to compute
        the next $AP^{(j+1)}$ and $[P^{(j+1)}]^{\T}AP^{(j+1)}$ at  costs $O(nk^2)$ and $O(k^3)$, respectively,
        instead of $O(n^2k)$ by reusing $AW$ and $W^{\T}AW$.

%We point out in passing that, for the case $D=0$, solving \eqref{eq:OptSTM+L21:eps} with $g(P)=\tr(P^{\T}AP)$
%is essentially computing a sparse PCA on $A$.

\subsection{$\Theta$-trace ratio with the $(2,1)$-norm regularization}\label{ssec:ThetaTR+L21}
In \cite{wazl:2023}, the $\theta$-trace ratio ($\Theta$TR)  problem was investigated, together
with an application to multi-view learning. The objective takes
the form
\begin{equation}\label{eq:obj-ThetaTR:0}
\frac {\tr(P^{\T}A_2P)+\tr(P^{\T}D)}{[\tr(P^{\T}A_1P)]^{\theta}},
\end{equation}
where  $0\le\theta\le 1$ is a contrastive parameter. Although it encompasses the objectives for LDA, OCCA, and
MAXBET,
it in general cannot be written as a convex composition of three matrix trace functions:
$\tr(P^{\T}A_1P)$, $\tr(P^{\T}A_2P)$ and $\tr(P^{\T}D)$. However, its square is, albeit for $0\le\theta\le 1/2$ only \cite[Remark~8.1]{li:2024}. For this reason, we will have to restrict $\theta$ to
$0\le\theta\le 1/2$ first and then offer a workaround for $1/2<\theta\le 1$.

\subsubsection{Case $0\le\theta\le 1/2$}\label{sssec:TTR-1}
Assume $0\le\theta\le 1/2$ and consider
\begin{subequations}\label{eq:obj-ThetaTR}
\begin{equation}\label{eq:obj-ThetaTR-1}
g(P)=\left[\frac {\tr(P^{\T}A_2P)+\tr(P^{\T}D)}{[\tr(P^{\T}A_1P)]^{\theta}}\right]^2
    =\psi\circ T_0(P),
\end{equation}
where $\bbR^{n\times n}\ni A_1,\, A_2\succeq 0$ with $\rank(A_1)>n-k$, $D\in\bbR^{n\times k}$, and
\begin{equation}\label{eq:obj-ThetaTR-2}
\psi(\bx)=\frac {(x_2+x_3)^2}{x_1^{2\theta}},\quad
T_0(P)=\begin{bmatrix}
                             \tr(P^{\T}A_1P) \\
                             \tr(P^{\T}A_2P) \\
                             \tr(P^{\T}D)
                           \end{bmatrix}.
\end{equation}
\end{subequations}
For $0\le\theta\le 1/2$, $\psi(\bx)$ is convex
\cite[Remark~8.1]{li:2024} for
\begin{equation}\label{eq:frakD-TTR}
\bx\in\mathfrak{D}=\{\bx\equiv[x_i]\in\bbR^3\,:\,x_1>0,\,x_2+x_3\ge 0\}.
\end{equation}
For $\OptSM_{2,1}^{\varepsilon}$~\eqref{eq:OptSTM+L21:eps} with $g(P)$  in \eqref{eq:obj-ThetaTR}, we have, corresponding to \eqref{eq:OptSTM+L21:NEPv'-eps} and \eqref{eq:KKT+L21:work-eps},
\begin{align}
\scrH_{\varepsilon_0}(P)%:=\frac {\partial f_{\varepsilon_0}(P)}{\partial P}
    &=\psi_1(\bx)2A_1P+\psi_2(\bx)2A_2P+\psi_3(\bx)D
      -\alpha\sum_{i=1}^n\frac {\be_i\be_i^{\T}P}{\sqrt{\tr(P^{\T}\be_i\be_i^{\T}P)+\varepsilon_0^2}},
         \label{eq:partD-ThetaTR+L21:scrH-eps} \\
H_{\varepsilon_0}(P)
     &=\psi_1(\bx)2A_1+\psi_2(\bx)2A_2+\psi_3(\bx)\big(DP^{\T}+PD^{\T})
       -\alpha\sum_{i=1}^n\frac {\be_i\be_i^{\T}}{\sqrt{\tr(P^{\T}\be_i\be_i^{\T}P)+\varepsilon_0^2}},
         \label{eq:opt-ThetaTR+L21:H(P)-eps}
\end{align}
where $\bx\equiv[x_i]=T_0(P)$, and partial derivatives
$$
\psi_1(\bx)=-2\theta\frac {(x_2+x_3)^2}{x_1^{2\theta+1}}, \quad
\psi_2(\bx)=\psi_3(\bx)=2\frac {x_2+x_3}{x_1^{2\theta}}.
$$
It can be verified that $H_{\varepsilon_0}(P)P-\scrH_{\varepsilon_0}(P)=P\big(\psi_3(\bx)D^{\T}P\big)$ for $P\in\bbO^{n\times k}$.
%\begin{equation}\label{eq:H(p)-scrH(P)-eps:ThetaTR}
%H_{\varepsilon_0}(P)P-\scrH_{\varepsilon_0}(P)=P\big(\psi_3(\bx)D^{\T}P\big).
%\end{equation}
Thus by \Cref{thm:H(P)-eligibility}, $P\in\bbO^{n\times k}$ is a solution to the KKT condition \eqref{eq:KKT+L21:work-eps-1} with
\eqref{eq:partD-ThetaTR+L21:scrH-eps} if and only if
it is a solution to NEPv \eqref{eq:OptSTM+L21:NEPv-eps} with \eqref{eq:opt-ThetaTR+L21:H(P)-eps}
and $D^{\T}P$ is  symmetric.
A reasonable normalization quantity $\xi$ in \eqref{eq:stop-1} for the stopping criterion is
$$
\xi=\psi_1(\bx)2\|A_1\|_{\F}+\psi_2(\bx)2\|A_2\|_{\F}+\psi_3(\bx)\|D\|_{\F}+n\alpha.
$$

For \Cref{alg:NEPvLOCG}, we notice from \eqref{eq:NEPv-LOCG-2} that now
%\Cref{thm:main-NEPv-cvx+L21:eps} and \Cref{alg:NEPvSCF4+L21} are now readily applicable, and so is
%\Cref{alg:NEPvLOCG}, upon noticing from \eqref{eq:NEPv-LOCG-2} that now
\begin{equation}
\wtd H_{\varepsilon_0}(Z)
   =\psi_1(\bx)2\wtd A_1+\psi_2(\bx)2\wtd A_2+\psi_3(\bx)\big(\wtd DZ^{\T}+Z\wtd D^{\T})
    -\alpha\sum_{i=1}^n\frac {\bw_i\bw_i^{\T}}{\sqrt{\tr(Z^{\T}\bw_i\bw_i^{\T}Z)+\varepsilon_0^2}}, \label{eq:NEPv-LOCG:ThetaTR}
\end{equation}
where $\wtd A_i=W^{\T}A_iW$, $\wtd D=W^{\T}D$, and $\bx=T_0(WZ)$.
As we commented in  subsection~\ref{ssec:MAXBET+L21}, some saving can be achieved from properly implementing
\Cref{alg:NEPvLOCG} as well.

\subsubsection{A workaround for case $1/2<\theta\le 1$}\label{sssec:TTR-2}
Just moments ago, we had to limit $\theta$ to $0\le\theta\le 1/2$ in order to ensure $\psi(\bx)$ in \eqref{eq:obj-ThetaTR}
is convex in $\mathfrak{D}$ of \eqref{eq:frakD-TTR}.
That can be unsatisfactory sometimes. Recall the goal of having the parameter $\theta$ in \eqref{eq:obj-ThetaTR:0}
is to contrast $\tr(P^{\T}A_2P)+\tr(P^{\T}D)$ against $[\tr(P^{\T}A_1P)]^{\theta}$ as $\theta$ varies from $0$ to $1$.
We were able to do that in \cite{wazl:2023} when the $(2,1)$-norm regularization is not present, but with the regularization,
our theory falls short to cover  $1/2<\theta\le 1$.

What if we still need to compare $\tr(P^{\T}A_2P)+\tr(P^{\T}D)$ against $[\tr(P^{\T}A_1P)]^{\theta}$ for $\theta$
varying in $(1/2,1]$? Roughly speaking, what we are trying to do is to compare $A_2$ and $D$ together with some fraction power of $A_1$.
With that in mind, we arrive at the following remedy: instead of $g(P)$ as in \eqref{eq:obj-ThetaTR-1}, we use
\begin{equation}\label{eq:obj-ThetaTR-modified}
\hat g(P)=\left[\frac {\tr(P^{\T}A_2P)+\tr(P^{\T}D)}{[\tr(P^{\T}A_1^2P)]^{\theta/2}}\right]^2,
\end{equation}
for $1/2<\theta\le 1$.
%Roughly speaking, for modelling purpose, we may regard that $\theta$ in \eqref{eq:obj-ThetaTR-1}
%and $\theta/2$ in \eqref{eq:obj-ThetaTR-modified} play the same role. In other words, the objective $\hat g(P)$ of
%\eqref{eq:obj-ThetaTR-modified}
%for $1/4\le\theta\le 1/2$ will do comparable things to the objective $g(P)$ of \eqref{eq:obj-ThetaTR-1} for $1/2\le\theta\le 1$.
Evidently,
\begin{equation}\tag{\ref{eq:obj-ThetaTR-modified}$'$}
\hat g(P)=\hat\psi\circ \what T_0(P)
\quad\mbox{with}\quad
\hat\psi(\bx)=\frac {(x_2+x_3)^2}{x_1^{\theta}},\,\,
\what T_0(P)=\begin{bmatrix}
                             \tr(P^{\T}A_1^2P) \\
                             \tr(P^{\T}A_2P) \\
                             \tr(P^{\T}D)
                           \end{bmatrix}.
\end{equation}
Every statement
we made in subsubsection~\ref{sssec:TTR-1} can be copied verbatim with $A_1$ replaced by $A_1^2$
and $\theta$ by $\theta/2$.

\subsection{LDA with the $(2,1)$-norm regularization}\label{ssec:LDA+L21}
The objective of LDA is a special case of $\Theta$TR \eqref{eq:obj-ThetaTR:0},
or equivalently \eqref{eq:obj-ThetaTR} with $\theta=1$ and $D=0$. It is not clear, if at all possible, how to make a convex
composition of $\tr(P^{\T}A_2P)$ and $\tr(P^{\T}A_1P)$ out of the function in \eqref{eq:obj-ThetaTR:0} for the case  $\theta=1$ and $D=0$ such that maximizing the function
is equivalent to maximizing the composition over the Stiefel manifold. So we return to the suggestion we made
in subsubsection~\ref{sssec:TTR-2}: using, instead,
\begin{equation}\label{eq:obj-LDA-modified}
g(P)=\frac {[\tr(P^{\T}A_2P)]^2}{\tr(P^{\T}A_1^2P)},
\end{equation}
%i.e., letting $\theta=1/2$ and $D=0$ in \eqref{eq:obj-ThetaTR-modified}.
Now we can reuse all formulas in subsubsection~\ref{sssec:TTR-1}
upon setting $\theta=1/2$ and $D=0$ and substituting $A_1^2$ for $A_1$.

\subsection{OCCA with the $(2,1)$-norm regularization}\label{ssec:OCCA+L21}
The workhorse for solving OCCA in \cite{zhwb:2022} is the NEPv approach for the OCCA subproblem whose objective is
exactly $\Theta$TR \eqref{eq:obj-ThetaTR:0} with $\theta=1/2$ and $A_2=0$. To incorporate the $(2,1)$-norm regularization,
we let $g(P)$ be the square of it, i.e.,
\begin{equation}\label{eq:obj-OCCA}
g(P)=\frac {[\tr(P^{\T}D)]^2}{\tr(P^{\T}AP)},
\end{equation}
where we have substitute $A$ for $A_1$.
All formulas in subsubsection~\ref{sssec:TTR-1}, upon minor modifications in notation, can be reused.
% Hence we can use $g(P)$
%in \eqref{eq:obj-ThetaTR} with $\theta=1/2$ and $A_2=0$ and the rest of subsubsection~\ref{sssec:TTR-1}.
More detail can be found in \cite{wazl:2026}.

\subsection{A more general problem}\label{ssec:GL+L21}
In view of \Cref{tbl:AF-NEPv}, in general, our development in \cref{sec:NEPv+L21} works for
$g(P)$ as in \eqref{eq:g-cvx-comp} with
$$
T_0(P)=\begin{bmatrix}
         T_1(P) \\
         T_2(P) \\
         T_3(P)
       \end{bmatrix}\in\bbR^N
$$
where $N=N_1+N_2+1$, and
$$
T_1(P)=\begin{bmatrix}
         \tr((P^{\T}A_1P)^{m_1}) \\
         \vdots \\
         \tr(P^{\T}A_{N_1}P)^{m_{N_1}})
       \end{bmatrix},\,\,
T_2(P)=\begin{bmatrix}
         \tr((P^{\T}B_1P) \\
         \vdots \\
         \tr(P^{\T}B_{N_2}P)
       \end{bmatrix},\,\,
T_3(P)= \tr((P^{\T}D)^m),
$$
where not all of the three $T_i$ have to be present,
%Up to any two of $T_i$ for $1\le i\le 3$ are allowed  not to be present,
and depending on their presences,
it is assumed that integer $m_i>1$ and $\bbR^{n\times n}\ni A_i\succeq 0$ for $1\le i\le N_1$, $B_i\in\bbR^{n\times n}$ for $1\le i\le N_2$
are merely symmetric, and integer $m\ge 1$,
and $\psi(\bx)$ is convex in $\bx\in\mathfrak{D}\subseteq\bbR^N$ with partial derivative $\psi_i(\bx)\ge 0$ for $i$ corresponding to the components
associated with $T_1$ and $T_3$, $\bbP=\bbO_{D_+}^{n\times k}$ if $T_3$ presents but $\bbP=\bbO^{n\times k}$ otherwise.

In the case when all $m_i=2$, the requirement that all $A_i\succeq 0$ can be removed by shifting $A_i$ to transfer the case to one we just considered, as in \cite[Example~8.2]{li:2024}.

\section{Numerical Experiments}\label{sec:egs}
We present numerical results in MATLAB for two problems:
\begin{enumerate}[(1)]
  \item the $(2,1)$-norm regularized MAXBET with $g(P)$ as in \eqref{eq:OptSTM-MAXBET},
  \item the $(2,1)$-norm regularized OCCA with $g(P)$ as in \eqref{eq:obj-OCCA}.
\end{enumerate}
Objectives for both problems are defined by the same two matrices $A\in\bbR^{n\times n}$ and $D\in\bbR^{n\times k}$. In our experiments, we randomly generate and save
\begin{equation}\label{eq:generateBD}
T={\tt randn}(n),\,
A=T{\tt *}T', \,
D=T{\tt *}{\tt randn}(n,k),
\end{equation}
as well as the same initial $P_0={\tt orth}({\tt randn}(n,k))$, for reproducibility,
where $n=1000$ and $k=10$, {\tt randn} is MATLAB's random matrix generator, and {\tt orth} is as before.
%$P_0$ is a preselected initial guess that will be used for all.
It is noted that $A\succeq 0$ in \eqref{eq:generateBD} that is required for the regularized OCCA problem
but not necessary for the regularized MAXBET problem for which
our NEPv approach is still guaranteed to work for symmetric $A$.
Although what follows is about just one set of $A$, $D$, and $P_0$, we have conducted  experiments with numerous different
sets of random
$A$, $D$, and $P_0$ constructed in the same way (with varying $n$ and $k$ as well), albeit the corresponding results are
not reported here due to similarity in behavior.

\begin{figure}[t]
{\centering
\begin{tabular}{ccc}
\resizebox*{0.28\textwidth}{0.15\textheight}{\includegraphics{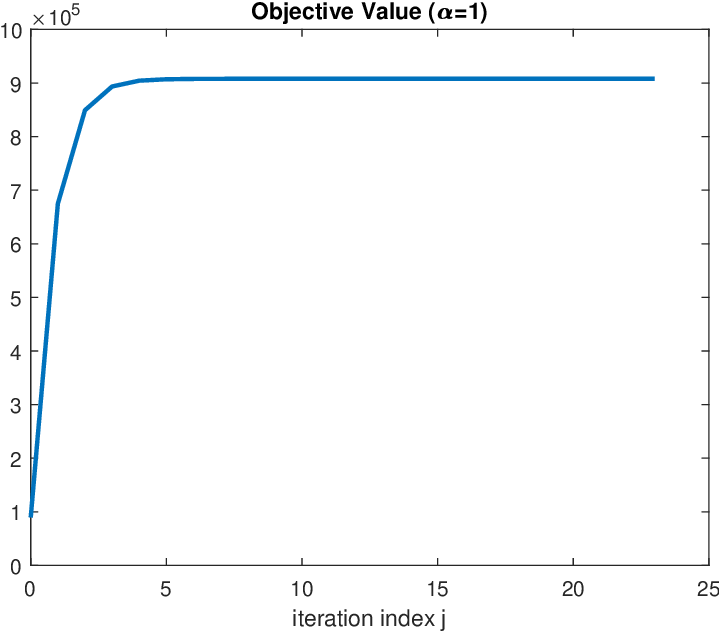}}
  & \resizebox*{0.28\textwidth}{0.15\textheight}{\includegraphics{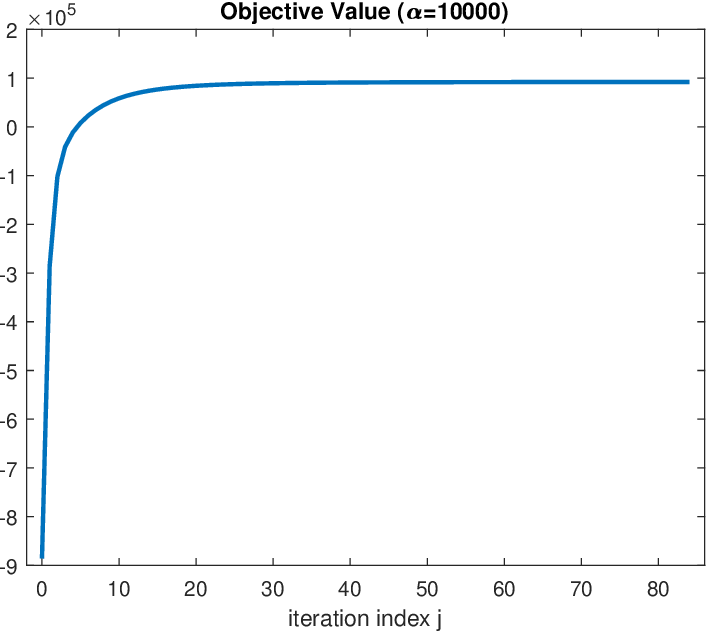}}
  & \resizebox*{0.28\textwidth}{0.15\textheight}{\includegraphics{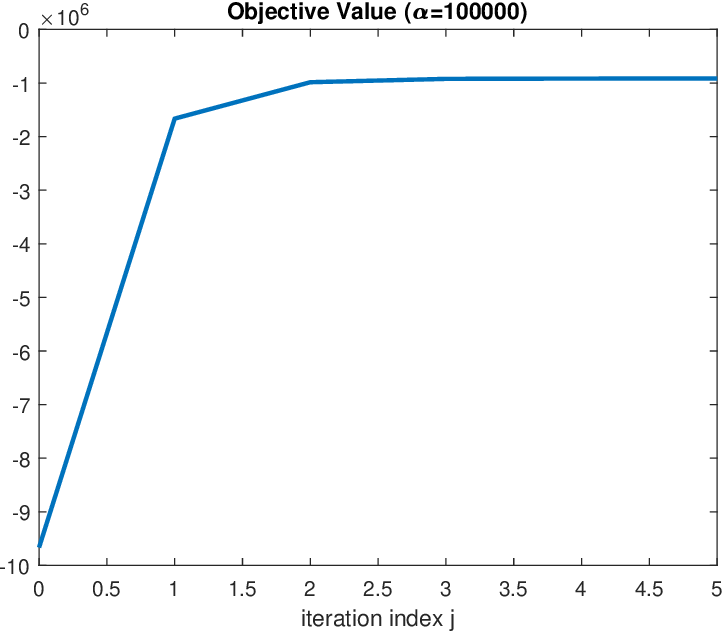}} \\
\resizebox*{0.28\textwidth}{0.15\textheight}{\includegraphics{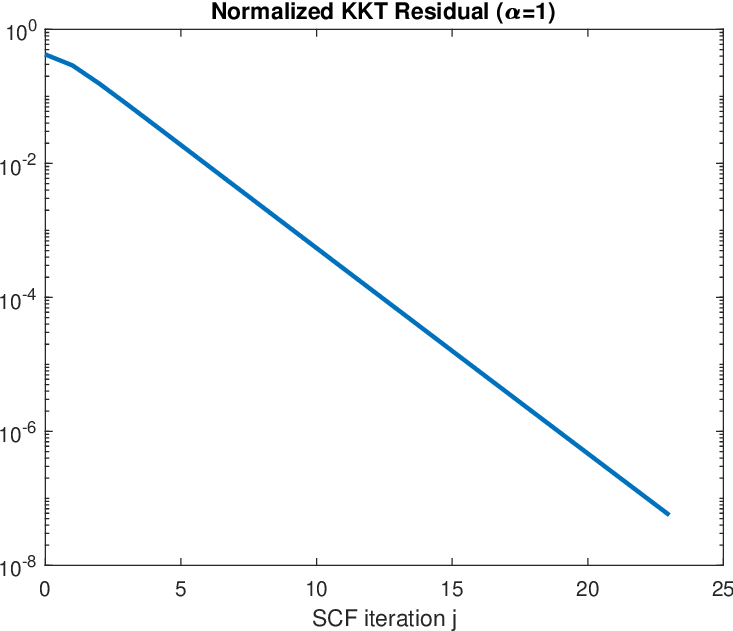}}
  & \resizebox*{0.28\textwidth}{0.15\textheight}{\includegraphics{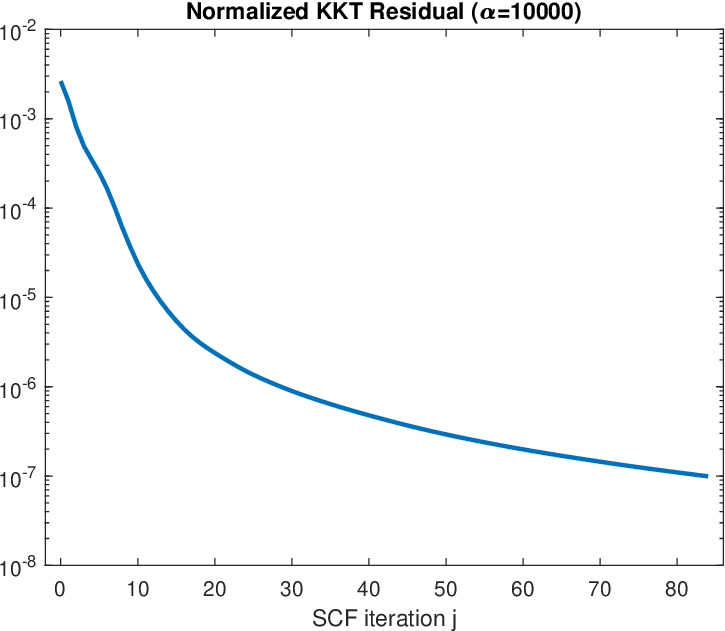}}
  & \resizebox*{0.28\textwidth}{0.15\textheight}{\includegraphics{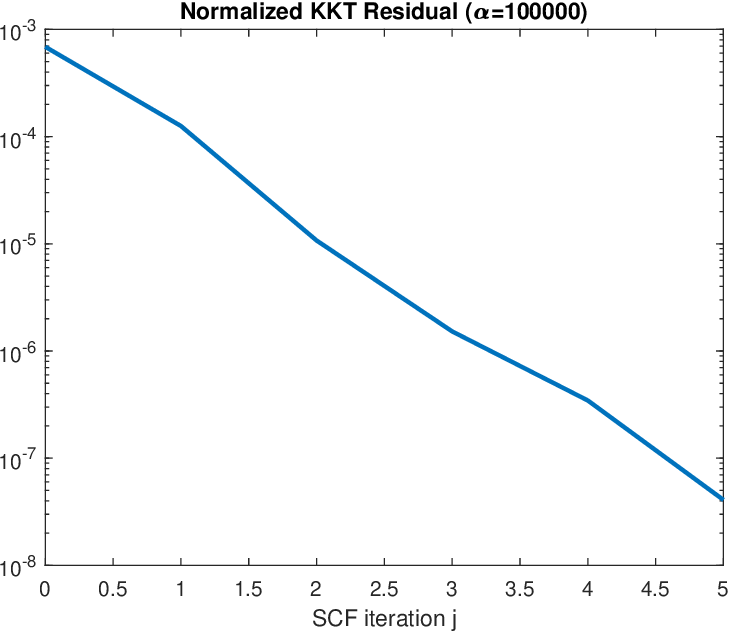}} \\
\resizebox*{0.28\textwidth}{0.15\textheight}{\includegraphics{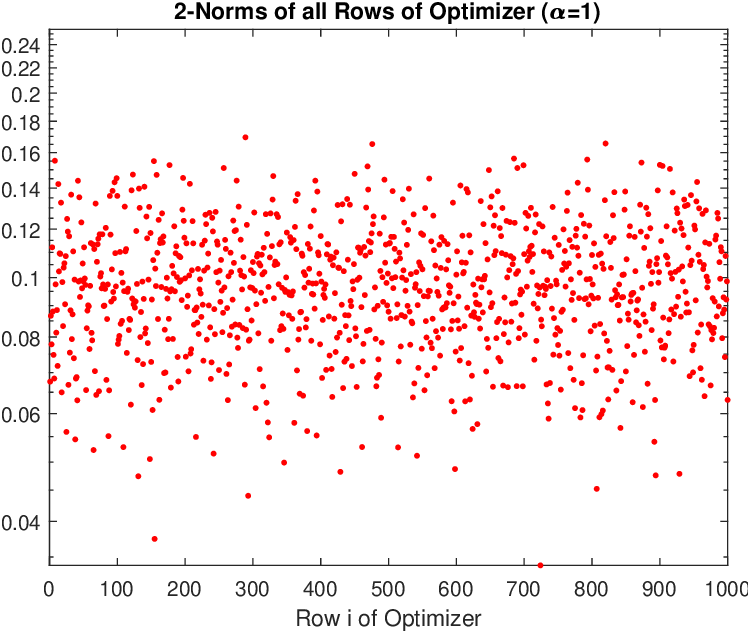}}
  & \resizebox*{0.28\textwidth}{0.15\textheight}{\includegraphics{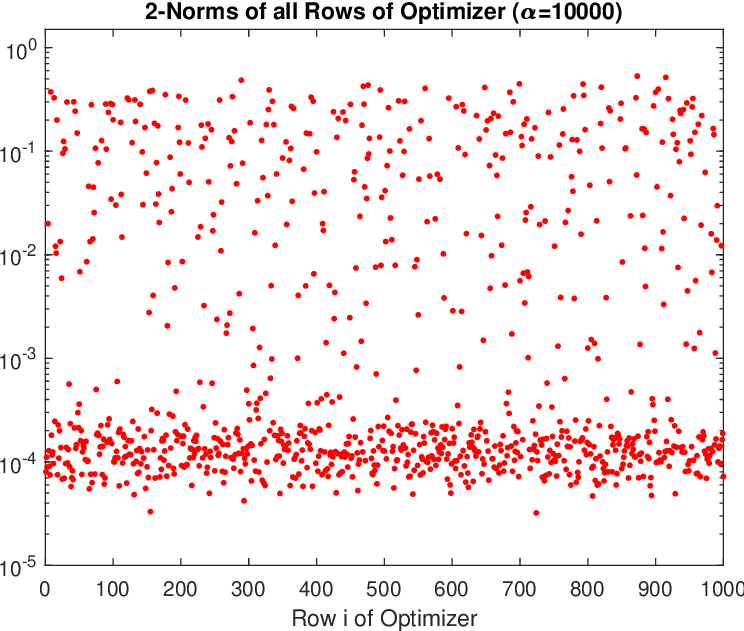}}
  & \resizebox*{0.28\textwidth}{0.15\textheight}{\includegraphics{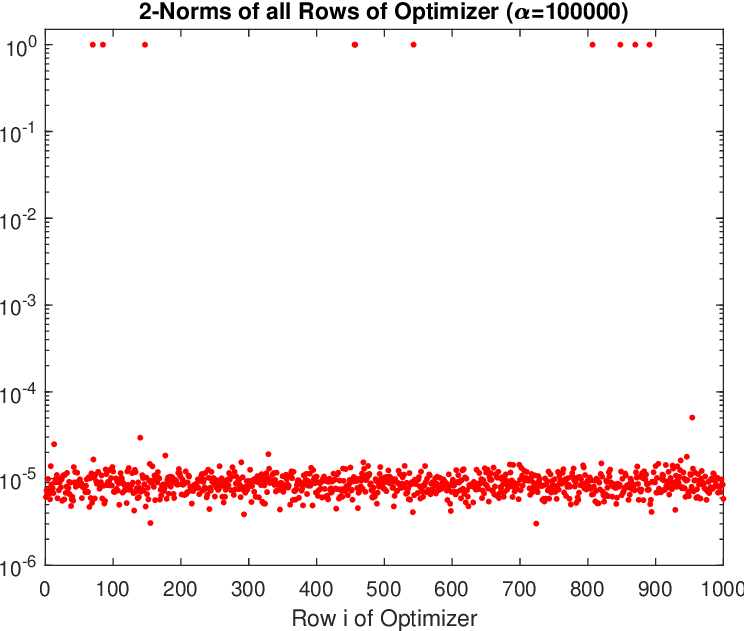}} \\
\end{tabular}\par
}
\vspace{-0.15 cm}
\caption{\small Regularized MAXBET problem with $\alpha=1$ (1st column), $10^4$ (second column), and $10^5$ (third column) by
  SCF  (\Cref{alg:NEPvSCF4+L21} with {\tt eig}). Associated with the last plot, the nontrivial rows of computed maximizer $P$ are its row 70, 85, 147, 456, 457, 543, 807, 848, 870, and 891 with the two collapsing dots for row 456 and 457.
  %The last plot has 10 dots on par with   $10^0$ but the two dots for rows 456 and 457 collapse together.
  }
\label{fig:MAXBET:SCF}
\end{figure}

%\iffalse
\begin{figure}[t]
{\centering
\begin{tabular}{ccc}
\resizebox*{0.28\textwidth}{0.15\textheight}{\includegraphics{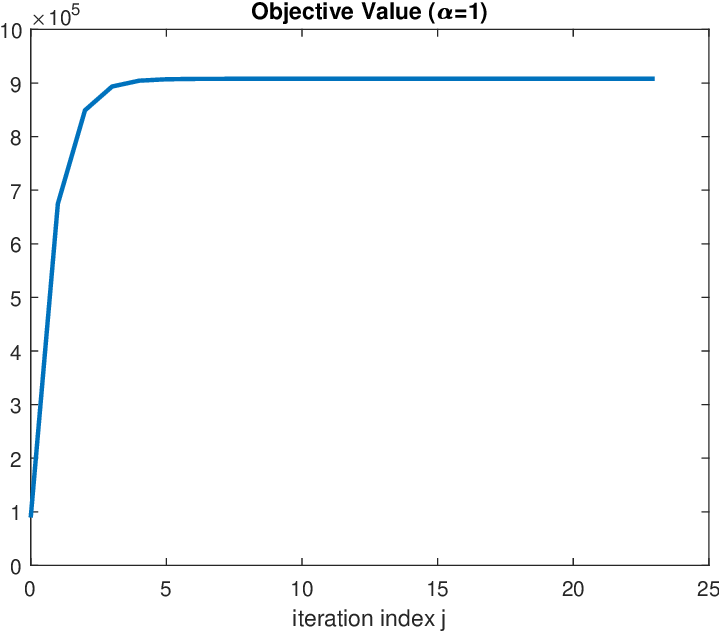}}
  & \resizebox*{0.28\textwidth}{0.15\textheight}{\includegraphics{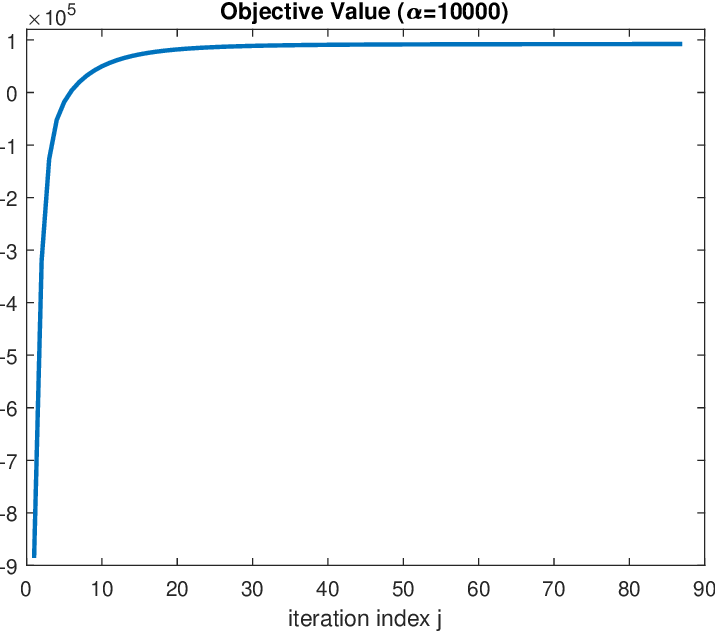}}
  & \resizebox*{0.28\textwidth}{0.15\textheight}{\includegraphics{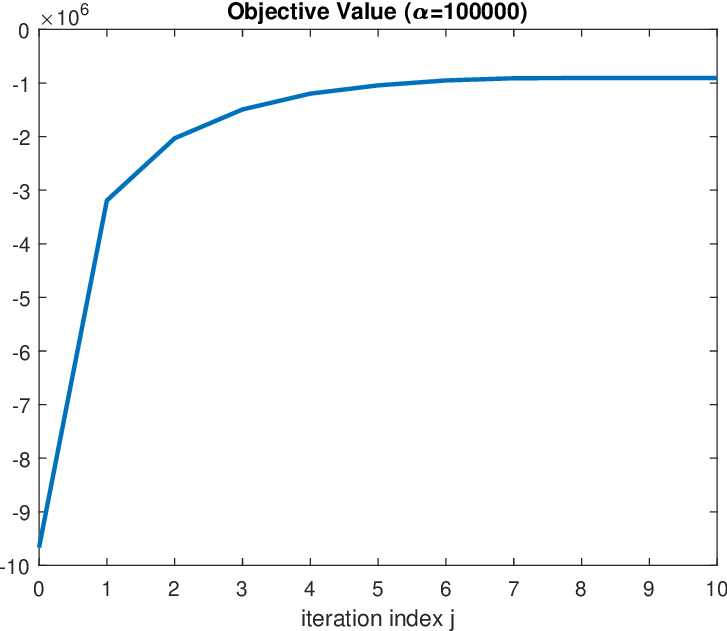}} \\
\resizebox*{0.28\textwidth}{0.15\textheight}{\includegraphics{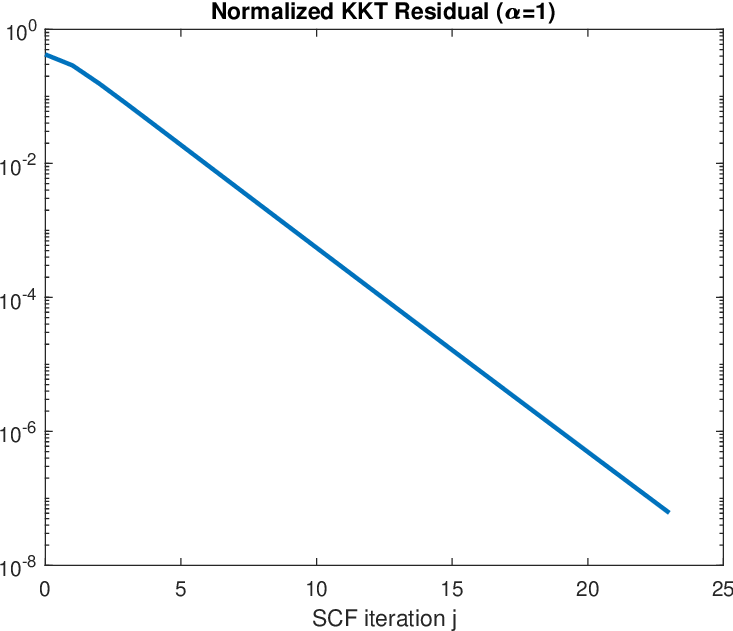}}
  & \resizebox*{0.28\textwidth}{0.15\textheight}{\includegraphics{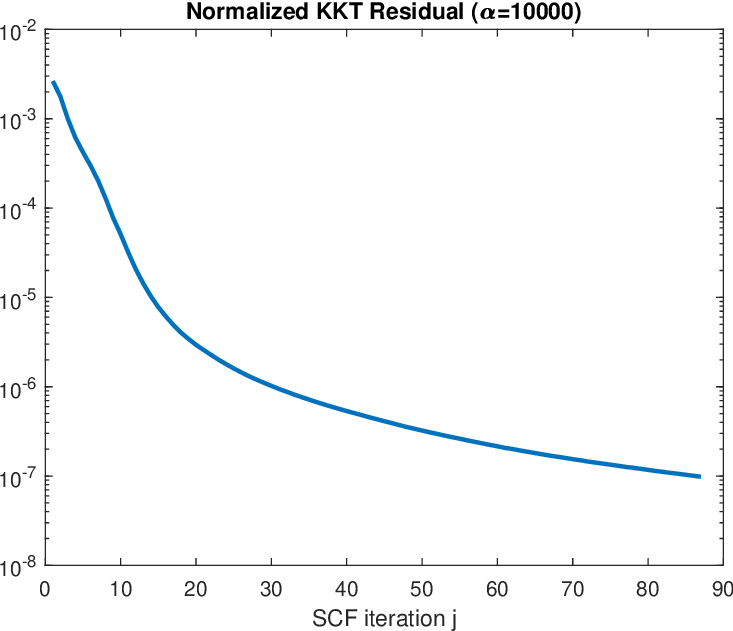}}
  & \resizebox*{0.28\textwidth}{0.15\textheight}{\includegraphics{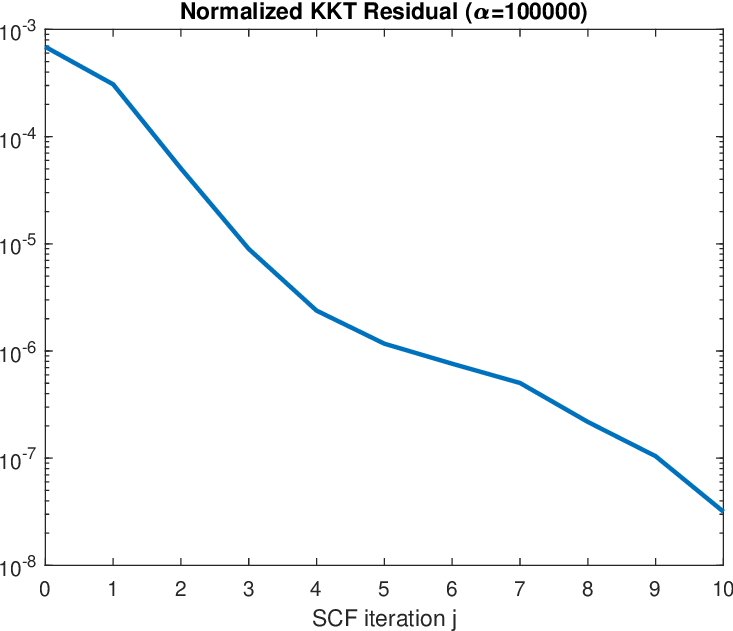}} \\
\resizebox*{0.28\textwidth}{0.15\textheight}{\includegraphics{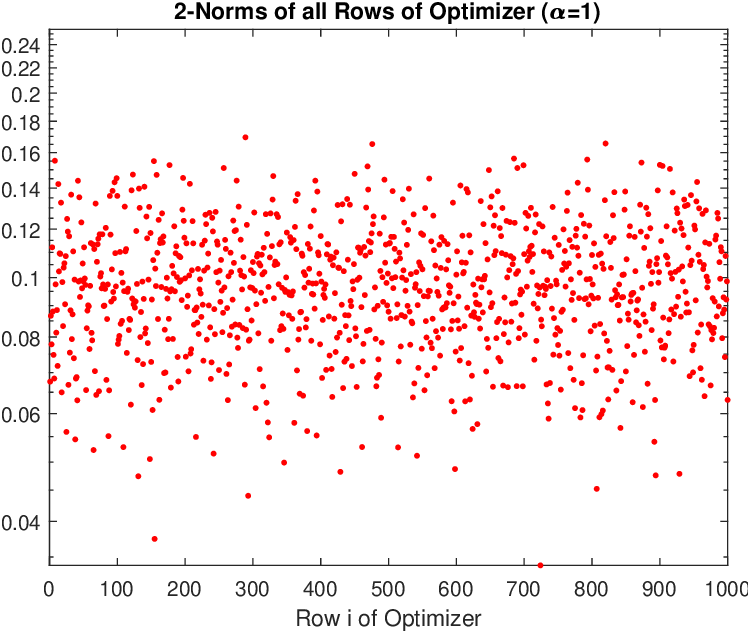}}
  & \resizebox*{0.28\textwidth}{0.15\textheight}{\includegraphics{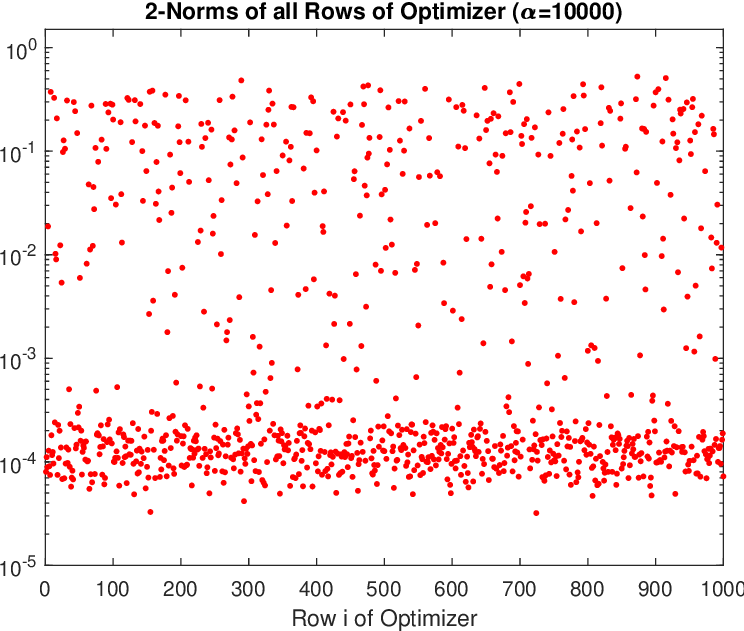}}
  & \resizebox*{0.28\textwidth}{0.15\textheight}{\includegraphics{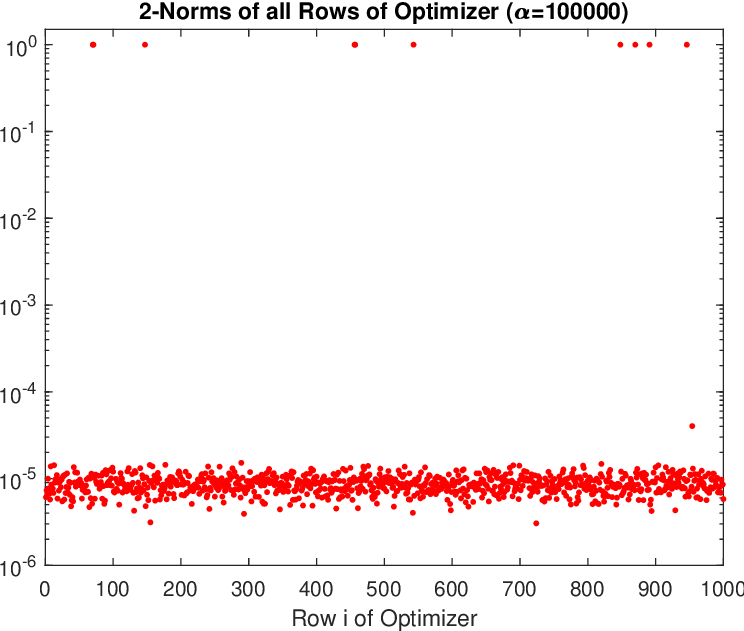}} \\
\end{tabular}\par
}
\vspace{-0.15 cm}
\caption{\small Regularized MAXBET problem with $\alpha=1$ (1st column), $10^4$ (second column), and $10^5$ (third column) by
  SCF  (\Cref{alg:NEPvSCF4+L21} with LOBPCG but without any preconditioning). Associated with the last plot, the nontrivial rows
  of computed maximizer $P$ are its row 8, 154, 289, 476, 685, 699, 793, 820, 873, and 904.
%  It can be observed that the objective (first row) monotonically increases
%  and the normalized KKT residual $\varepsilon_{\KKT}$ as defined in \eqref{eq:stop-1} goes towards $0$.
%  Smaller regularizing parameter $\alpha$ does not induce sparse rows. In fact,
%  row sparsity begins to show at $\alpha=10^4$ but completely emerges at $10^5$. It seems that there are just 9 dots on par with
%  $10^0$ in the last plot, but there are actually 10 dots due to that the two dots for rows 456 and 457 are visually inseparable.
  }
\label{fig:MAXBET:SCF-LOBPCG}
\end{figure}
%\fi

\begin{figure}[t]
{\centering
\begin{tabular}{ccc}
\resizebox*{0.28\textwidth}{0.15\textheight}{\includegraphics{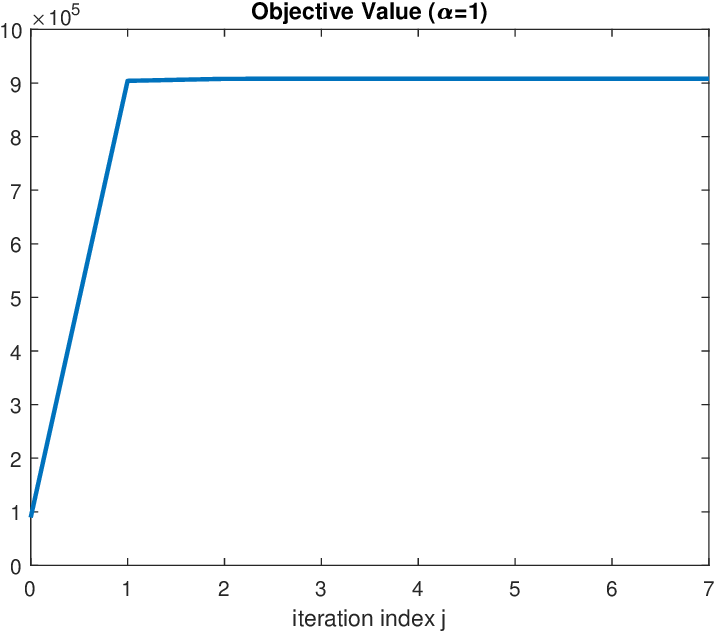}}
  & \resizebox*{0.28\textwidth}{0.15\textheight}{\includegraphics{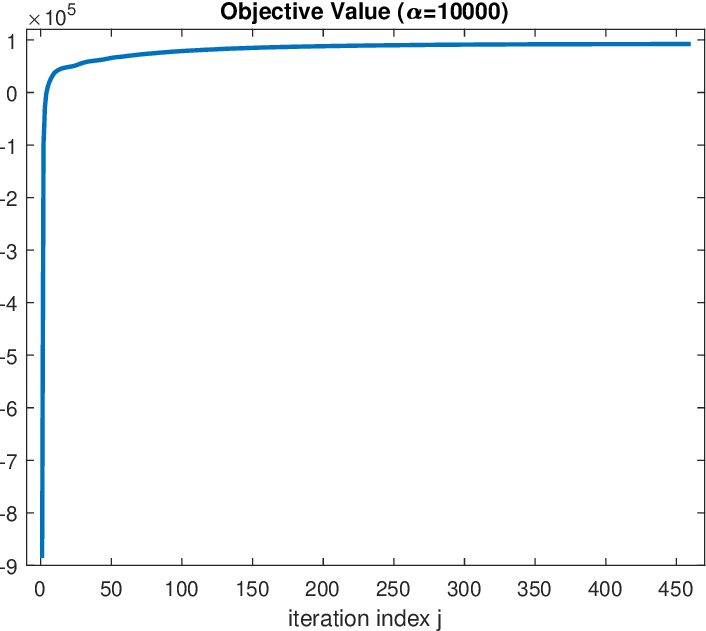}}
  & \resizebox*{0.28\textwidth}{0.15\textheight}{\includegraphics{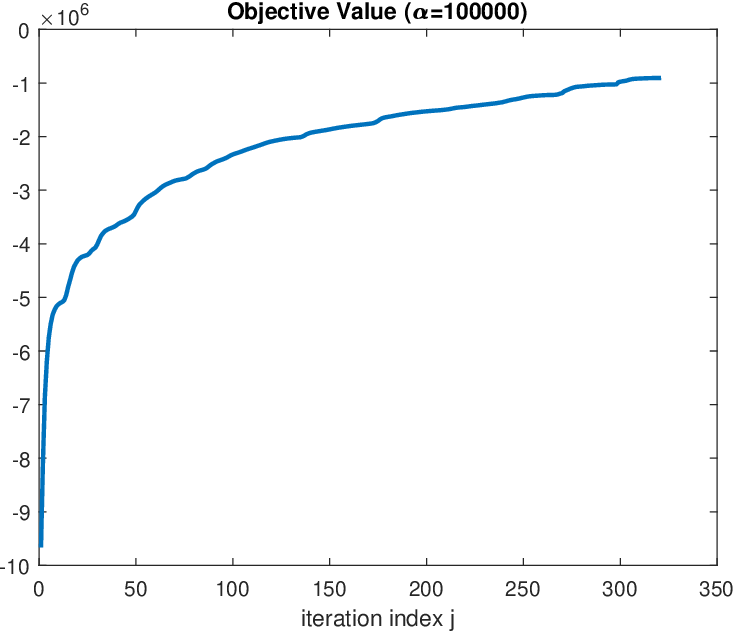}} \\
\resizebox*{0.28\textwidth}{0.15\textheight}{\includegraphics{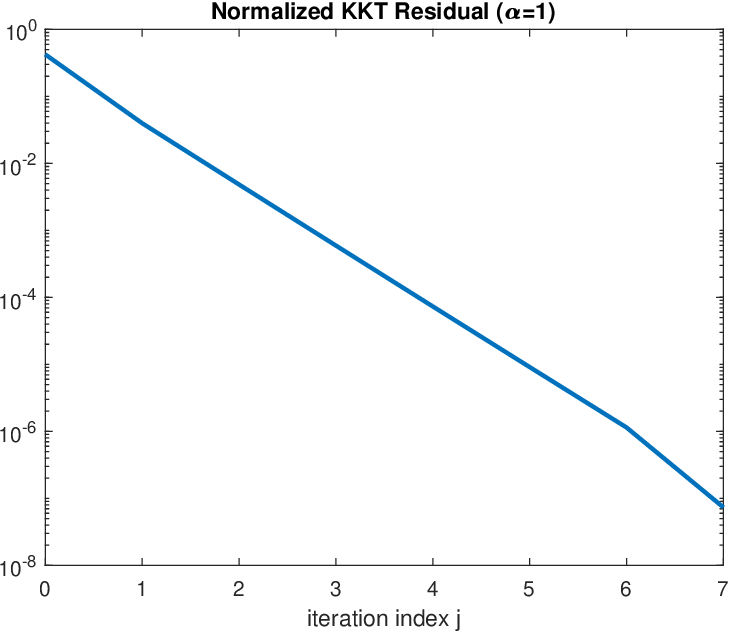}}
  & \resizebox*{0.28\textwidth}{0.15\textheight}{\includegraphics{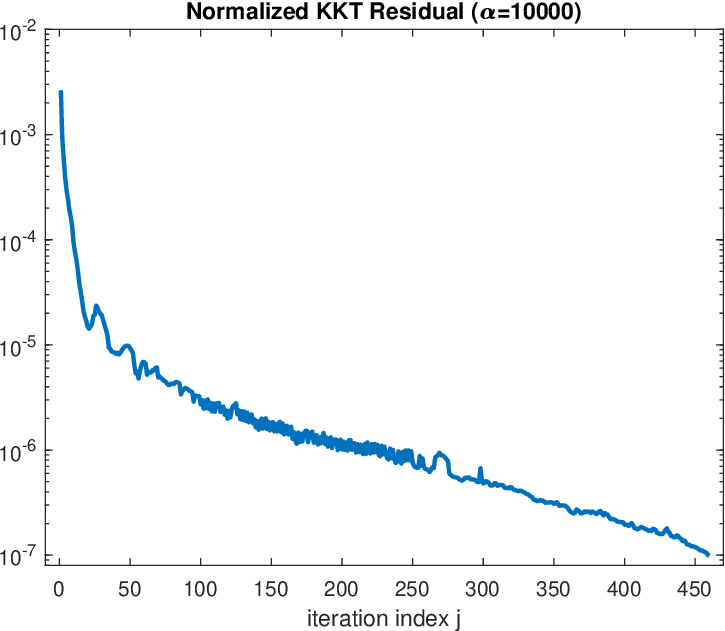}}
  & \resizebox*{0.28\textwidth}{0.15\textheight}{\includegraphics{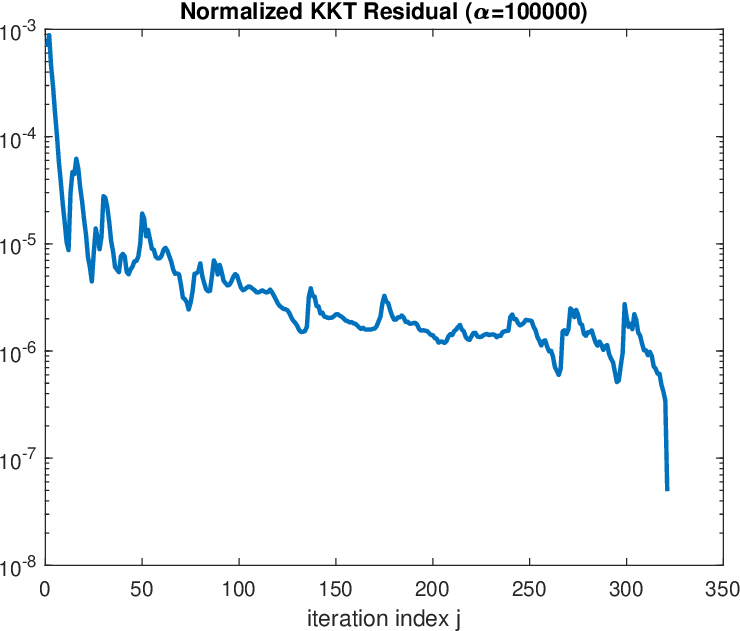}} \\
\resizebox*{0.28\textwidth}{0.15\textheight}{\includegraphics{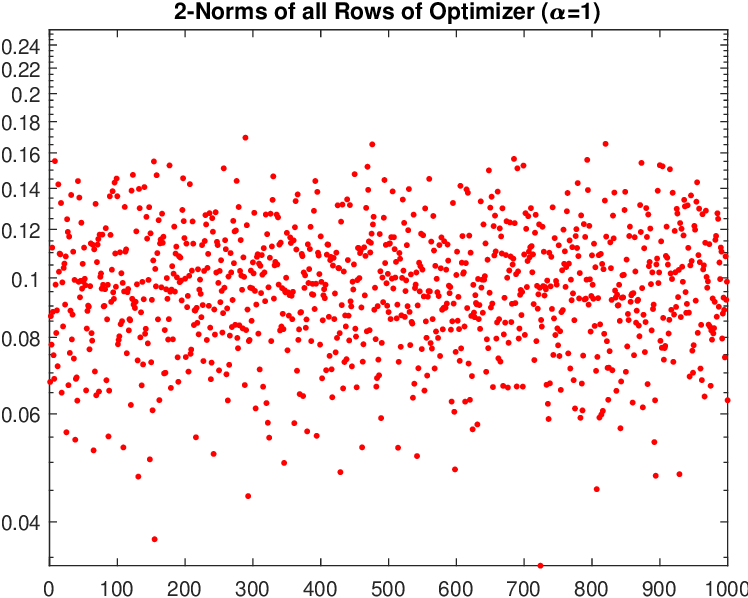}}
  & \resizebox*{0.28\textwidth}{0.15\textheight}{\includegraphics{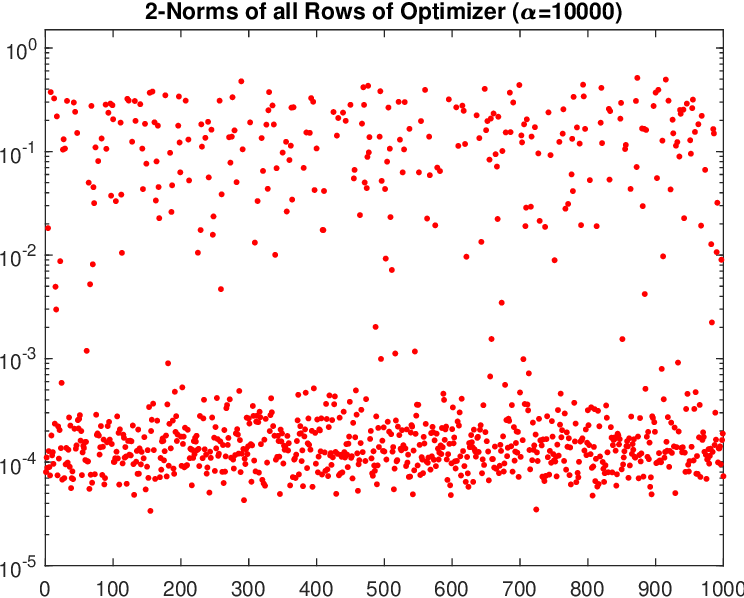}}
  & \resizebox*{0.28\textwidth}{0.15\textheight}{\includegraphics{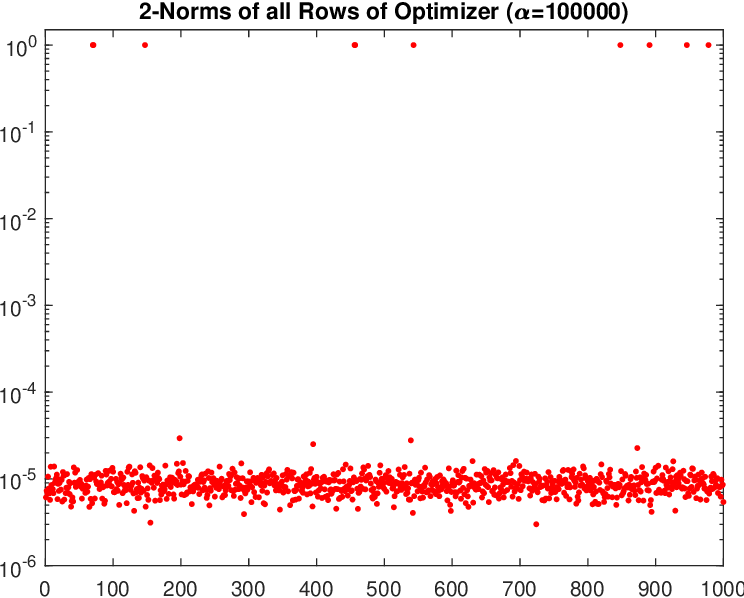}} \\
\end{tabular}\par
}
\vspace{-0.15 cm}
\caption{\small Regularized MAXBET problem with $\alpha=1$ (1st column), $10^4$ (second column), and $10^5$ (third column) by
  SCF with acceleration (\Cref{alg:NEPvLOCG}). Associated with the last plot, the nontrivial rows of computed maximizer $P$ are its row 70, 71, 147, 456, 457, 543, 848, 891, 946, and 978.
  %The last plot has 10 dots on par with   $10^0$ but the two dots for rows 70 and 71 and the ones for rows 456 and 457 collapse together.
%  It can be observed that the objective (durst row) monotonically increases
%  and the normalized KKT residual $\varepsilon_{\KKT}$ as defined in \eqref{eq:stop-1} goes towards $0$.
%  Smaller regularizing parameter $\alpha$ does not induce sparse rows. In fact,
%  row sparsity begins to show at $\alpha=10^4$ but completely emerges at $10^5$. It seems that there are just 9 dots on par with
%  $10^0$ in the last plot, but there are actually 10 dots due to that the two dots for rows 456 and 457 are visually inseparable.
  }
\label{fig:MAXBET:SCFacc}
\end{figure}

We will test three different methods to solve the two problems:
\begin{enumerate}[(a)]
  \item \Cref{alg:NEPvSCF4+L21} with MATLAB's {\tt eig} to solve the SEP at its line 3,
  \item \Cref{alg:NEPvSCF4+L21} with LOBPCG (without preconditioning) to solve the SEP there, and
  \item \Cref{alg:NEPvLOCG} (the acceleration version) which calls \Cref{alg:NEPvSCF4+L21} with {\tt eig} to solve each reduced
problem~\eqref{eq:LOCGsub}.
\end{enumerate}
%1) \Cref{alg:NEPvSCF4+L21} with MATLAB's {\tt eig} to solve the SEP at its line 3,
%2) \Cref{alg:NEPvSCF4+L21} with LOBPCG (without preconditioning) to solve the SEP there, and
%3) \Cref{alg:NEPvLOCG} (the acceleration version) which calls \Cref{alg:NEPvSCF4+L21} with {\tt eig} to solve each reduced
%problem~\eqref{eq:LOCGsub}.
For each of the two problems with the saved set of $A$, $D$, and $P_0$, we will report numerical results on three different regularizing parameters $\alpha$ carefully selected to demonstrate: 1) no row sparsity, 2) emergence of row sparsity, and
3) complete row sparsity (exactly $k$ nontrivial rows in an optimizer),
as $\alpha$ increases to certain level.

In using LOBPCG, each time it is iterated to an accuracy that is fractionally better than
the current approximation $P^{(j)}$ to KKT \eqref{eq:KKT+L21:work-eps}. Specifically,
the SEP at Line 3 of \Cref{alg:NEPvSCF4+L21} is solved by LOBPCG to satisfy
\begin{equation}\label{eq:stop-LOBPCG}
\big\|H^{(j)}\what P^{(j)}-\what P^{(j)}\big(\big[\what P^{(j)}\big]^{\T}H^{(j)}\what P^{(j)}\big)\big\|_{\F}
     \le \frac 18\cdot\big\|\scrH_{\varepsilon_0}(P^{(j)})-P\Lambda_{\varepsilon_0}(P^{(j)})\big\|_{\F},
\end{equation}
where $\Lambda_{\varepsilon_0}(\cdot)$ is the same one as in \eqref{eq:stop-1}.
Ideally, $\what P^{(j)}$  should be computed  accurately just enough so that the
eventual $P^{(j+1)}$ is so good as if the SEP is solved
fully accurately in the working precision. But in general, that is hard to do, if at all possible,
because there is no simple way to find out what the just right accuracy in computed $\what P^{(j)}$ should have
for achieving the ideal outcome. Nonetheless,
\eqref{eq:stop-LOBPCG} seems to strike a good balance between accuracy and effort and it works well
in our experiments.
Similarly, in our implementation of \Cref{alg:NEPvLOCG}, each reduced
problem~\eqref{eq:LOCGsub}  is solved by \Cref{alg:NEPvSCF4+L21} with {\tt eig} just fractionally better
than the current $P^{(j)}$ as an approximate solution of $\OptSM_{2,1}^{\varepsilon}$~\eqref{eq:OptSTM+L21:eps}, as we commented in
\Cref{rk:SCF4npd+LOCG}(iv).

%By default, MATLAB's {\tt eig} is used for solving the SEP at Line 3 of \Cref{alg:NEPvSCF4+L21}.
%We also tested LOBPCG \cite{knya:2001} (without any preconditioning however) to solve the SEP approximately for that purpose
%but did not see much algorithmic behavior difference, except faster in speed.

\Cref{fig:MAXBET:SCF} shows the numerical results for the regularized MAXBET problem for
$\alpha=1,\, 10^4,\, 10^5$ by \Cref{alg:NEPvSCF4+L21} with {\tt eig}.
It can be observed that the objective (first row) monotonically increases
and the normalized KKT residual $\varepsilon_{\KKT}$ as defined in \eqref{eq:stop-1} goes towards $0$.
At $\alpha=1$, there is no visible row sparsity among the rows, but at $\alpha=10^4$, row sparsity begins to emerge
as the norms of quite a number of rows have made their ways to around $\varepsilon_0=10^{-3}\sqrt{k/n}=10^{-4}$, and finally
there are only $k=10$ rows standing out with norms approximately $1$.
It is noted that \Cref{alg:NEPvSCF4+L21} has an easy time for small and sufficiently large $\alpha$, but works hard for the intermediate $\alpha$. It is not clear if the behavior for $\alpha$ varying from small up to certain point can be generalized but the behavior for sufficiently large $\alpha$ is generalizable. The reason hides in the analysis
in \cref{sec:EffL21}. That is, for $\alpha$ sufficiently large, that $g(P)$ in both $\OptSM_{2,1}$~\eqref{eq:OptSTM+L21} and
$\OptSM_{2,1}^{\varepsilon}$~\eqref{eq:OptSTM+L21:eps} becomes negligible. As a consequence, effectively the associated optimization problem,
becomes the one for slightly perturbed $\|P\|_{2,1}$,
i.e.,
$$
\sum_{i=1}^n\sqrt{\tr(P^{\T}\be_i\be_i^{\T}P)+\varepsilon_0^2}.
$$
For this objective, by \Cref{thm:extreme-STM-pt:work-eps},
there are essentially $\binom nk$ equivalent classes of KKT points each of which has $k$ nontrivial rows and $n-k$ rows exactly $0$.
Therefore, it is rather easy for any iteration scheme to get quickly trapped into somewhere sufficiently near one of those KKT points for
$\alpha$ sufficiently large. Unfortunately, it is not good news for feature selection
because then selected features are essentially
obtained without input from the criterion as defined by $g(P)$ and highly sensitive with respect to many factors such as
different initial guesses, and even rounding errors during computations.

\setlength{\tabcolsep}{4pt}
\renewcommand{\arraystretch}{1.15}
\begin{table}
\caption{Performance statistics for regularized MAXBET}\label{tbl:MAXBET21}
\centerline{\scriptsize
\begin{tabular}{|c|c|c|c|c|c|c|c|c|c|}
  \hline
  & \multicolumn{3}{c|}{\Cref{alg:NEPvSCF4+L21} (with {\tt eig})}
  & \multicolumn{3}{c|}{\Cref{alg:NEPvSCF4+L21} (with LOBPCG)}
  & \multicolumn{3}{c|}{\Cref{alg:NEPvLOCG}} \\ \hline
$\alpha$ & $1$ & $10^4$ & $10^5$
         & $1$ & $10^4$ & $10^5$
         & $1$ & $10^4$ & $10^5$ \\ \hline
Obj. & $9.1\cdot 10^5$  & $9.2\cdot 10^4$ & $-9.2\cdot 10^5$
     & $9.1\cdot 10^5$  & $9.2\cdot 10^4$ & $-9.1\cdot 10^5$
     & $9.1\cdot 10^5$  & $9.2\cdot 10^4$ & $-9.0\cdot 10^5$ \\ \hline
CPU & $9.4\cdot 10^{-1}$ & $3.3$ & $2.3\cdot 10^{-1}$
    & $1.2\cdot 10^{-1}$ & $2.0$  & $2.9\cdot 10^{-1}$
    & $2.4\cdot 10^{-2}$ & $7.6\cdot 10^{-1}$ & $5.7\cdot 10^{-1}$\\ \hline
$\varepsilon_{\KKT}$ & $5.6\cdot 10^{-8}$ & $9.9\cdot 10^{-8}$ & $4.1\cdot 10^{-8}$
                     & $6.1\cdot 10^{-8}$ & $9.8\cdot 10^{-8}$ & $3.2\cdot 10^{-8}$
                     & $7.3\cdot 10^{-8}$ & $1.0\cdot 10^{-7}$ & $4.8\cdot 10^{-8}$ \\ \hline
it'ns & 23 & 84 & 5 & 23 & 86 & 10 & 7 & 459 & 320 \\
\hline
\end{tabular}
}
\end{table}

\Cref{fig:MAXBET:SCF-LOBPCG,fig:MAXBET:SCFacc} show the numerical results for the same problem as \Cref{fig:MAXBET:SCF},
except that \Cref{fig:MAXBET:SCF-LOBPCG} is by \Cref{alg:NEPvSCF4+L21}
with LOBPCG \cite{knya:2001} (without any preconditioning however) while
\Cref{fig:MAXBET:SCFacc} is by \Cref{alg:NEPvLOCG}, the accelerated version.
It is noted that for $\alpha=1$ \Cref{alg:NEPvLOCG}
indeed accelerates over \Cref{alg:NEPvSCF4+L21}: \Cref{alg:NEPvLOCG} takes $7$ outer iterations
whereas \Cref{alg:NEPvSCF4+L21} takes 23. Since each inner problem of \Cref{alg:NEPvLOCG} is of size $3k$,
\Cref{alg:NEPvLOCG} runs significantly faster (see \Cref{tbl:MAXBET21}), but
the normalized KKT residual $\varepsilon_{\KKT}$ from \Cref{alg:NEPvLOCG} for $\alpha=10^4,\, 10^5$ wobbles towards $0$.
Likely, this may be caused by the degradation in smoothness of regularized $g(P)$ for very large $\alpha$.
We emphasizes that the nontrivial rows of optimizers revealed in the last plots of
\Cref{fig:MAXBET:SCF,fig:MAXBET:SCF-LOBPCG,fig:MAXBET:SCFacc} are different, highlighting the point we made moments ago
on somewhat ``randomness in feature selection'' with too large $\alpha$.
In light of this, it is our recommendation not to use regularizing parameters $\alpha$ that is so large
for achieving complete row sparsity but rather one that leads to an emergence of row sparsity and then select $1.5k$ up to $2k$ features for any following up data science tasks.

%In fact, it is noted that sets of the row indices for the nontrivial rows associated with the last plots in
%\Cref{fig:MAXBET:SCF,fig:MAXBET:SCF-LOBPCG,fig:MAXBET:SCFacc} are all different.

%%%%%%%%%%%%%%%%%%%%%%%%%%%%%%%%%%%%%%%%%%%%%%%%%%%%%%%%%%%%%%%%%%%%%%%%%%%%%%%%%%%%%%%%%%%%%%%%%%%%%%%%%%%%%%%%%%%%%%%

\begin{figure}[t]
{\centering
\begin{tabular}{ccc}
\resizebox*{0.28\textwidth}{0.15\textheight}{\includegraphics{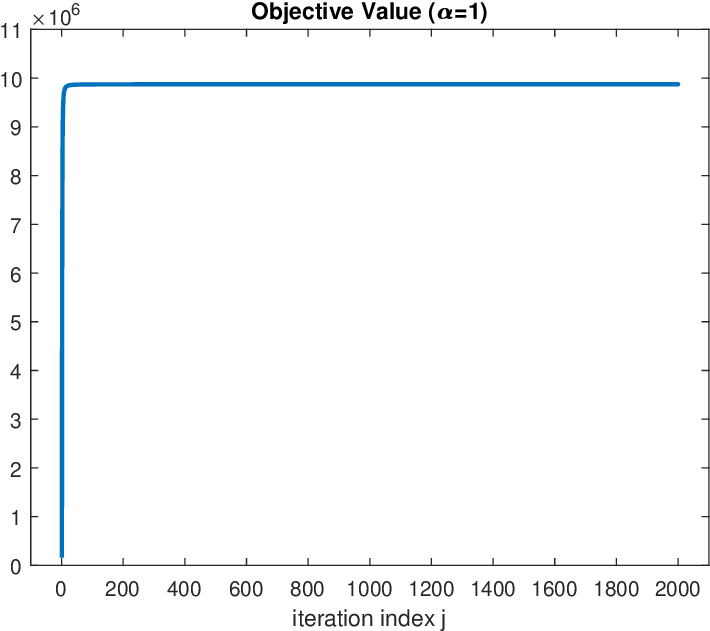}}
  & \resizebox*{0.28\textwidth}{0.15\textheight}{\includegraphics{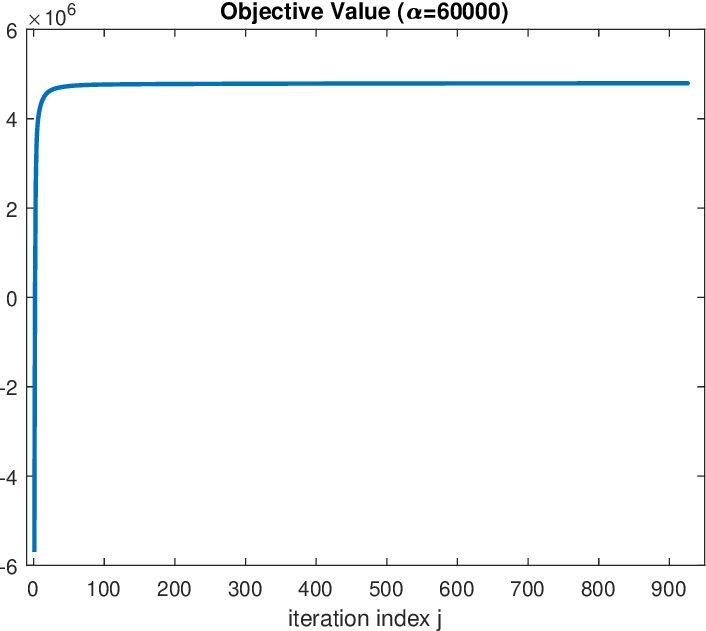}}
  & \resizebox*{0.28\textwidth}{0.15\textheight}{\includegraphics{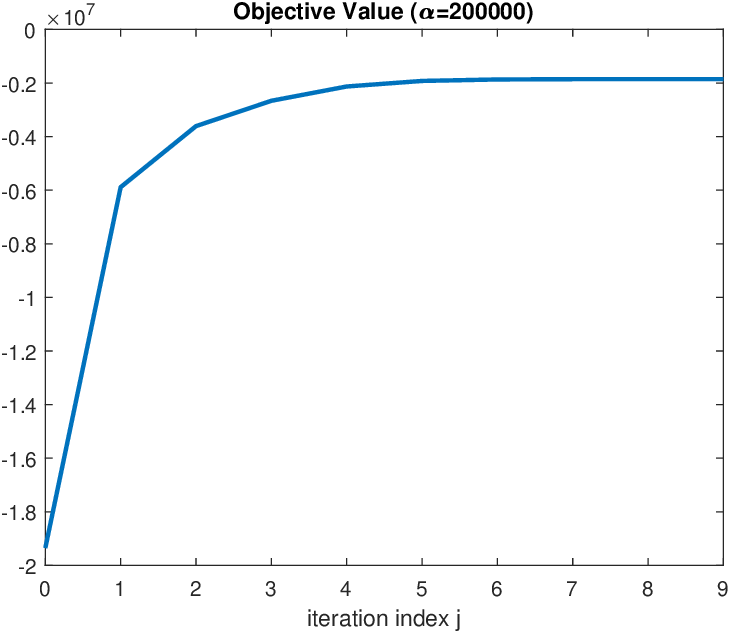}} \\
\resizebox*{0.28\textwidth}{0.15\textheight}{\includegraphics{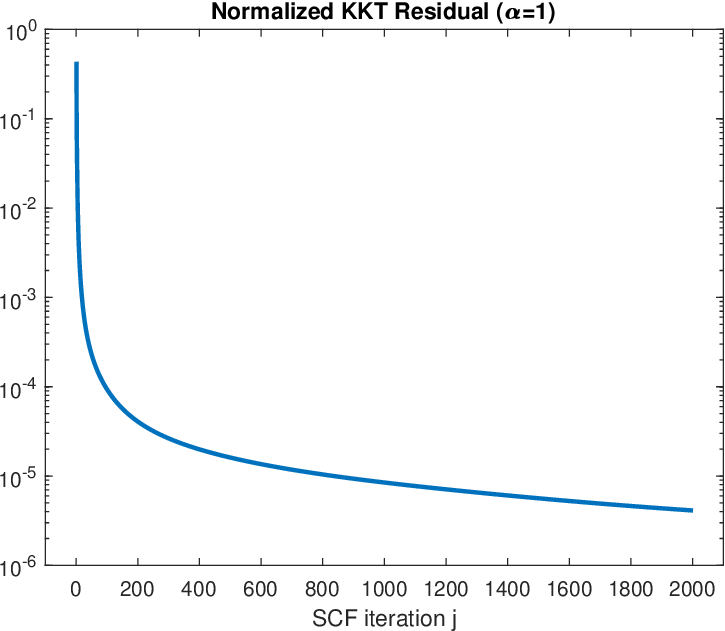}}
  & \resizebox*{0.28\textwidth}{0.15\textheight}{\includegraphics{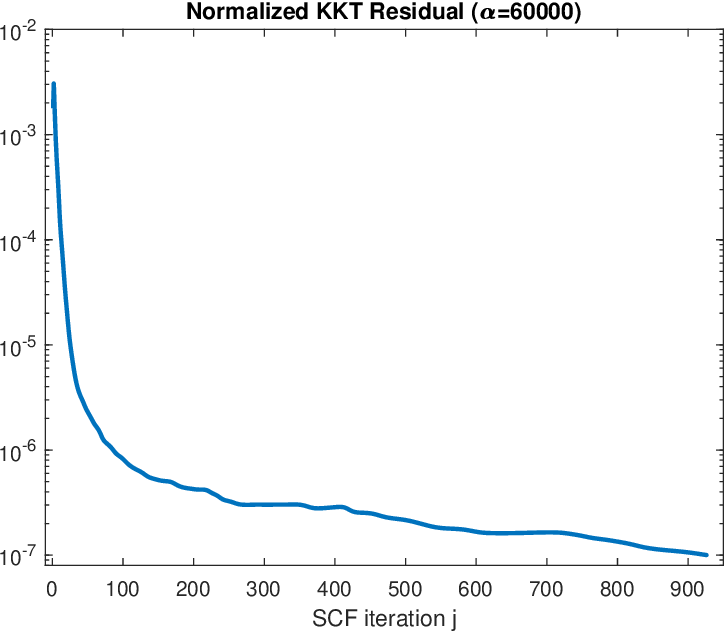}}
  & \resizebox*{0.28\textwidth}{0.15\textheight}{\includegraphics{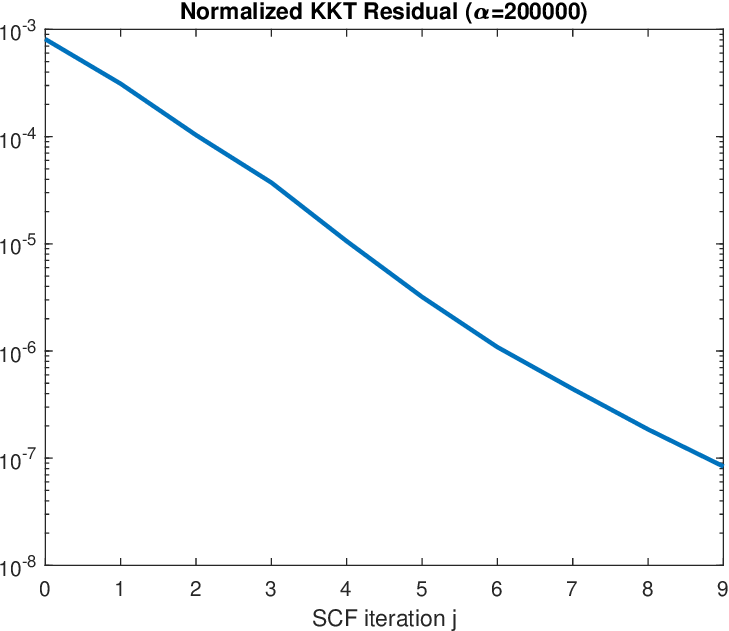}} \\
\resizebox*{0.28\textwidth}{0.15\textheight}{\includegraphics{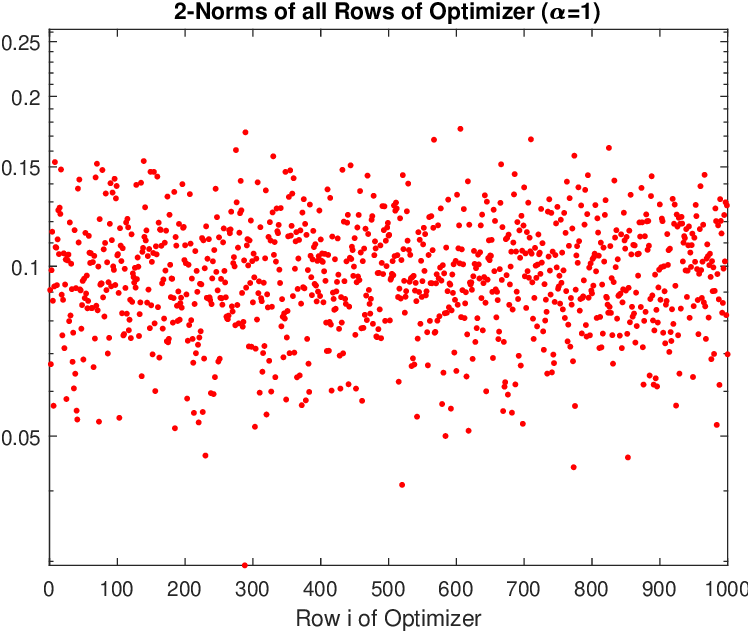}}
  & \resizebox*{0.28\textwidth}{0.15\textheight}{\includegraphics{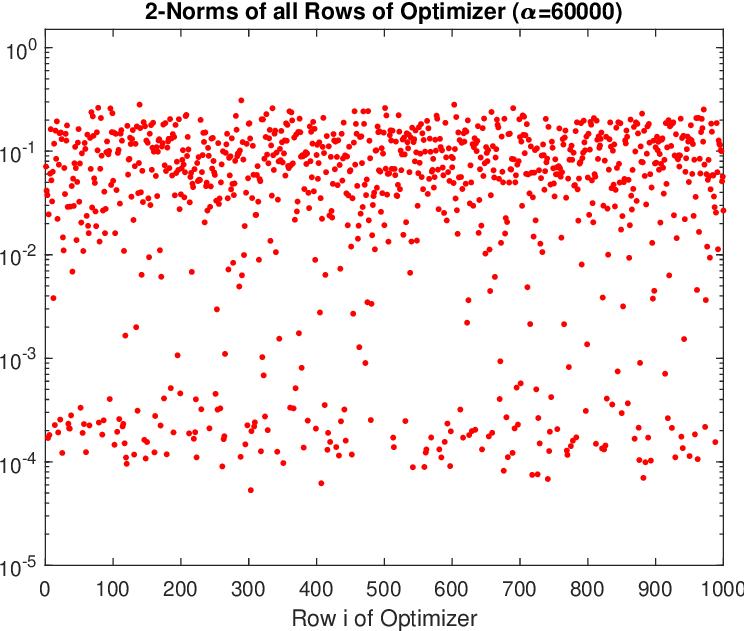}}
  & \resizebox*{0.28\textwidth}{0.15\textheight}{\includegraphics{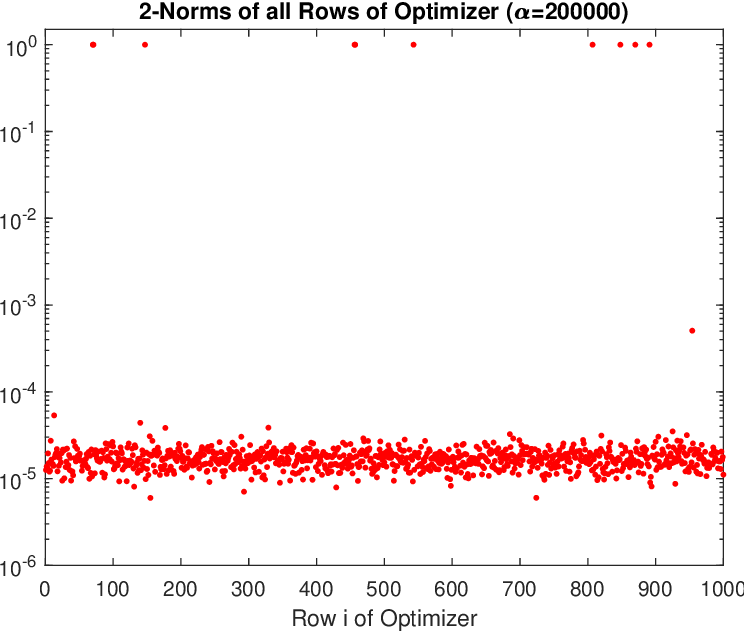}} \\
\end{tabular}\par
}
\vspace{-0.15 cm}
\caption{\small Regularized OCCA problem with $\alpha=1$ (1st column), $6\cdot 10^4$ (second column), and $2\cdot 10^5$ (third column) by
  SCF  (\Cref{alg:NEPvSCF4+L21} with {\tt eig}). Associated with the last plot, the nontrivial rows of computed maximizer $P$ are its row 70, 71, 147, 456, 457, 543, 807, 848, 870, and 891.
  %The last plot has 10 dots on par with   $10^0$ but the two dots for rows 70 and 71 and the ones for rows 456 and 457 collapse together.
%  It can be observed that the objective (durst row) monotonically increases
%  and the normalized KKT residual $\varepsilon_{\KKT}$ as defined in \eqref{eq:stop-1} goes towards $0$.
%  Smaller regularizing parameter $\alpha$ does not induce sparse rows. In fact,
%  row sparsity begins to show at $\alpha=10^4$ but completely emerges at $10^5$. It seems that there are just 9 dots on par with
%  $10^0$ in the last plot, but there are actually 10 dots due to that the two dots for rows 456 and 457 are visually inseparable.
  }
\label{fig:OCCA:SCF}
\end{figure}

%\iffalse
\begin{figure}[t]
{\centering
\begin{tabular}{ccc}
\resizebox*{0.28\textwidth}{0.15\textheight}{\includegraphics{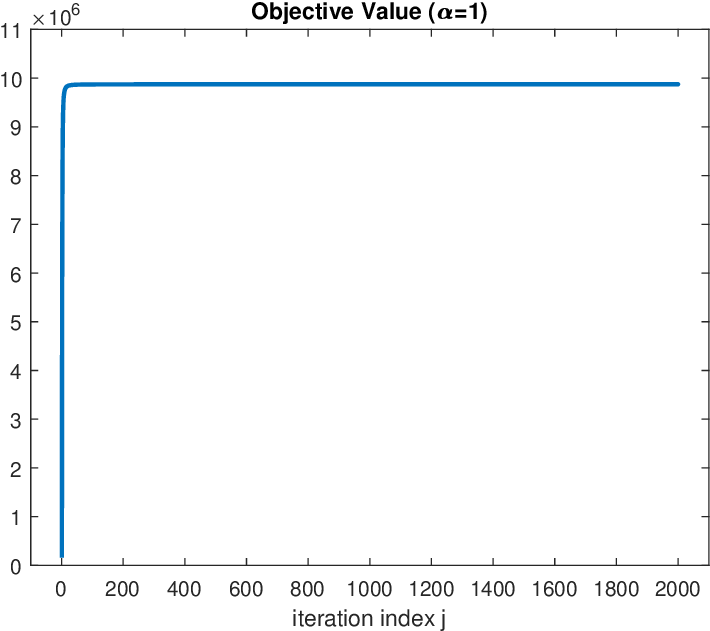}}
  & \resizebox*{0.28\textwidth}{0.15\textheight}{\includegraphics{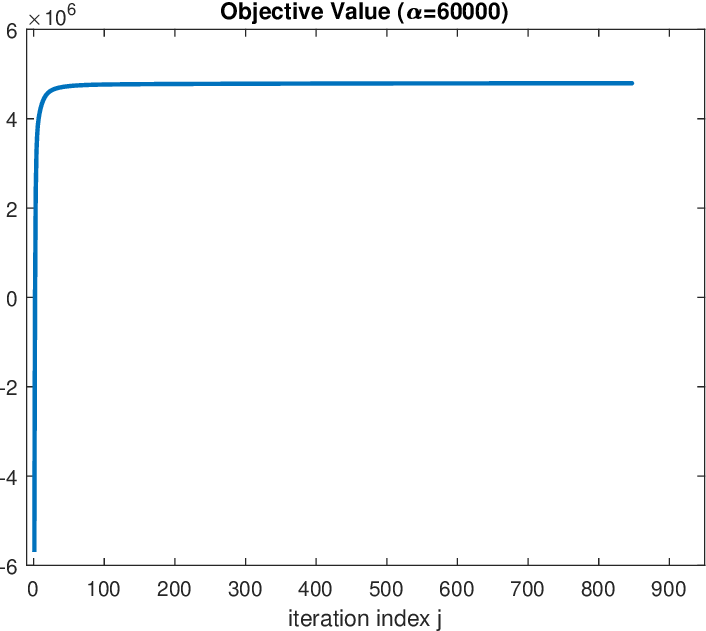}}
  & \resizebox*{0.28\textwidth}{0.15\textheight}{\includegraphics{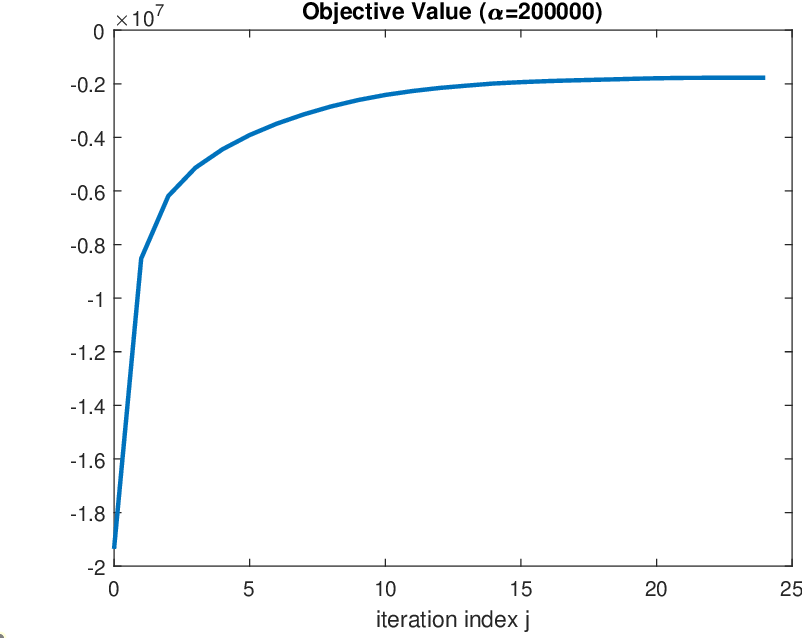}} \\
\resizebox*{0.28\textwidth}{0.15\textheight}{\includegraphics{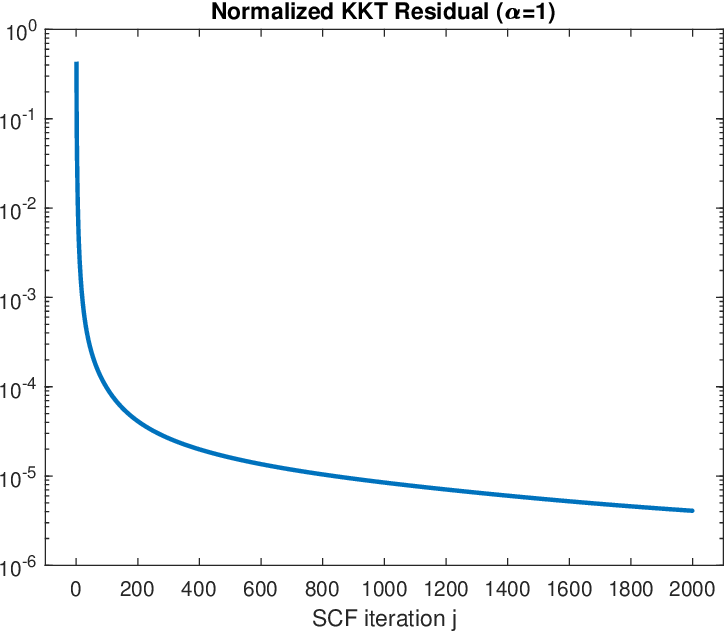}}
  & \resizebox*{0.28\textwidth}{0.15\textheight}{\includegraphics{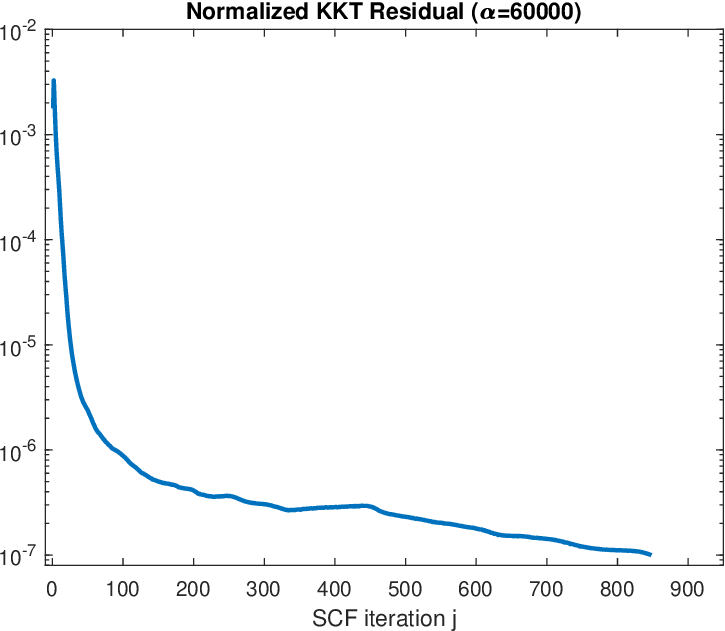}}
  & \resizebox*{0.28\textwidth}{0.15\textheight}{\includegraphics{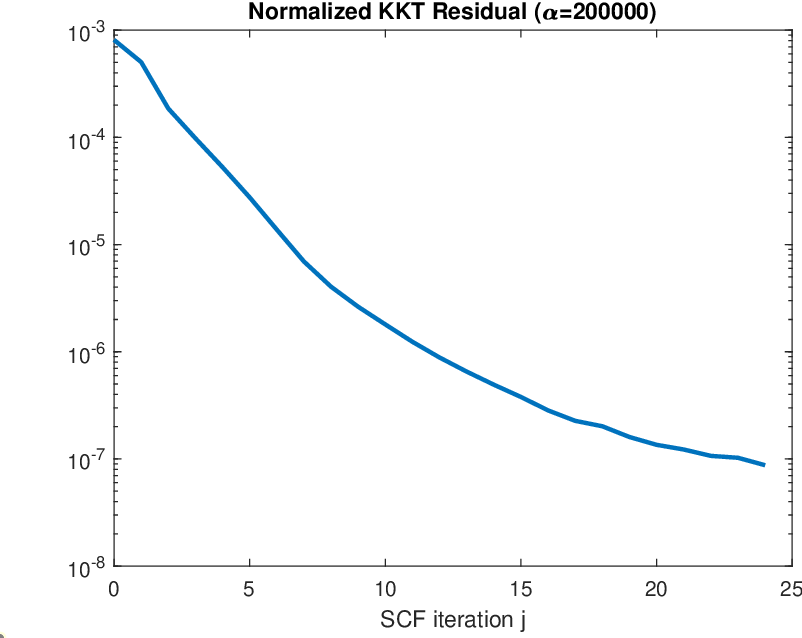}} \\
\resizebox*{0.28\textwidth}{0.15\textheight}{\includegraphics{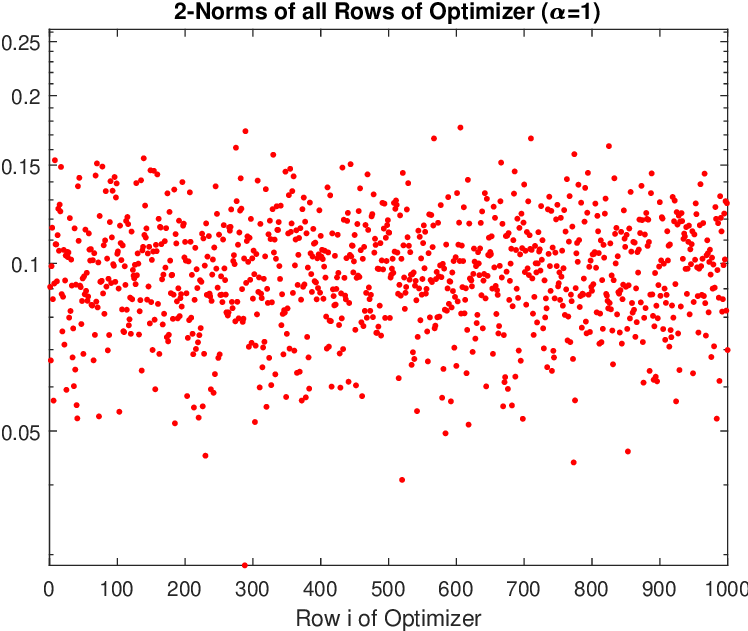}}
  & \resizebox*{0.28\textwidth}{0.15\textheight}{\includegraphics{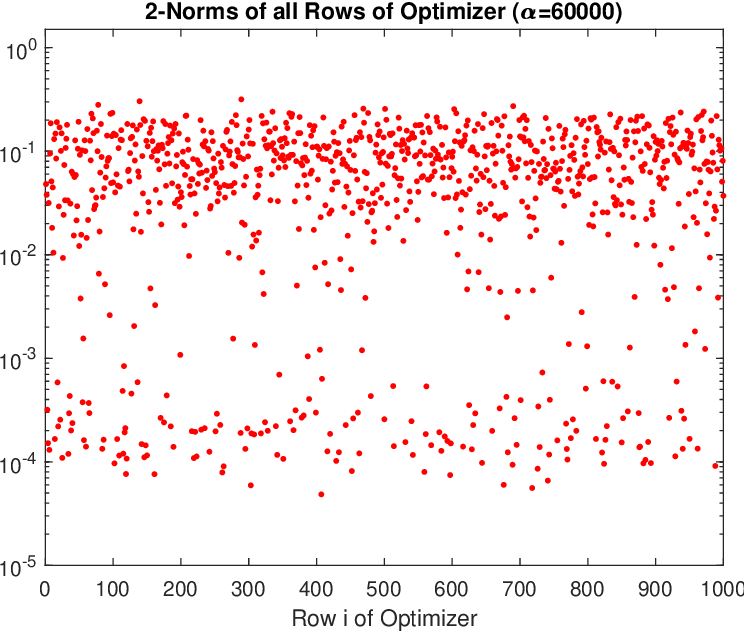}}
  & \resizebox*{0.28\textwidth}{0.15\textheight}{\includegraphics{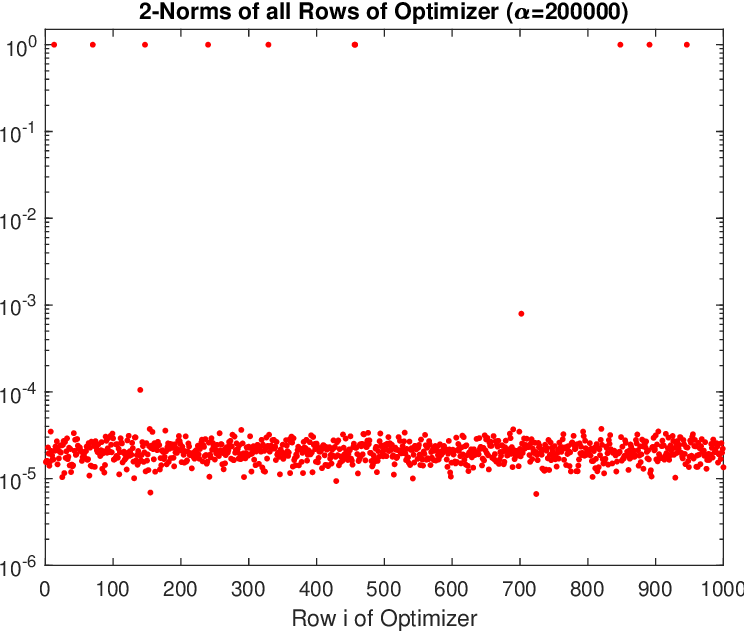}} \\
\end{tabular}\par
}
\vspace{-0.15 cm}
\caption{\small Regularized OCCA problem with $\alpha=1$ (1st column), $6\cdot 10^4$ (second column), and $2\cdot 10^5$ (third column) by
  SCF  (\Cref{alg:NEPvSCF4+L21} with LOBPCG but without any preconditioning). Associated with the last plot, the nontrivial rows of computed maximizer $P$ are its row 96, 289, 330, 521, 601, 690, 793, 849, 900, and 937.
%  It can be observed that the objective (durst row) monotonically increases
%  and the normalized KKT residual $\varepsilon_{\KKT}$ as defined in \eqref{eq:stop-1} goes towards $0$.
%  Smaller regularizing parameter $\alpha$ does not induce sparse rows. In fact,
%  row sparsity begins to show at $\alpha=10^4$ but completely emerges at $10^5$. It seems that there are just 9 dots on par with
%  $10^0$ in the last plot, but there are actually 10 dots due to that the two dots for rows 456 and 457 are visually inseparable.
  }
\label{fig:OCCA:SCF-LOBPCG}
\end{figure}
%\fi

\begin{figure}[t]
{\centering
\begin{tabular}{ccc}
\resizebox*{0.28\textwidth}{0.15\textheight}{\includegraphics{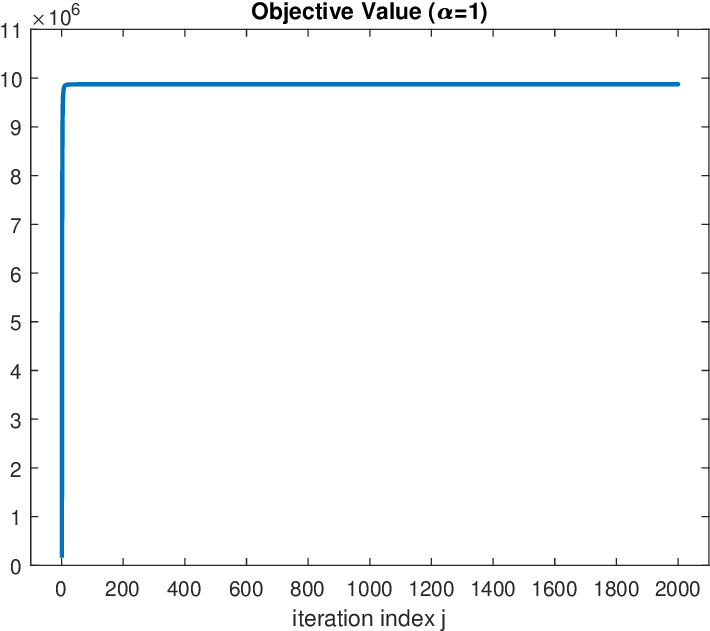}}
  & \resizebox*{0.28\textwidth}{0.15\textheight}{\includegraphics{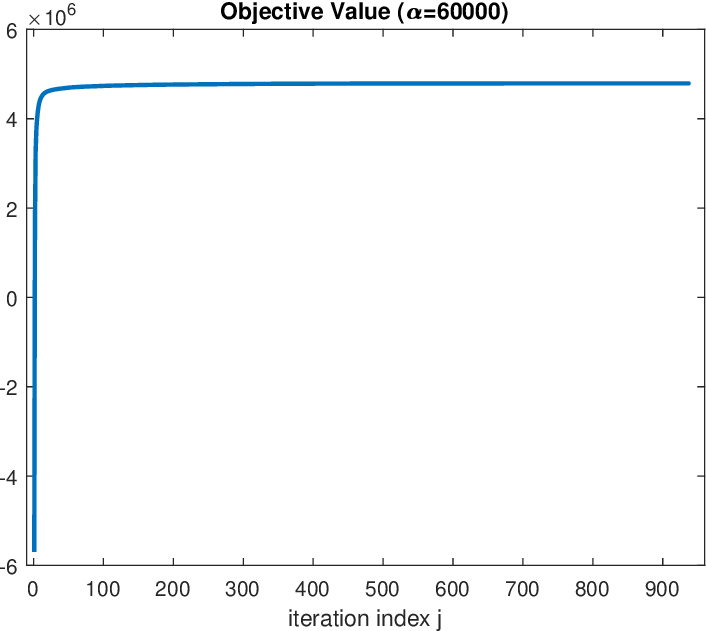}}
  & \resizebox*{0.28\textwidth}{0.15\textheight}{\includegraphics{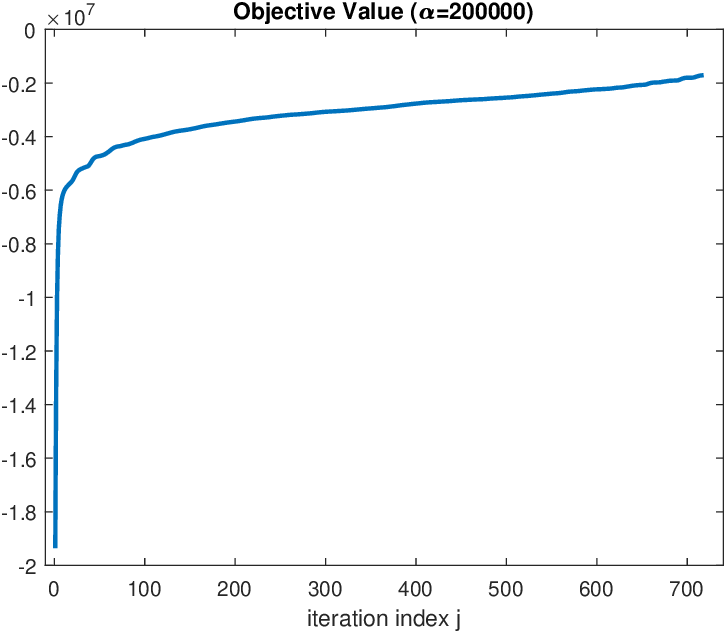}} \\
\resizebox*{0.28\textwidth}{0.15\textheight}{\includegraphics{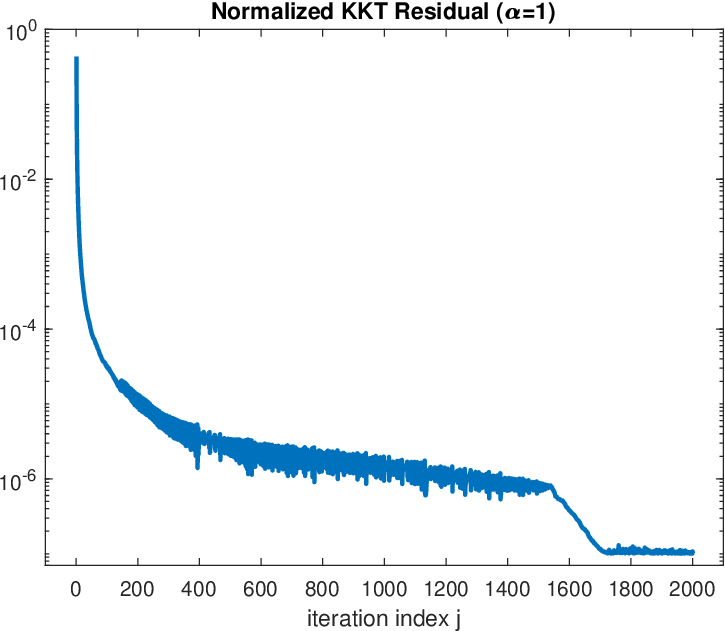}}
  & \resizebox*{0.28\textwidth}{0.15\textheight}{\includegraphics{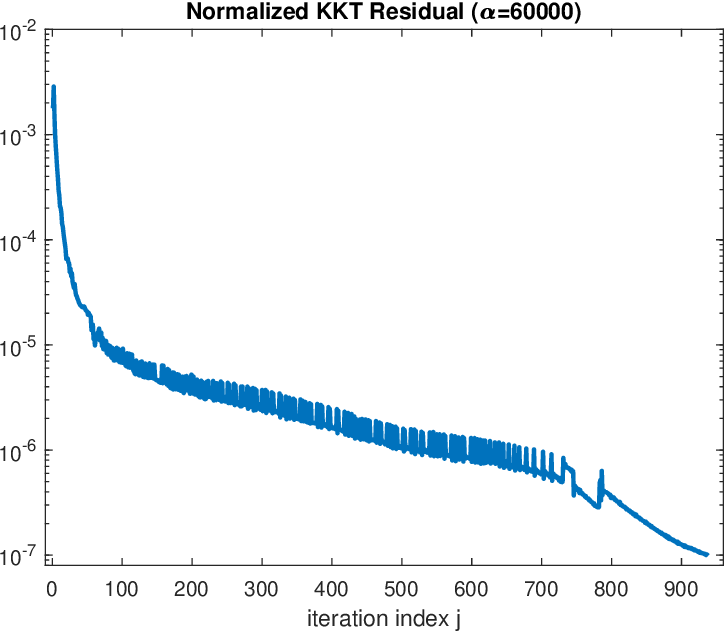}}
  & \resizebox*{0.28\textwidth}{0.15\textheight}{\includegraphics{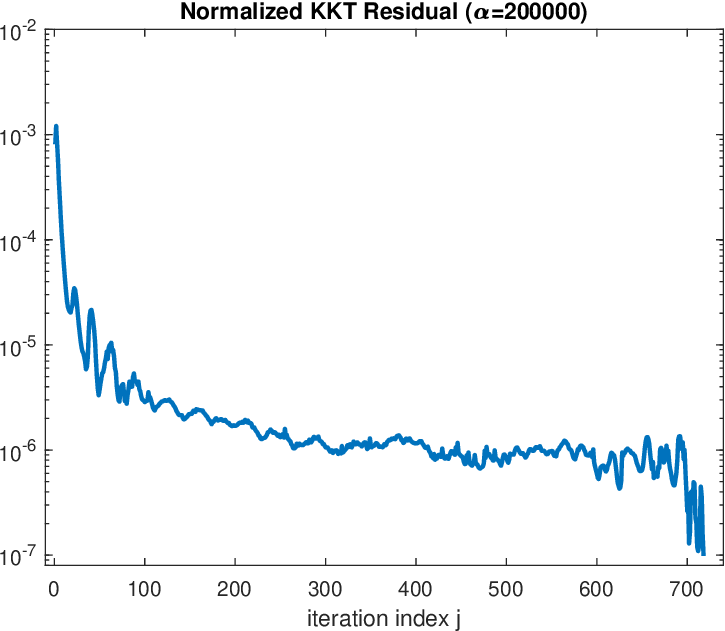}} \\
\resizebox*{0.28\textwidth}{0.15\textheight}{\includegraphics{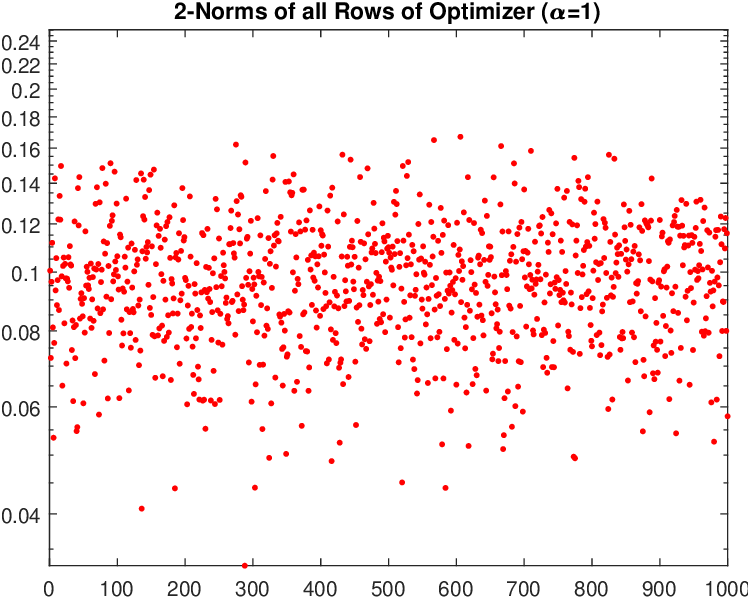}}
  & \resizebox*{0.28\textwidth}{0.15\textheight}{\includegraphics{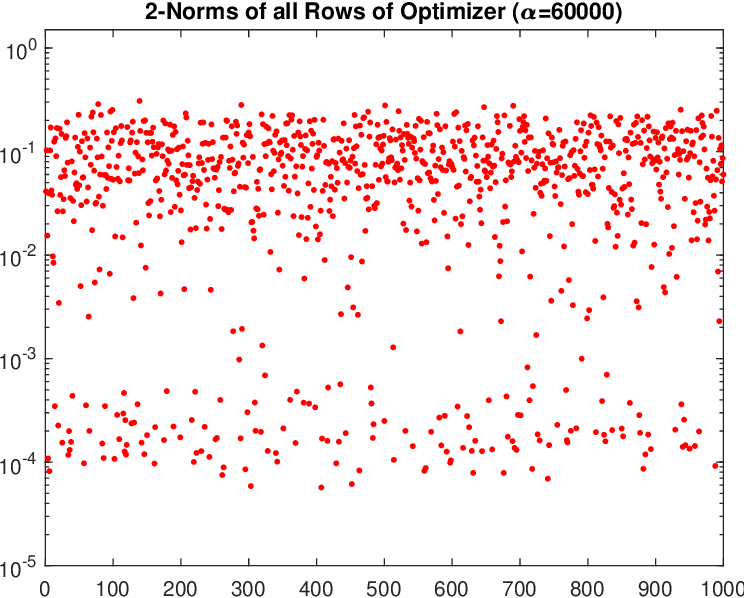}}
  & \resizebox*{0.28\textwidth}{0.15\textheight}{\includegraphics{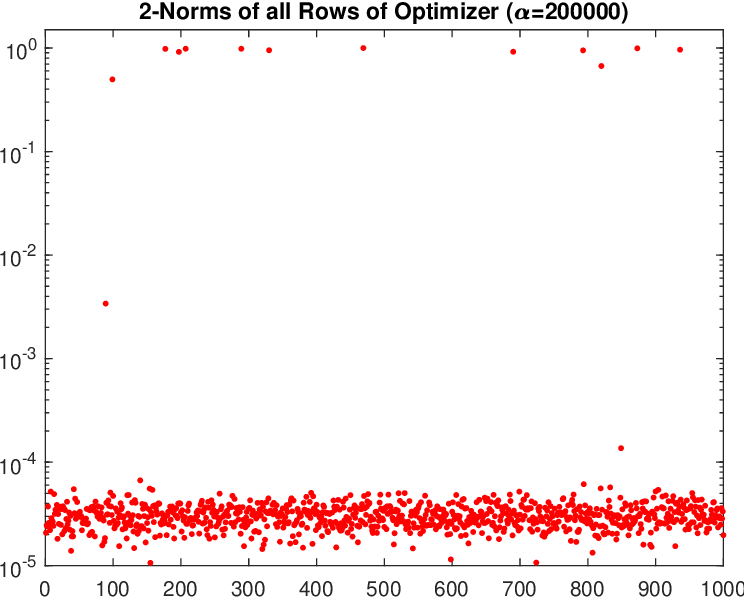}} \\
\end{tabular}\par
}
\vspace{-0.15 cm}
\caption{\small Regularized OCCA problem with $\alpha=1$ (1st column), $6\cdot 10^4$ (second column), and $2\cdot 10^5$ (third column) by
  SCF with acceleration (\Cref{alg:NEPvLOCG}). Associated with the last plot, the nontrivial rows of computed maximizer $P$ are its row 177, 197, 207, 289, 330, 469, 690, 793, 873, and 936.
%  It can be observed that the objective (durst row) monotonically increases
%  and the normalized KKT residual $\varepsilon_{\KKT}$ as defined in \eqref{eq:stop-1} goes towards $0$.
%  Smaller regularizing parameter $\alpha$ does not induce sparse rows. In fact,
%  row sparsity begins to show at $\alpha=10^4$ but completely emerges at $10^5$. It seems that there are just 9 dots on par with
%  $10^0$ in the last plot, but there are actually 10 dots due to that the two dots for rows 456 and 457 are visually inseparable.
  }
\label{fig:OCCA:SCFacc}
\end{figure}

\Cref{tbl:MAXBET21} collects performance statistics:
objective value, CPU time, KKT residual, and (outer) iteration number,
associated with the numerical experiments shown in
\Cref{fig:MAXBET:SCF,fig:MAXBET:SCF-LOBPCG,fig:MAXBET:SCFacc}. While all three methods delivered comparable
normalized KKT residual at the end, it exposes three additional phenomena that are hard to tell or cannot be told
by the plots:
\begin{itemize}
  \item At $\alpha=1$ and $10^4$, respectively, all algorithms deliver pretty much the same objective value, but not at $\alpha=10^5$ at which the regularizing term dominates  and leads to complete row sparsity but different
       implementations/methods produce different optimizers and,
      as a result, different selected features,
      confirming the ``randomness in feature selection'' with too large $\alpha$.
  \item At $\alpha=1$ and $10^4$, \Cref{alg:NEPvSCF4+L21} with LOBPCG beats \Cref{alg:NEPvSCF4+L21} with {\tt eig} in CPU time
        but not at $\alpha=10^5$ again, ay which the former takes 10 outer iterations, twice as many as the latter.
  \item At $\alpha=1$ and $10^4$,  \Cref{alg:NEPvLOCG} beats the two variants of \Cref{alg:NEPvSCF4+L21} in CPU time but not at $\alpha=10^5$ again. % at which the regularizing term dominates.
      It turns out that, at $\alpha=10^5$,
      \Cref{alg:NEPvLOCG} takes about 45 and 64 times as many as the numbers of iterations by both variants of \Cref{alg:NEPvSCF4+L21}, respectively.
      We also note that, at $\alpha=10^4$, \Cref{alg:NEPvLOCG}  takes about just 5 times as many as the numbers of iterations by the variants, but \Cref{alg:NEPvLOCG} still wins in CPU time due to the fact that
      each inner iteration of \Cref{alg:NEPvLOCG} works with a problem of size $3k=30$ instead of $n=1000$.
\end{itemize}

Next we get to the regularized OCCA problem.
\Cref{fig:OCCA:SCF,fig:OCCA:SCF-LOBPCG,fig:OCCA:SCFacc} show the numerical results for
$\alpha=1,\, 6\cdot 10^4,\, 2\cdot 10^5$ with  \Cref{fig:OCCA:SCF} for \Cref{alg:NEPvSCF4+L21} with {\tt eig},
\Cref{fig:OCCA:SCF-LOBPCG} for \Cref{alg:NEPvSCF4+L21}  with LOBPCG \cite{knya:2001} (without any preconditioning however), and finally,
\Cref{fig:OCCA:SCFacc} for \Cref{alg:NEPvLOCG}.
While there are many similar things to what we have seen for the regularized MAXBET problem,
 we notice that at $\alpha=1$ \Cref{alg:NEPvSCF4+L21} takes 2000 SCF iterations
without reducing the normalized KKT residual to $10^{-7}$, but it succeeds at other two $\alpha$.
Once again, \Cref{alg:NEPvSCF4+L21} has an easier time for sufficient larger $\alpha$ than much smaller $\alpha$ for the same reason as we explained before.
Also the sets of the row indices for the nontrivial rows associated with the last plots in
\Cref{fig:OCCA:SCF,fig:OCCA:SCF-LOBPCG,fig:OCCA:SCFacc} are different.

Finally, \Cref{tbl:OCCA21} collects performance statistics associated with the numerical experiments shown in
\Cref{fig:OCCA:SCF,fig:OCCA:SCF-LOBPCG,fig:OCCA:SCFacc}. Our earlier observations from \Cref{tbl:MAXBET21}
remain true, unsurprisingly.

\setlength{\tabcolsep}{4pt}
\renewcommand{\arraystretch}{1.25}
\begin{table}[b]
\caption{Performance statistics for regularized OCCA}\label{tbl:OCCA21}
\centerline{\scriptsize
\begin{tabular}{|c|c|c|c|c|c|c|c|c|c|}
  \hline
  & \multicolumn{3}{c|}{\Cref{alg:NEPvSCF4+L21} (with {\tt eig})}
  & \multicolumn{3}{c|}{\Cref{alg:NEPvSCF4+L21} (with LOBPCG)}
  & \multicolumn{3}{c|}{\Cref{alg:NEPvLOCG}} \\ \hline
$\alpha$ & $1$ & $6\cdot 10^4$ & $2\cdot 10^5$
         & $1$ & $6\cdot 10^4$ & $2\cdot 10^5$
         & $1$ & $6\cdot 10^4$ & $2\cdot 10^5$ \\ \hline
Obj. & $9.9\cdot 10^6$  & $4.8\cdot 10^6$ & $-1.9\cdot 10^6$
     & $9.9\cdot 10^6$  & $4.8\cdot 10^6$ & $-1.8\cdot 10^6$
     & $9.9\cdot 10^6$  & $4.8\cdot 10^6$ & $-1.7\cdot 10^6$ \\ \hline
CPU & $7.8\cdot 10^1$ & $3.5\cdot 10^1$ & $3.9\cdot 10^{-1}$
    & $1.2\cdot 10^1$ & $1.5\cdot 10^1$ & $5.6\cdot 10^{-1}$
    & $3.4\cdot 10^1$ & $1.5$           & $1.1$\\ \hline
$\varepsilon_{\KKT}$
    & $4.1\cdot 10^{-6}$ & $1.0\cdot 10^{-7}$ & $8.3\cdot 10^{-8}$
    & $4.1\cdot 10^{-6}$ & $1.0\cdot 10^{-7}$ & $8.1\cdot 10^{-8}$
    & $1.0\cdot 10^{-7}$ & $1.0\cdot 10^{-7}$ & $7.6\cdot 10^{-8}$ \\ \hline
it'ns & 2000 & 926 & 9 & 2000 & 847 & 24 & 2000 & 937 & 717\\
\hline
\end{tabular}
}
\end{table}

\clearpage
\section{Conclusion}\label{sec:concl}
%Linear dimensionality reduction is about computing a proper projection matrix $P$ that optimizes
%certain suitably learning objective constructed from given sample data points.
%When $P$ is restricted to the Stiefel manifold, such as PCA and
%orthogonal LDA, an optimization problem on the Stiefel manifold arises. Further, when the objective is properly regularized
%by the matrix $(2,1)$-norm, optimizer $P$ may have numerous negligible rows whose corresponding features in
%the data points may be regarded irrelevant for the purpose of feature selection. However,
%optimization on the Stiefel manifold is usually difficult to deal with numerically, and that combined with the matrix $(2,1)$-norm which is nonsmooth results in even more challenging optimization problems.

Optimization on Stiefel manifold with the $(2,1)$-norm regularization can lead to row-sparse projections
in linear dimensionality reduction for the purpose of feature selections. In this paper,
we have developed a unifying NEPv framework for effectively handling such optimization problems, both theoretically and numerically.
Our success in building such a unifying framework critically relies  on an observation that the matrix $(2,1)$-norm $\|P\|_{2,1}$
can be reformulated as
$$
\|P\|_{2,1}=\sum_{i=1}^n\sqrt{\tr(P^{\T}\be_i\be_i^{\T}P)},
$$
a concave composition of matrix trace functions, which makes it possible to utilize the recent NEPv development in
\cite{li:2024}. In that regards, we point out that what we have done so far in this paper can be extended straightforwardly to
\begin{equation}\label{eq:OptSTM+L21-p}
\max_{P\in\bbO^{n\times k}}\Big\{f(P):=g(P)-\sum_{i=1}^M\alpha_i\big[\tr(P^{\T}A_iP)\big]^p\Big\},
\end{equation}
where $g(P)$ is the same as before, $0<p\le 1$, $\alpha_i>0$ and $\bbR^{n\times n}\ni A_i\succeq 0$ for $1\le i\le M$.
Evidently, $\OptSM_{2,1}$~\eqref{eq:OptSTM+L21} is a special case of \eqref{eq:OptSTM+L21-p}, which can be viewed as
objective $g(P)$ regularized by $\big[\tr(P^{\T}A_iP)\big]^p$ for $1\le i\le M$ with $M$
regularizing parameters $\alpha_i$.

It is explained that how the framework can be used to optimize today's commonly used objectives, such as
PCA, LDA, OCCA, MAXBET and the unbalanced Procrustes problem, etc, combined with the $(2,1)$-norm regularization
to yield row-sparse projections.
Preliminary numerical experiments are reported on regularized OCCA and MAXBET for illustration purposes.
These experiments reveal that too large regularizing parameters will nicely lead to complete row-sparse
optimizers but may  overlook any effect from the original objective $g(P)$ and, as a result,
any such complete row-sparse optimizer may not be practically useful. It is suggested that
a proper regularizing parameter should be the ones that result in emergence of row-sparsity, but exactly how to
select such parameters  warrants further investigations.

Throughout the article, we limit ourselves to  the field of real numbers because most optimization on the Stiefel manifold from machine learning dominantly falls into. But our developments can be extended to
the field of complex numbers, with minor modifications: replacing
        transpose $(\cdot)^{\T}$ with conjugate transpose $(\cdot)^{\HH}$ and some functions with
their real parts, such as $\tr((P^{\T}D)^m)$ with $\Re(\tr((P^{\HH}D)^m))$
        (where $\Re(\cdot)$ takes the real part of a complex number).

\clearpage
{\small
\def\noopsort#1{}\def\l{\char32l}\def\v#1{{\accent20 #1}}
  \let\^^_=\v\def\hbk{hardback}\def\pbk{paperback}

}

\end{document}